\documentclass[11pt]{article}
\usepackage[a4paper,margin=1in]{geometry}
\usepackage{amsmath,amssymb,amsthm,mathrsfs,mathtools}
\usepackage{microtype}
\usepackage{enumitem}
\usepackage{hyperref}
\hypersetup{
  colorlinks=true,
  linkcolor=blue,
  citecolor=blue,
  urlcolor=blue,
  pdftitle={Optimal Linear Dependence on Boundary Type for Local Gromov
  Hyperbolicity of the Kobayashi Metric},
  pdfkeywords={Kobayashi distance, local Gromov hyperbolicity, universal hyperbolicity profile, convex finite type, D'Angelo type}
}

\newtheorem{theorem}{Theorem}[section]
\newtheorem{lemma}[theorem]{Lemma}
\newtheorem{example}[theorem]{Example}
\newtheorem{proposition}[theorem]{Proposition}
\newtheorem{corollary}[theorem]{Corollary}
\theoremstyle{definition}
\newtheorem{definition}[theorem]{Definition}
\newtheorem{remark}[theorem]{Remark}

\newcommand{\C}{\mathbb{C}}
\newcommand{\R}{\mathbb{R}}
\newcommand{\D}{\mathbb{D}}
\newcommand{\HH}{\mathbb{H}}

\newcommand{\Hol}{\operatorname{Hol}}
\newcommand{\dist}{\operatorname{dist}}
\newcommand{\eps}{\varepsilon}
\newcommand{\dd}{\mathop{}\!\mathrm{d}}

\renewcommand{\Re}{\operatorname{Re}}
\renewcommand{\Im}{\operatorname{Im}}
\DeclareMathOperator{\ord}{ord}
\newcommand{\Aff}{\operatorname{Aff}}
\newcommand{\Id}{\operatorname{id}}
\begin{document}

\title{Optimal Linear Dependence on Boundary Type for Local Gromov
	Hyperbolicity of the Kobayashi Metric}

\author{%
	\parbox{\dimexpr\textwidth-2\tabcolsep\relax}{\centering
		Cheng Lou \quad Jianyong Qiao \quad Hongyu Wang \quad Yumin Zhong
		\\[1.2ex]
		{\small
			School of Science, Beijing University of Posts and Telecommunications,\\
			Beijing 100876, China\\
			Key Laboratory of Mathematics and Information Networks\\
			(Beijing University of Posts and Telecommunications),\\
			Ministry of Education, China\endgraf}
	}%
}

\date{}

\maketitle

\begin{abstract}

The Gromov hyperbolicity constant of a metric space \((X,d)\) is the infimum of all \(\delta\ge0\) such that \((X,d)\) is \(\delta\)-hyperbolic. For a Kobayashi hyperbolic domain \(\Omega\subset\C^n\) and a
boundary point \(p\in\partial\Omega\), let
\(\delta_{\mathrm{loc}}(\Omega,p)\) denote the local Gromov
hyperbolicity constant obtained by restricting the points to arbitrarily
small Euclidean neighborhoods of \(p\), while distances are still
measured by the ambient Kobayashi distance \(K_\Omega\).  For
an even integer \(M\ge2\), let \(\mathfrak H(M)\) be the supremum of these
constants over all complex dimensions \(n\ge2\) and all domains whose
boundary is smooth and convex near the distinguished point and has D'Angelo type
at most \(M\) at that point.  We prove
\[
\frac{\log 2}{2} M
\le
\mathfrak H(M)
<
36M.
\]
 Thus the optimal universal dependence of the local
Gromov hyperbolicity constant on boundary type is linear.  We also show that
no analogous bound holds for the Gromov hyperbolicity constant,
even among strongly convex domains.

\end{abstract}


\setcounter{tocdepth}{2}
\tableofcontents
\section{Introduction}\label{sec:introduction}

The Kobayashi distance is one of the basic intrinsic metrics in several
complex variables.  It is invariant under biholomorphic maps and
non-increasing under holomorphic maps \cite{Kobayashi}.  On the unit ball,
it agrees, up to normalization, with the complex hyperbolic distance.
On general domains, however, the associated infinitesimal metric is
typically Finsler rather than Riemannian, so classical Riemannian
curvature is not directly available.

Gromov hyperbolicity provides a natural framework for studying the
coarse negative-curvature behavior of the Kobayashi distance
\cite{Gromov,BridsonHaefliger}. Although Gromov hyperbolicity is a
large-scale metric property, large-scale Kobayashi geometry can already
be visible in arbitrarily small Euclidean neighborhoods of the boundary.
Indeed, points approaching the same boundary point can be arbitrarily
close in Euclidean distance while being arbitrarily far apart with
respect to the Kobayashi distance. Consequently, the question considered
here is local in the Euclidean topology but asymptotic in the Kobayashi
geometry.

The first systematic connection between boundary geometry and Gromov
hyperbolicity for domains equipped with the Kobayashi distance was
established by Balogh and Bonk. For bounded strongly pseudoconvex
domains, they compared the Kobayashi distance, up to a bounded additive
error, with a model metric involving the boundary distance and the
Carnot--Carath\'eodory metric associated with the complex tangent
bundle of the boundary. This comparison allowed them to prove that
every bounded strongly pseudoconvex domain is Gromov hyperbolic with
respect to its Kobayashi distance \cite{BaloghBonk}.  Weakly pseudoconvex domains exhibit a wider range
of behavior.  The polydisc is not Gromov hyperbolic, and more generally,
Gaussier and Seshadri proved that a smooth bounded convex domain whose
boundary contains a nonconstant analytic disk cannot be Gromov
hyperbolic for the Kobayashi distance
\cite{GaussierSeshadri}.

D'Angelo type measures the order of contact of holomorphic curves with
the boundary \cite{DAngelo}.  Strongly pseudoconvex points have type
two, while a nonconstant analytic disk contained in the boundary
forces infinite type.  Zimmer proved the fundamental qualitative
characterization that a bounded convex domain with smooth boundary is
Gromov hyperbolic for the Kobayashi distance if and only if its
boundary has finite D'Angelo type \cite{ZimmerFiniteType}.  Beyond
convexity, the sufficient direction has also been established in
several important finite-type settings.  Fiacchi treated smooth bounded
pseudoconvex finite-type domains in \(\C^2\), while Zhang extended this
to smooth bounded pseudoconvex finite-type domains whose Levi form has
corank at most one \cite{Fiacchi,ZhangLeviCorankOne}.

Scaling methods play an important role in these results.  Such methods
were developed by Pinchuk and Frankel and later extended in work of
Bedford--Pinchuk, Kim--Krantz, and others to finite-type boundary
problems and invariant metrics
\cite{PinchukScaling,FrankelAffine,BedfordPinchukScaling,KimKrantzScaling}.
The basic idea is to magnify the geometry near a boundary point by
suitable affine or holomorphic transformations and study the resulting
limit domains.  In particular, scaling is a central ingredient in
Zimmer's characterization of finite-type convex domains.  These results
answer the qualitative question of whether a finite hyperbolicity
constant exists, but they do not determine its quantitative dependence
on the boundary type.

There are also precise metric descriptions of finite-type domains.
Catlin developed distinguished coordinates adapted to the non-isotropic
boundary geometry, while McNeal introduced adapted polydiscs reflecting
the corresponding boundary scales
\cite{CatlinMetrics,McNealConvexFiniteType,McNealBergman}.
 For each fixed smoothly bounded convex finite-type domain, H.~Wang derived a pairwise comparison formula for the Kobayashi distance, up to a bounded additive error, in which the model expression involves McNeal's boundary quasi-distance \cite{WangJLMS}.
Li, Pu, and L.~Wang proved related estimates for smooth bounded
pseudoconvex finite-type domains in \(\C^2\) \cite{LiPuWang}.
Recently, T.~Wang and Zimmer developed a different approach based on linear
isoperimetric inequalities and obtained hyperbolicity results for several
intrinsic metrics, including the Kobayashi metric on convex finite-type
domains \cite{WangZimmerIso}.

These results establish qualitative Gromov hyperbolicity or provide
quantitative metric estimates under additional geometric hypotheses.
The associated constants, however, may depend on the particular domain,
the complex dimension, or finer boundary data. It is therefore natural
to ask whether D'Angelo type alone can provide uniform quantitative
control.

No such control is possible for the global Gromov hyperbolicity constant
of \((\Omega,K_\Omega)\), namely, the infimum of all
\(\delta\geq 0\) such that \((\Omega,K_\Omega)\) is
\(\delta\)-hyperbolic. The
D'Angelo type is a local boundary invariant, whereas the global
hyperbolicity constant also reflects the geometry of the entire domain.
In Example~\ref{ex:no-global-bound}, we construct, in every fixed complex
dimension, a sequence of smooth bounded strongly convex domains, all of
D'Angelo type two, whose global Gromov hyperbolicity constants tend to
infinity. These domains increasingly approximate the geometry of the
polydisc and develop large product-like regions in their Kobayashi
geometry, even though each domain in the sequence is strongly convex.

This global obstruction motivates the following local formulation. Let
\(\Omega\subset\C^n\) be a Kobayashi hyperbolic domain, and let
\(p\in\partial\Omega\). We define
\[
\delta_{\mathrm{loc}}(\Omega,p)
=
\inf_{V\ni p}
\delta\left(
\Omega\cap V,\,
\left.K_\Omega\right|_{\Omega\cap V}
\right),
\]
where \(V\) ranges over all Euclidean neighborhoods of \(p\),
\(\delta(X,d)\) denotes the Gromov hyperbolicity constant of the metric
space \((X,d)\), and \(\left.K_\Omega\right|_{\Omega\cap V}\) denotes
the restriction of the ambient Kobayashi distance \(K_\Omega\) to the
subset \(\Omega\cap V\).

Thus only the points under consideration are localized near \(p\); their
mutual distances are still measured by the ambient Kobayashi distance.
This ambient localization differs from the intrinsic localization
studied by Bracci, Gaussier, Nikolov, and Thomas
\cite{BracciGaussierNikolovThomas}, in which the truncated domain
\(\Omega\cap V\) is equipped with its own Kobayashi distance
\(K_{\Omega\cap V}\). Accordingly, \(\delta_{\mathrm{loc}}(\Omega,p)\)
records the ambient Kobayashi geometry seen arbitrarily close to the
boundary point \(p\), without introducing the additional effects
associated with the artificial boundary \(\partial V\cap\Omega\).

The quantitative problem is therefore to determine whether
\(\delta_{\mathrm{loc}}(\Omega,p)\) admits an upper bound depending only
on an upper bound for the D'Angelo type near \(p\), uniformly over all
domains and complex dimensions, and to determine the optimal order of
growth of such a bound. To encode this problem, for an even integer \(M\ge2\), define
\(\mathfrak H(M)=\sup\delta_{\mathrm{loc}}(\Omega,p)\), where the
supremum is taken over all complex dimensions \(n\ge2\), all Kobayashi
hyperbolic domains \(\Omega\subset\C^n\), and all \(p\in\partial\Omega\)
such that \(\Omega\cap U\) is convex for some neighborhood \(U\) of \(p\),
\(\partial\Omega\) is smooth near \(p\), and its D'Angelo type at \(p\)
is at most \(M\). Since the finite D'Angelo type at a smooth convex boundary point is an
even integer, it suffices to consider even \(M\). A
priori, this supremum could be infinite even though every individual
pointed domain in the class has a finite local hyperbolicity constant.

Our main result determines the growth of this profile.

\begin{theorem}
	\label{thm:universal-growth}
	For every even integer \(M\ge2\), the universal local hyperbolicity profile
	satisfies
	\begin{equation*}
		\frac{\log 2}{2}M
		\le
		\mathfrak H(M)
		<
		36M.
	\end{equation*}
	Moreover, the lower bound remains valid if the supremum is restricted
	to any fixed complex dimension \(n\ge2\).
\end{theorem}

Thus \(\mathfrak H(M)\asymp M\). In particular, every local
hyperbolicity constant in the class admits an upper bound depending only on
the boundary type, uniformly over all domains and dimensions, and the linear
order of growth of the universal profile is optimal.   The size of the Euclidean neighborhood on which the local
estimate is obtained may still depend on the pointed domain.

A main ingredient in the upper bound is a type-uniform local comparison
for the Kobayashi distance.  In
Section~\ref{sec:kobayashi-comparison}, we construct from the
non-isotropic boundary geometry a finite-chain quasi-distance \(r_m\).
For each fixed smoothly bounded convex finite-type domain, this
quasi-distance is locally Lipschitz equivalent to McNeal's
quasi-distance.  We then set
\[
g_m(x,y)
=
\log
\frac{
	r_m(x,y)+\max\{\delta(x),\delta(y)\}
}{
	\sqrt{\delta(x)\delta(y)}
}.
\]
Although \(g_m\) reflects the geometry of the particular domain, the
additive errors in the following comparison depend only on the type
bound.

\begin{theorem}
	\label{thm:two-sided-kobayashi}
	Let \(\Omega\subset\C^n\) be a convex domain, possibly unbounded, and
	let \(p\in\partial\Omega\).  Suppose that \(\partial\Omega\) is smooth
	near \(p\) and has D'Angelo type at most an integer \(m\ge2\)
	at \(p\).  Then there exists a neighborhood
	\(V\) of \(p\) such that
	\begin{equation*}
		g_m(x,y)-31m
		\le
		K_\Omega(x,y)
		\le
		g_m(x,y)+\frac m2+2+\log3,
	\end{equation*}
	for all \(x,y\in\Omega\cap V\).
	The constants are independent of \(\Omega\), \(p\), and the dimension
	\(n\).
\end{theorem}

For comparison, H.~Wang's estimate \cite{WangJLMS} gives, for each fixed
smoothly bounded convex finite-type domain, a pairwise formula involving
McNeal's quasi-distance with constants that may depend on the domain and the dimension.
The following local hyperbolicity estimate follows directly from Theorem~\ref{thm:two-sided-kobayashi}.

\begin{theorem}
	\label{thm:main}
	Let \(\Omega\subset\C^n\) be a convex domain(possibly unbounded), and
	let \(p\in\partial\Omega\).  Suppose that \(\partial\Omega\) is smooth
	near \(p\) and has D'Angelo type at most an integer \(m\ge2\)
	at \(p\).  Then
	\[
	\delta_{\mathrm{loc}}(\Omega,p)<36m.
	\]
\end{theorem}

The linear lower bound is realized by the following homogeneous
model domains.

\begin{theorem}\label{thm:linear-lower-intro}
For each integer $m\geq1$, let
\[
\mathcal T_m
=
\left\{(Q,Z)\in\mathbb C^2:
\operatorname{Re}Q>(\operatorname{Re}Z)^{2m}\right\}.
\]
Then $\mathcal T_m$ is a convex Kobayashi hyperbolic domain with smooth boundary of D'Angelo type at most $2m$ . Moreover,
\[
\delta_{\mathrm{loc}}(\mathcal T_m,0)
\geq m\log 2.
\]
\end{theorem}

The proof exploits the symmetry and homogeneity of $\mathcal T_m$. On a
suitable real slice, the Kobayashi distance factors as $m$ times an auxiliary
length metric, and appropriately chosen points yield the lower bound
$m\log 2$.  Moreover,
holomorphic retractions transfer the same construction to every fixed
complex dimension.

\paragraph{Organization of the paper.}
Section~\ref{sec:preliminaries} recalls the basic notions and estimates,
including the global and local hyperbolicity constants, and explains
why boundary type alone cannot control the global constant.
Section~\ref{sec:balanced-geometry} develops the non-isotropic boundary
geometry for finite-type convex domains.
Section~\ref{sec:kobayashi-comparison} establishes the local distance
comparison and deduces the local hyperbolicity estimate.
Section~\ref{sec:lower-bound} proves the linear lower bound and the
optimality of the growth rate.
Finally, Section~\ref{sec:applications} discusses applications to
global Gromov hyperbolicity and asymptotic upper curvature.

\section{Preliminaries}\label{sec:preliminaries}

\subsection{Notation}

\begin{enumerate}[label=(\arabic*),leftmargin=2.4em]
\item Throughout the paper, \(\dist\) denotes Euclidean distance in the ambient space. For \(z\in\C^n\), \(|z|\) denotes the Euclidean norm, and \(B(z,r)\) is the Euclidean ball of radius \(r\) centered at \(z\).

\item If \(\Omega\subsetneq\C^n\) is a domain and \(z\in\Omega\), set
\(\delta_\Omega(z)=\dist(z,\partial\Omega)\).
When \(\Omega\) is fixed, we write simply \(\delta(z)\). For
\(z\in\Omega\) and a nonzero vector \(v\in\C^n\), define
\(\delta_\Omega(z;v)=
\dist\bigl(z,(z+\C v)\cap\partial\Omega\bigr)\), where
\(z+\C v=\{z+\lambda v:\lambda\in\C\}\) is the complex affine line
through \(z\) in the direction \(v\).
This is the radius of the largest Euclidean disc centered at \(z\) in the
slice \(\Omega\cap(z+\C v)\).  It depends only on the complex line spanned by
\(v\), so \(\delta_\Omega(z;av)=\delta_\Omega(z;v)\) for \(a\ne0\).  If
\(z+\C v\subset\Omega\), we set \(\delta_\Omega(z;v)=+\infty\).

\item If \(U,V\) are open subsets of a Euclidean space, then
\(V\Subset U\) means that \(\overline V\) is a compact subset of \(U\).

\item For positive quantities \(A,B\), the notation \(A\asymp B\) means
that \(C^{-1}B\le A\le CB\) for a constant \(C\ge1\).

\end{enumerate}
\subsection{The signed distance function}
\label{subsec:signed-distance}

Let \(\Omega\subsetneq\C^n\) be a domain.  Its signed Euclidean
distance function is defined by
\[
 \rho_\Omega(z)=
 \begin{cases}
  -\delta_\Omega(z),& z\in\Omega,\\
  \dist(z,\partial\Omega),& z\in\C^n\setminus\Omega.
 \end{cases}
\]
When the domain is fixed, we simply write \(\rho=\rho_\Omega\).
Throughout this subsection, we identify \(\C^n\) with
\(\R^{2n}\); thus convexity, gradients, and unit normals are
understood with respect to real Euclidean structure.

The signed distance function will be used in two complementary ways.
Its convexity is global and follows solely from the convexity of the
domain, whereas its regularity is local and follows from the smoothness
of the boundary near the distinguished point.

The global convexity of \(\rho\) requires no regularity assumption on
\(\partial\Omega\) and is due to Hiriart--Urruty.

\begin{lemma}[{\cite[Proposition~4]{HiriartUrruty}}]
	\label{lem:signed-distance-global-convexity}
	Let \(\Omega\subsetneq\C^n\) be a convex domain. Then its signed
	distance function \(\rho\) is convex on \(\C^n\).
\end{lemma}

Convexity does not, however, imply smoothness. Even when
\(\partial\Omega\) is smooth, the signed distance function need not be
smooth globally; for instance, it fails to be differentiable at points
having more than one nearest boundary point. What is needed below is only local regularity near a fixed
smooth boundary point. The following results, due respectively to
Balogh--Bonk and Krantz--Parks, provide the required local smoothness
of the signed distance function and the nearest-point projection.

\begin{lemma}[{\cite[Lemma~2.1]{BaloghBonk}}]
\label{lem:signed-distance-tubular}
Let \(p\in\partial\Omega\), and suppose that \(\partial\Omega\) is of class \(C^2\) near \(p\). Then there exists a neighborhood \(U\) of \(p\) in \(\C^n\) such that every \(z\in U\) has a unique nearest point on \(\partial\Omega\). Hence the nearest-point projection \(\pi:U\to\partial\Omega\), \(z\mapsto\pi(z)\), is well defined. For \(\xi\in\partial\Omega\cap U\), let \(\nu(\xi)\) denote the outer unit normal vector to \(\partial\Omega\) at \(\xi\). The identities
\begin{equation*}
z=\pi(z)+\rho(z)\nu(\pi(z)),\qquad \nabla\rho(z)=\nu(\pi(z))
\end{equation*}
hold for every \(z\in U\).
In particular, \(|\nabla\rho(z)|=1\) for every \(z\in U\), the fibers of \(\pi\) are normal segments, and \(\pi\) is of class \(C^1\).
\end{lemma}
\begin{lemma}[
{\cite[Theorem~3, p.~119]{KrantzParks}}]
\label{lem:signed-distance-smooth}
Let \(p\in\partial\Omega\), and suppose that \(\partial\Omega\) is of
class \(C^k\), \(k\ge2\), near \(p\). Then there exists a neighborhood \(U\) of \(p\) in \(\C^n\) such that
\(\rho|_U\in C^k(U)\).
\end{lemma}
\begin{remark}
The results cited above are stated under global hypotheses. Balogh and
Bonk assume that \(\Omega\) is bounded and that \(\partial\Omega\) is
globally of class \(C^2\). The boundedness of \(\Omega\) makes
\(\partial\Omega\) compact, while compactness together with the global
\(C^2\)-regularity permits the choice of a tubular radius that is uniform
along the entire boundary. Likewise, Krantz--Parks assume that the
\(C^k\)-hypersurface is compact in order to obtain a uniform neighborhood
on which the signed distance is of class \(C^k\).

Here, however, we need these conclusions only near a fixed point
\(p\in\partial\Omega\).  One may choose a relatively compact smooth
boundary patch containing \(p\) and then shrink the ambient
neighborhood of \(p\) so that every nearest boundary point lies in this
patch.  The proofs in the cited references then apply in essentially
the same way, with the uniform estimates taken only over the chosen
local patch rather than over the entire boundary.  Consequently,
neither boundedness of \(\Omega\) nor regularity of \(\partial\Omega\)
away from \(p\) is required for the local statements used above.
\end{remark}

\subsection{The Kobayashi metric}

The Kobayashi pseudodistance on \(\Omega\), introduced by Kobayashi \cite{Kobayashi}, will be denoted by \(K_\Omega\). Its associated infinitesimal pseudometric, the Kobayashi--Royden pseudometric, will be denoted by \(k_\Omega\). For \(z\in\Omega\) and \(v\in\C^n\), define
\[
k_\Omega(z;v)=\inf\left\{|\lambda|:\lambda\in\C,\ f\in\Hol(\D,\Omega),\ f(0)=z,\ \lambda f'(0)=v\right\},
\]
where \(\D=\{\zeta\in\C:|\zeta|<1\}\), and \(\Hol(\D,\Omega)\) denotes the set of holomorphic maps from \(\D\) into \(\Omega\).

Following an observation of Royden, Venturini proved that the Kobayashi pseudodistance \(K_\Omega\) coincides with the length pseudodistance induced by \(k_\Omega\); see \cite{Royden,Venturini}. More precisely, if \(\gamma:[a,b]\to\Omega\) is absolutely continuous, its Kobayashi length is
\[
L_k(\gamma)=\int_a^b k_\Omega\bigl(\gamma(t);\gamma'(t)\bigr)\,\dd t.
\]
For \(x,y\in\Omega\), one has
\[
K_\Omega(x,y)=\inf_\gamma L_k(\gamma),
\]
where the infimum is taken over all absolutely continuous curves \(\gamma:[a,b]\to\Omega\) satisfying \(\gamma(a)=x\) and \(\gamma(b)=y\).
In general, \(k_\Omega\) and \(K_\Omega\) may be degenerate.  The domain
\(\Omega\) is called \emph{Kobayashi hyperbolic} if \(K_\Omega\) is a
genuine distance, that is,
\[
K_\Omega(x,y)>0
\qquad\text{whenever }x\ne y.
\]

When the domain is the unit disk \(\D\) or the right half-plane \(\HH\),
the Kobayashi metric is the Poincar\'e metric with our normalization:
\[
 k_\D(\zeta;v)=\frac{|v|}{1-|\zeta|^2},
 \qquad
 k_\HH(z;v)=\frac{|v|}{2\Re z}.
\]

The main functorial property of the Kobayashi metric is holomorphic
contraction. Let \(\Omega_1\subset\C^{n_1}\) and
\(\Omega_2\subset\C^{n_2}\) be domains, and let
\(F:\Omega_1\to\Omega_2\) be holomorphic. Then, for every
\(z\in\Omega_1\), \(v\in\C^{n_1}\), and \(x,y\in\Omega_1\),
the following contraction inequalities hold:
\[
k_{\Omega_2}\bigl(F(z);\dd F_z(v)\bigr)
\le
k_{\Omega_1}(z;v),
\qquad
K_{\Omega_2}\bigl(F(x),F(y)\bigr)
\le
K_{\Omega_1}(x,y).
\]
Here \(\dd F_z:\C^{n_1}\to\C^{n_2}\) denotes the differential of \(F\)
at \(z\).
Consequently, both quantities are biholomorphically invariant, and an inclusion
\(\Omega_1\subset\Omega_2\) gives
\[
 k_{\Omega_2}(z;v)\le k_{\Omega_1}(z;v),
 \qquad
 K_{\Omega_2}(x,y)\le K_{\Omega_1}(x,y).
\]
Another consequence will be useful later. Let \(\Omega_1\subset\C^{n_1}\) and \(\Omega_2\subset\C^{n_2}\) be domains. Suppose that there exist holomorphic maps
\[
i:\Omega_1\to\Omega_2,\qquad r:\Omega_2\to\Omega_1
\]
such that \(r\circ i=\operatorname{id}_{\Omega_1}\). Applying holomorphic contraction to \(i\) and \(r\), we obtain
\[
K_{\Omega_2}(i(x),i(y))=K_{\Omega_1}(x,y),\qquad x,y\in\Omega_1.
\]
Thus \(i\) is an isometric embedding for the Kobayashi pseudodistances, and \(i(\Omega_1)\) is a holomorphic retract of \(\Omega_2\).

Substantial progress has been made in estimating the Kobayashi metric
on convex domains.  Here, we record only the estimates needed below.

\begin{lemma}[{\cite[Theorem~5]{GrahamConvex}}]
\label{lem:graham}
If \(\Omega\subsetneq\C^n\) is convex, \(z\in\Omega\), and \(v\in\C^n\setminus\{0\}\), then
\begin{equation*}
 \frac{|v|}{2\delta_\Omega(z;v)}
 \le k_\Omega(z;v)
 \le\frac{|v|}{\delta_\Omega(z;v)}.
\end{equation*}
\end{lemma}

\begin{lemma}[{\cite[Proposition~2(i)]{NikolovTrybula}}]
\label{lem:normal-lower}
	If \(\Omega\subsetneq\C^n\) is convex and
	\(x,y\in\Omega\), then
	\begin{equation*}
		K_\Omega(x,y)
		\ge
		\frac12
		\left|
		\log\frac{\delta_\Omega(x)}{\delta_\Omega(y)}
		\right|.
	\end{equation*}
\end{lemma}

\begin{lemma}\label{lem:ball-radial-distance}
Let \(B(z_0,R)\subset\C^n\) be a Euclidean ball.  Suppose that \(x,y\in B(z_0,R)\) lie on the same radial segment from \(z_0\) and satisfy
\(\delta_{B(z_0,R)}(x)<\delta_{B(z_0,R)}(y)\).
Then
\[
K_{B(z_0,R)}(x,y)
=
\frac12
\log
\left(
\frac{2R-\delta_{B(z_0,R)}(x)}
     {2R-\delta_{B(z_0,R)}(y)}
\frac{\delta_{B(z_0,R)}(y)}
     {\delta_{B(z_0,R)}(x)}
\right).
\]
\end{lemma}
\begin{proof}
This follows from the standard formula for the Kobayashi distance on
the unit ball and the biholomorphic invariance of the Kobayashi
distance; see, for example,
\cite[Corollary~2.3.5 and Theorem~11.2.1]{JP}.
\end{proof}

\subsection{Finite type}
\label{subsec:finite-type}

Let \(g\) be a smooth function defined in a neighborhood of \(0\in\C\), with complex variable \(\zeta\). The \emph{order of vanishing} of \(g\) at \(0\) is defined by
\[
\ord_0 g:=\min\left\{a+b:a,b\in\mathbb Z_{\ge0},\ \frac{\partial^{a+b}g}{\partial\zeta^a\partial\bar\zeta^b}(0)\ne0\right\},
\]
with the convention that \(\ord_0 g=+\infty\) if all derivatives of \(g\) vanish at \(0\).

If
\[
\phi=(\phi_1,\ldots,\phi_n):(\C,0)\longrightarrow(\C^n,\xi),\qquad \xi=(\xi_1,\ldots,\xi_n),
\]
is a nonconstant holomorphic curve germ, the \emph{order of \(\phi\) at \(0\)} is defined by
\[
\ord_0\phi:=\min_{1\le j\le n}\ord_0(\phi_j-\xi_j).
\]
Let \(\rho\) be a \(C^\infty\) defining function for \(\Omega\) near a
boundary point \(\xi\). The D'Angelo type of \(\partial\Omega\) at \(\xi\) is
\[
	\Delta_{\partial\Omega}(\xi)
	:=
	\sup_{\phi}
	\frac{\ord_0(\rho\circ\phi)}
	{\ord_0\phi},
\]
where the supremum is taken over all nonconstant holomorphic curve
germs through \(\xi\). This quantity is independent of the choice of
defining function; see \cite{DAngelo}.

For each \(v\in\C^n\setminus\{0\}\), regard \(\zeta\mapsto\rho(\xi+\zeta v)\) as a smooth function of the complex variable \(\zeta\) near \(0\). The line type of \(\partial\Omega\) at \(\xi\) is defined by
\[
\ell_{\partial\Omega}(\xi)
:=
\sup_{v\in\C^n\setminus\{0\}}
\ord_0\bigl(\rho(\xi+\zeta v)\bigr).
\]
Since an affine parametrization \(\zeta\mapsto\xi+\zeta v\) has order
one, the line type is precisely the restriction of the D'Angelo-type
definition to complex affine lines. By
\cite{McNealConvexFiniteType}, if \(\Omega\) is convex, then
\[
\Delta_{\partial\Omega}(\xi)
=
\ell_{\partial\Omega}(\xi)
\qquad
\text{for every }\xi\in\partial\Omega.
\]

Consequently, if
\[
\Delta_{\partial\Omega}(\xi)\le m,
\]
for every \(\xi\in\partial\Omega\cap U\), where \(U\) is a neighborhood of
\(p\), then
\[
\ord_0\bigl(\rho(\xi+\zeta v)\bigr)\le m,
\]
for every \(\xi\in\partial\Omega\cap U\) and every \(v\in\C^n\setminus\{0\}\).
Equivalently, for each such pair \((\xi,v)\) there exist
\(a,b\geq0\) with \(a+b\leq m\) such that
\[
	\frac{\partial^{a+b}}
	{\partial\zeta^a\partial\bar\zeta^b}
	\rho(\xi+\zeta v)\bigg|_{\zeta=0}\neq0.
\]

\begin{lemma}\label{lem:local-type-bound}
Let \(\Omega\subset\C^n\) be convex with smooth boundary near
\(p\in\partial\Omega\), and let \(m\ge2\) be an integer.
If \(\Delta_{\partial\Omega}(p)\le m\),
then \(\Delta_{\partial\Omega}(\xi)\le m\) at every boundary point
\(\xi\) in some neighborhood of \(p\).
\end{lemma}
\begin{proof}
Choose a smooth defining function \(\rho\) near \(p\), and denote
the coefficients of the degree-\(m\) Taylor polynomial of
\(\rho(q+\zeta v)-\rho(q)\), in \(\Re\zeta,\Im\zeta\), by
\(a_{jk}(q,v)\), \(1\le j+k\le m\).
For every unit \(v\), the line-type bound at \(p\) gives
\(\max_{1\le j+k\le m}|a_{jk}(p,v)|>0\).
Continuity and compactness of the unit sphere imply
\[
\eta:=\min_{|v|=1}\max_{1\le j+k\le m}|a_{jk}(p,v)|>0.
\]
After restricting \(q\) to a sufficiently small neighborhood of \(p\),
the same maximum is at least \(\eta/2\), uniformly for \(|v|=1\).
Thus every complex line through a nearby boundary point has contact
order at most \(m\). The equality of line type and D'Angelo type
for convex domains proves the result.
\end{proof}

A convex domain is called \emph{\(\C\)-proper} if it contains no
complex affine line.  The following lemma is a direct consequence of results
of Bracci and Saracco~\cite{BracciSaracco}.

\begin{lemma}\label{lem:finite-type-c-proper}
Let \(\Omega\subset\C^n\) be a convex domain, possibly unbounded, and
let \(p\in\partial\Omega\).  Suppose that \(\partial\Omega\) is smooth
near \(p\) and has finite D'Angelo type at \(p\).  Then \(\Omega\) is
\(\C\)-proper.  Consequently, \(\Omega\) is complete Kobayashi
hyperbolic.
\end{lemma}

\begin{proof}
By \cite[Proposition~1.2]{BracciSaracco}, after a complex affine change
of coordinates, which preserves D'Angelo type, we may write
\(\Omega=\Omega_0\times\C^\ell\),
where \(\ell\ge0\) and \(\Omega_0\) is complete Kobayashi hyperbolic.

Suppose that \(\ell\ge1\). Since
\[
\partial\Omega=\partial\Omega_0\times\C^\ell,
\]
write
\[
p=(p_0,w_0),\qquad p_0\in\partial\Omega_0,\quad w_0\in\C^\ell.
\]
Choose \(v\in\C^\ell\setminus\{0\}\) and define
\[
\gamma(\zeta)=(p_0,w_0+\zeta v),\qquad \zeta\in\C.
\]
Then \(\gamma(\C)\subset\partial\Omega\). Hence, if \(\rho\) is a smooth local defining function near \(p\), the nonconstant holomorphic curve \(\gamma\) satisfies
\(\rho\circ\gamma\equiv0\)
near \(0\).  It follows that
\[
\Delta_{\partial\Omega}(p)=+\infty,
\]
contrary to the finite-type hypothesis.  Therefore \(\ell=0\), and
\(\Omega\) is complete Kobayashi hyperbolic.  By
\cite[Theorem~1.1]{BracciSaracco}, this is equivalent to
\(\C\)-properness.
\end{proof}

\subsection{Global and local hyperbolicity constants}\label{sub2.5}

We recall the basic definitions of Gromov hyperbolicity; see
\cite{Gromov,BridsonHaefliger}.  Let \((X,d)\) be a metric space and
fix \(o\in X\).  The \emph{Gromov product} based at \(o\) is
\((x\mid y)_o=\frac12\bigl(d(x,o)+d(y,o)-d(x,y)\bigr)\).

For \(\delta\ge0\), the space \((X,d)\) is called
\emph{\(\delta\)-hyperbolic} if
\begin{equation}\label{eq:gromov-product-inequality}
(x\mid z)_o
\ge
\min\{(x\mid y)_o,(y\mid z)_o\}-\delta,
\end{equation}
for all \(o,x,y,z\in X\).  It is \emph{Gromov hyperbolic} if this holds
for some finite \(\delta\).  Equivalently, \((X,d)\) is
\(\delta\)-hyperbolic if and only if
\begin{equation*}
d(x,z)+d(y,w)
\le
\max\{d(x,y)+d(z,w),d(x,w)+d(y,z)\}
+2\delta,
\end{equation*}
for all \(x,y,z,w\in X\).  For geodesic spaces, these conditions are
also equivalent, up to a universal change of the constant, to the
thin-triangle formulation; see
\cite[Chapter~III.H]{BridsonHaefliger}.

\begin{definition}\label{def:hyperbolicity-constant}
For a metric space \((X,d)\), its \emph{hyperbolicity constant} is
\[
\delta(X,d)
=
\inf\{\delta\ge0:(X,d)\text{ is }\delta\text{-hyperbolic}\},
\]
with the convention that \(\delta(X,d)=+\infty\) if no finite
\(\delta\) exists.
\end{definition}

Equivalently, for four points \(x_1,x_2,x_3,x_4\in X\), consider the
three sums
\[
d(x_1,x_2)+d(x_3,x_4),\qquad
d(x_1,x_3)+d(x_2,x_4),\qquad
d(x_1,x_4)+d(x_2,x_3).
\]
If \(S_{\max}\) and \(S_{\mathrm{mid}}\) denote the largest and the
second largest of these three numbers, then
\begin{equation}\label{eq:global-fourpoint-constant}
\delta(X,d)
=
\frac12
\sup_{x_1,x_2,x_3,x_4\in X}
\bigl(S_{\max}-S_{\mathrm{mid}}\bigr).
\end{equation}
In particular, \(X\) is Gromov hyperbolic if and only if
\(\delta(X,d)<\infty\).

\begin{example}
For the unit ball $\mathbb B^n\subset\mathbb C^n$,
\[
\delta(\mathbb B^n,K_{\mathbb B^n})
=
\begin{cases}
\frac12\log 2, & n=1,\\
\log 2, & n\geq 2.
\end{cases}
\]

\end{example}
\begin{proof}
 $K_{\mathbb B^n}$ is the complex hyperbolic metric normalized to
have holomorphic sectional curvature $-4$. Thus, when $n=1$, it is the
real hyperbolic plane of curvature $-4$, and the first identity follows
by scaling from \cite[Corollary~5.4]{NicaSpakula}. When $n\geq2$, its
sectional curvatures lie in $[-4,-1]$, so it is $\operatorname{CAT}(-1)$
and hence $\log 2$-hyperbolic by
\cite[Theorems~4.2 and~5.1]{NicaSpakula}. Conversely, $\mathbb B^n$
contains an isometrically embedded totally real hyperbolic plane of
curvature $-1$, and \cite[Corollary~5.4]{NicaSpakula} shows that the
constant $\log 2$ is optimal. For the Rips definition in terms of thin triangles, the corresponding
constants are $\frac12\log(1+\sqrt2)$ for $n=1$ and
$\log(1+\sqrt2)$ for $n\ge2$.
\end{proof}
The global hyperbolicity constant of \((\Omega,K_\Omega)\) involves
configurations throughout the domain.  To measure the geometry near a
boundary point \(p\), we restrict the points to small Euclidean
neighborhoods of \(p\), while keeping the ambient Kobayashi distance
\(K_\Omega\).  The notion above applies to arbitrary metric spaces.  In this
paper, however, we only need a local version for the Kobayashi metric.

\begin{definition}\label{def:local-hyperbolicity-constant}
Let \(\Omega\subset\C^n\) be a Kobayashi hyperbolic domain, and let
\(p\in\partial\Omega\).  The \emph{local hyperbolicity constant} of
\(\Omega\) at \(p\) is
\begin{equation*}
\delta_{\mathrm{loc}}(\Omega,K_\Omega,p)
=
\inf_{V\ni p}
\delta\bigl(
\Omega\cap V,
\left.K_\Omega\right|_{\Omega\cap V}
\bigr),
\end{equation*}
where \(V\) ranges over all Euclidean neighborhoods of \(p\) and
\(\left.K_\Omega\right|_{\Omega\cap V}\) denotes the restriction of
\(K_\Omega\) to \(\Omega\cap V\).
When the metric is clear from the context, we write
\(\delta_{\mathrm{loc}}(\Omega,p)\).
\end{definition}

If \(V_1\subset V_2\), then every quadruple in \(\Omega\cap V_1\) is
also a quadruple in \(\Omega\cap V_2\).  Hence
\[
\delta\bigl(
\Omega\cap V_1,
\left.K_\Omega\right|_{\Omega\cap V_1}
\bigr)
\le
\delta\bigl(
\Omega\cap V_2,
\left.K_\Omega\right|_{\Omega\cap V_2}
\bigr).
\]
Therefore
\[
\delta_{\mathrm{loc}}(\Omega,p)
=
\lim_{r\to0^+}
\delta\bigl(
\Omega\cap B(p,r),
\left.K_\Omega\right|_{\Omega\cap B(p,r)}
\bigr),
\]
and the limit exists by monotonicity.

Taking \(V_2=\C^n\) in the monotonicity above gives
\(\delta_{\mathrm{loc}}(\Omega,p)\le\delta(\Omega,K_\Omega)\).

\subsection{No global bound in terms of type}

The global hyperbolicity constant does not provide a meaningful
quantitative measure of the dependence on boundary type.  In fact, even
in a fixed dimension and among bounded strongly convex domains, all of
boundary type \(2\), the constants
\(\delta(\Omega,K_\Omega)\) can be arbitrarily large.  Hence the global
constant cannot be controlled in terms of the dimension and the
D'Angelo type alone.  This is why we instead consider the local
hyperbolicity constant.

\begin{example}\label{ex:no-global-bound}
Fix \(n\ge2\).  For every integer \(j\ge2\), define
\[
F_j(z)
=
\left(\sum_{k=1}^n |z_k|^{2j}\right)^{1/j}
+
\frac1j\sum_{k=1}^n |z_k|^2,
\]
and set
\(\Omega_j=\{z\in\C^n:F_j(z)<1\}\).
Then each \(\Omega_j\) is a bounded strongly convex domain with smooth
boundary, and
\[
\delta(\Omega_j,K_{\Omega_j})\longrightarrow\infty.
\]
\end{example}

\begin{proof}
The function
\[
z\longmapsto
\left(\sum_{k=1}^n|z_k|^{2j}\right)^{1/j}
\]
is convex. At a point \(z\) where \(F_j\) is smooth, let
\(D^2F_j(z)\) denote its real Hessian, identifying
\(\C^n\) with \(\R^{2n}\). For \(v\in\R^{2n}\), we write
\(D^2F_j(z)[v,v]=\left.\frac{d^2}{dt^2}F_j(z+tv)\right|_{t=0}\),
where \(t\in\R\). Convexity of the first summand in \(F_j\) gives
\(D^2F_j(z)[v,v]\ge(2/j)|v|^2\).
Since \(\partial\Omega_j\) does not contain the origin, \(F_j\) is
smooth near \(\partial\Omega_j\).  Moreover, \(F_j\) is homogeneous of
real degree two, so Euler's identity gives
\[
\langle \nabla F_j(z),z\rangle=2F_j(z)=2,
\qquad z\in\partial\Omega_j.
\]
Thus \(\nabla F_j\ne0\) on \(\partial\Omega_j\), and \(F_j=1\) is a
regular level set.  The Hessian estimate above now shows that \(\Omega_j\)
is strongly convex with smooth boundary.  In particular,
\(\partial\Omega_j\) is strongly pseudoconvex, so by
Balogh--Bonk \cite{BaloghBonk},
\[
\delta(\Omega_j,K_{\Omega_j})<\infty,
\]
for every fixed \(j\).

We next compare \(\Omega_j\) with the polydisc.  Clearly,
\(
\Omega_j\subset\D^n.
\)
Set
\[
\rho_j
=
\left(n^{1/j}+\frac nj\right)^{-1/2}.
\]
If \(z\in\rho_j\D^n\), then the coordinate inequalities are strict, and
hence
\[
F_j(z)
<
\left(n^{1/j}+\frac nj\right)\rho_j^2
=1.
\]
Thus
\[
\rho_j\D^n\subset\Omega_j\subset\D^n.
\]
Since \(\rho_j\to1\), monotonicity of the Kobayashi distance gives
\[
K_{\D^n}\le K_{\Omega_j}\le K_{\rho_j\D^n},
\]
and hence
\(
K_{\Omega_j}\rightarrow K_{\D^n}
\)
locally uniformly on \(\D^n\times\D^n\).

Define \(\iota:\R^2\to\D^n\) by
\(\iota(s,t):=(\tanh s,\tanh t,0,\ldots,0)\).
This map is an isometric embedding of
\((\R^2,\|\cdot\|_\infty)\) into \((\D^n,K_{\D^n})\).
Consider the four points in
\(\R^2\)
\[
(R,0),\qquad
(-R,0),\qquad
(0,R),\qquad
(0,-R).
\]
For these points, the three four-point sums are
\(4R, 2R, 2R\),
so that
\(
S_{\max}-S_{\mathrm{mid}}=2R.
\)

Fix \(A>0\) and choose \(R>2A\).  The four points obtained by applying
\(\iota\) are contained in \(\Omega_j\) for all sufficiently large
\(j\).  The local uniform convergence of \(K_{\Omega_j}\) to
\(K_{\D^n}\) implies that
\[
S_{\max}(\Omega_j)-S_{\mathrm{mid}}(\Omega_j)>2A,
\]
for all sufficiently large \(j\).  Therefore,
by \eqref{eq:global-fourpoint-constant},
\(
\delta(\Omega_j,K_{\Omega_j})>A.
\)
Since \(A>0\) is arbitrary, we conclude that
\[
\delta(\Omega_j,K_{\Omega_j})\longrightarrow\infty.
\]
\end{proof}

\section{Canonical non-isotropic geometry near convex finite-type boundary points}
\label{sec:balanced-geometry}

The goal of this section is to construct a canonical non-isotropic geometry
near a finite-type boundary point of a convex domain.  The construction is
motivated by McNeal's polydiscs \cite[Section~2]{McNealBergman}, whose side
lengths, relative to an extremal basis, reflect the non-isotropic boundary
geometry of the domain.

For quantitative estimates, however, a basis-dependent description may be less
convenient.  Comparing the geometry at different centers requires repeated
coordinate changes and comparison estimates.  The corresponding constants
may depend on both the type and the dimension, and repeated applications can
accumulate these dependencies.  As a result, the final bounds may acquire an
unnecessary dependence on the dimension or a stronger dependence on the type
than the geometry itself suggests.

Instead, we keep the information in every complex direction and collect it
into a complex balanced convex set.  Its Minkowski functional gives a
basis-independent non-isotropic norm.  The comparison properties of these
norms then lead to a quasi-distance, which will be used to obtain
the linear dependence on the type for the local Gromov hyperbolicity
constant.

Throughout this section, let \(\Omega\subsetneq\C^n\) be a convex domain,
possibly unbounded, and let \(\rho=\rho_\Omega\) denote its signed Euclidean
distance function, as defined in
Subsection~\ref{subsec:signed-distance}.  By
Lemma~\ref{lem:signed-distance-global-convexity}, \(\rho\) is convex on
\(\C^n\); moreover, \(\Omega=\{\rho<0\}\) and
\(\partial\Omega=\{\rho=0\}\).
\subsection{Non-isotropic balanced convex sets}
\label{subsec:directional-bodies}

For \(q\in\Omega\) and \(\eps>0\), define
\(\Omega_{q,\eps}=\{z:\rho(z)<\rho(q)+\eps\}\).
This is an open convex set containing \(q\): openness follows from the
continuity of the signed distance function, while convexity follows from the
convexity of \(\rho\).  Thus it provides a convex ambient level set in which
to measure the size of complex discs centered at \(q\).
Following McNeal \cite{McNealBergman} and Bruna--Charpentier--Dupain
\cite{BrunaCharpentierDupain}, for \(v\in\C^n\setminus\{0\}\), define
\begin{equation*}
	\tau(q,v,\eps)
	=
	\sup\left\{
	r>0:
	q+\lambda v\in\Omega_{q,\eps}
	\text{ for every }|\lambda|<r
	\right\},
\end{equation*}
which is the maximal radius in the parameter
\(\lambda\) for which the complex disc
\(\lambda\mapsto q+\lambda v\) remains in \(\Omega_{q,\eps}\).
The corresponding Euclidean radius in the complex line \(q+\C v\) is
\(\tau(q,v,\eps)|v|\).  The parameter radius \(\tau(q,v,\eps)\) depends
on the normalization of \(v\), whereas the Euclidean radius depends only
on the complex line spanned by \(v\).  In particular,
\begin{equation}\label{eq:tau-homogeneity}
 \tau(q,av,\eps)=|a|^{-1}\tau(q,v,\eps),\qquad a\in\C\setminus\{0\}.
\end{equation}

McNeal's construction selects an extremal basis by choosing \(n\)
distinguished directions successively and uses the corresponding lengths as
the axes of an adapted polydisc.  Here we keep the information in every
complex direction through \(\tau(q,v,\eps)\), rather than reducing it to a
finite collection of extremal directions.  This basis-independent viewpoint
allows us to describe the non-isotropic geometry intrinsically and to compare
the corresponding geometry at different centers.  It also provides the
natural framework for constructing the boundary quasi-distance below.

Recall that if \(B\subset\C^n\) is an open complex balanced convex neighborhood of the origin, then its Minkowski functional is the map
\[
\begin{aligned}
p_B:\C^n&\longrightarrow[0,\infty),\\
h&\longmapsto\inf\{t>0:h\in tB\}.
\end{aligned}
\]
The map \(p_B\) is finite, absolutely homogeneous, and subadditive. It may be degenerate when \(B\) is unbounded, whereas it is a norm when \(B\) is bounded; see, for example, \cite[Chapter~1]{RudinFA}.

For \(q\in\Omega\) and \(\eps>0\), define
\begin{equation*}
\mathbb B_\eps(q)=\left\{h\in\C^n:q+\lambda h\in\Omega_{q,\eps}\text{ for every }|\lambda|\le1\right\}.
\end{equation*}
By Lemma~\ref{lem:directional-norm} below, \(\mathbb B_\eps(q)\) is an
open complex balanced convex neighborhood of the origin.  Hence its
Minkowski functional is well defined; we denote it by
\(p_{q,\eps}:=p_{\mathbb B_\eps(q)}:\C^n\to[0,\infty)\).  The same lemma
relates this functional to the directional radius \(\tau\).

\begin{lemma}\label{lem:directional-norm}
	Let \(\Omega\subset\C^n\) be a convex domain, possibly unbounded,
	\(q\in\Omega\), and \(\eps>0\).  Then
	\(\mathbb B_\eps(q)\) is an open complex balanced convex neighborhood
	of the origin.  Its Minkowski functional \(p_{q,\eps}\) is finite,
	absolutely homogeneous, and subadditive.  Moreover, the identity
	\begin{equation*}
		p_{q,\eps}(h)
		=
		\frac{1}{\tau(q,h,\eps)}
	\end{equation*}
	holds for every \(h\in\C^n\setminus\{0\}\), with \(1/\infty\)
	interpreted as \(0\).  In particular, if
	\(\mathbb B_\eps(q)\) is bounded, then \(p_{q,\eps}\) is a norm on
	\(\C^n\).
\end{lemma}

\begin{proof}
	Let \(h_1,h_2\in\mathbb B_\eps(q)\), \(0\le t\le1\), and
	\(|\lambda|\le1\).  Since
	\[
	q+\lambda\bigl(th_1+(1-t)h_2\bigr)
	=
	t(q+\lambda h_1)+(1-t)(q+\lambda h_2)
	\]
	and \(\Omega_{q,\eps}\) is convex, we have
	\(th_1+(1-t)h_2\in\mathbb B_\eps(q)\).  Thus
	\(\mathbb B_\eps(q)\) is convex.  If
	\(h\in\mathbb B_\eps(q)\) and \(|a|\le1\), then
	\[
	q+\lambda(ah)=q+(a\lambda)h\in\Omega_{q,\eps},
	\qquad |\lambda|\le1,
	\]
	so \(\mathbb B_\eps(q)\) is complex balanced.
	
	Since \(\Omega_{q,\eps}\) is open and contains \(q\),
	\(\mathbb B_\eps(q)\) contains a neighborhood of the origin.  To see
	that it is open, fix \(h\in\mathbb B_\eps(q)\).  The compact set
	\(\{q+\lambda h:|\lambda|\le1\}\)
	is contained in the open set \(\Omega_{q,\eps}\).  Hence the same is
	true with \(h\) replaced by every \(h'\) sufficiently close to \(h\).
	
	It follows from the standard properties of Minkowski functionals that
	\(p_{q,\eps}\) is finite, absolutely homogeneous, and subadditive.
	For \(t>0\), the following equivalences hold:
	\[
		h\in t\mathbb B_\eps(q)
	\iff
		\frac{h}{t}\in\mathbb B_\eps(q)\iff
		\tau\left(q,\frac{h}{t},\eps\right)>1\iff
		t\,\tau(q,h,\eps)>1,
	\]
	where the last equivalence follows from
	\eqref{eq:tau-homogeneity}.  Taking the infimum over \(t>0\) gives
	the asserted identity, with the convention \(1/\infty=0\).
	
	Finally, if \(\mathbb B_\eps(q)\) is bounded, then its Minkowski
	functional vanishes only at the origin.  Hence \(p_{q,\eps}\) is a
	norm.
\end{proof}

We will also use the following elementary comparison of these
Minkowski functionals.

\begin{lemma}\label{lem:p-scale}
	Let \(q\in\Omega\), \(\eps>0\), and \(\lambda\ge1\).  Then
	\[
	\frac1\lambda p_{q,\eps}(h)
	\le
	p_{q,\lambda\eps}(h)
	\le
	p_{q,\eps}(h),
	\qquad h\in\C^n.
	\]
	Equivalently, for every \(v\ne0\), the following bounds hold:
	\[
	\tau(q,v,\eps)
	\le
	\tau(q,v,\lambda\eps)
	\le
	\lambda\,\tau(q,v,\eps).
	\]
	In particular, \(p_{q,\eps}(h)\) is nonincreasing in \(\eps\), while
	\(\tau(q,v,\eps)\) is nondecreasing in \(\eps\).
\end{lemma}
\begin{proof}
	Since
	\(
	\Omega_{q,\eps}\subset\Omega_{q,\lambda\eps},
	\)
	we have
\(	
	\mathbb B_\eps(q)\subset\mathbb B_{\lambda\eps}(q),
\)
	and therefore
	\(p_{q,\lambda\eps}(h)\le p_{q,\eps}(h)\).
	For the reverse comparison, let
	\(h\in\mathbb B_{\lambda\eps}(q)\). Then, for every \(|\mu|\le1\),
	convexity of \(\rho\) gives
	\[
	\begin{aligned}
		\rho\left(q+\frac{\mu h}{\lambda}\right)
		&=
		\rho\left(
		\left(1-\frac1\lambda\right)q
		+\frac1\lambda(q+\mu h)
		\right)\\
		&\le
		\left(1-\frac1\lambda\right)\rho(q)
		+\frac1\lambda\rho(q+\mu h)
		<\rho(q)+\eps.
	\end{aligned}
	\]
	Hence \(h/\lambda\in\mathbb B_\eps(q)\), and thus
	\(
	\mathbb B_{\lambda\eps}(q)
	\subset
	\lambda\mathbb B_\eps(q).
	\)
	Taking Minkowski functionals yields
	\[
	p_{q,\lambda\eps}(h)
	\ge
	\frac1\lambda p_{q,\eps}(h).
	\]
	This proves the lemma.
\end{proof}

\subsection{Non-isotropic scale control via Chebyshev estimates}
\label{subsec:finite-type-scales}
Lemma~\ref{lem:p-scale} gives monotonicity of \(\tau(q,v,\eps)\) in
\(\eps\) and a linear upper bound under dilation of this parameter.
It does not, however, give the quantitative lower growth estimate
needed below when \(\eps\) is enlarged by a prescribed factor.  We will obtain it using the finite-type assumption,
with explicit dependence on the D'Angelo type bound \(m\).  The key ingredient is the
extremal property of Chebyshev polynomials.

We now impose the local boundary hypotheses.  Fix
\(p\in\partial\Omega\) and an integer \(m\ge2\), and assume that
\(\partial\Omega\) is smooth near \(p\) and has D'Angelo type at most
\(m\) at \(p\). By Lemma~\ref{lem:local-type-bound}, the same type
bound holds throughout a sufficiently small boundary neighborhood. By
Lemmas~\ref{lem:signed-distance-tubular}
and~\ref{lem:signed-distance-smooth}, there is a neighborhood \(U_0\) of
\(p\) on which \(\rho\) is a smooth defining function and the
nearest-point projection \(\pi\) is well defined.  These local facts are
needed below to obtain Taylor estimates that are uniform in the center
\(q\) and the complex direction.  The working neighborhood \(U\subset U_0\)
will therefore be chosen, and reduced when necessary, in the statements
where those uniform estimates are established.

Chebyshev polynomials are classical objects in approximation theory and arise
naturally in extremal problems for polynomials.  A basic question is how large
a polynomial of fixed degree can become outside an interval when it is
uniformly bounded on that interval.  Chebyshev polynomials give the optimal
bound: among real polynomials \(P\) of degree at most \(m\) satisfying
\(|P(x)|\le1\) on \([-1,1]\), the Chebyshev polynomial \(T_m\) has the
maximal possible growth outside the interval; see
\cite{Rivlin,SachdevaVishnoi}.

The Chebyshev polynomial of the first kind \(T_m\) is characterized by
\(T_m(\cos\theta)=\cos(m\theta)\), which implies
\(|T_m(x)|\le1\) for \(-1\le x\le1\).  It also satisfies the
alternation property
\[
	T_m\left(\cos\frac{j\pi}{m}\right)=(-1)^j,
	\qquad j=0,\ldots,m.
\]
For \(A\ge1\), its growth outside \([-1,1]\) is given by
\begin{equation}\label{eq:chebyshev-formula}
	\begin{aligned}
		T_m(A)
		&=\cosh\bigl(m\,\operatorname{arcosh}A\bigr)\\
		&=
		\frac{(A+\sqrt{A^2-1})^m+
			(A-\sqrt{A^2-1})^m}{2}.
	\end{aligned}
\end{equation}
See, for example, \cite[Proposition~2.5]{SachdevaVishnoi}.

This alternation property gives the following
extremal estimate.

\begin{lemma}\label{lem:chebyshev-extremal}
Let \(P\) be a real polynomial of degree at most \(m\) satisfying
\(|P(x)|\le1\) for every \(x\in[-1,1]\).  Then the estimate
\begin{equation*}
 |P(A)|\le T_m(A)
\end{equation*}
holds for every \(A\ge1\).
\end{lemma}

\begin{proof}This is the classical extremal property of Chebyshev polynomials; see, for
example, \cite[Proposition~2.4]{SachdevaVishnoi}.  We include the standard argument for
completeness.

Since \(T_m(1)=1\), the case \(A=1\) follows directly from the hypothesis.
Suppose, to the contrary, that
\[
|P(A)|>T_m(A),
\]
for some \(A>1\).  Define
\[
Q(x)=\frac{T_m(A)}{P(A)}P(x).
\]
Then
\[
|Q(x)|<1,\qquad x\in[-1,1],
\]
and
\(
Q(A)=T_m(A).
\)
Consider
\[
R(x)=T_m(x)-Q(x).
\]
At the alternating points
\[
x_j=\cos\frac{j\pi}{m},
\qquad j=0,\ldots,m,
\]
we have
\(
T_m(x_j)=(-1)^j.
\)
Since \(|Q(x_j)|<1\), the values \(R(x_j)\) have alternating signs.
Therefore \(R(x)\) vanishes at least once between each pair of consecutive
points \(x_j\), giving at least \(m\) zeros in \([-1,1]\).  Moreover,
\(
R(A)=0.
\)
Hence \(R(x)\) has at least \(m+1\) distinct zeros, while its degree is at
most \(m\).  This contradiction proves the claim.
\end{proof}

In our setting, the Taylor polynomial obtained by restricting the defining
function to a complex line is a polynomial of degree at most \(m\) in two
real variables.  By restricting this polynomial to real lines through the
origin, the preceding one-dimensional Chebyshev estimate gives the following
useful consequence.

\begin{lemma}
	\label{lem:radial-chebyshev}
    Let \(Q:\R^2\to\R\) be a polynomial of total degree at most \(m\).
    Under the identification \(\C\simeq\R^2\), write
	\(Q(\zeta):=Q(\operatorname{Re}\zeta,\operatorname{Im}\zeta)\).
	The following estimate holds for every \(A\ge1\) and \(r>0\):
	\begin{equation*}
		\sup_{|\zeta|\le Ar}|Q(\zeta)|
		\le
		T_m(A)\sup_{|\zeta|\le r}|Q(\zeta)|.
	\end{equation*}
\end{lemma}

\begin{proof}
	Set \(M_r:=\sup_{|\xi|\le r}|Q(\xi)|\).
	If \(M_r=0\), then \(Q\)
	vanishes on an open disk and hence is identically zero, so the conclusion
	is immediate.  We may therefore assume that \(M_r>0\).

	For each \(\zeta\in\C\), define \(P_\zeta:\R\to\R\) by
	\[
	P_\zeta(t):=\frac{1}{M_r}Q\left(\frac{t\zeta}{A}\right).
	\]
	Since
	\(\operatorname{Re}(t\zeta/A)=t\operatorname{Re}\zeta/A\) and similarly
	for the imaginary part, \(P_\zeta\) is a real polynomial in \(t\) of
	degree at most \(m\).

	Now suppose that \(|\zeta|\le Ar\).  If \(|t|\le1\), then
	\(|t\zeta/A|\le r\), and hence \(|P_\zeta(t)|\le1\).

	Applying Lemma~\ref{lem:chebyshev-extremal} to \(P_\zeta\) at \(t=A\)
	gives
	\[
	|Q(\zeta)|=M_r|P_\zeta(A)|
	\le M_rT_m(A)
	=T_m(A)\sup_{|\xi|\le r}|Q(\xi)|.
	\]
	Taking the supremum over \(|\zeta|\le Ar\) proves the lemma.
\end{proof}
\begin{remark}\label{rem:chebyshev-versus-convex-polynomials}
Related polynomial estimates have been used in the study of convex
finite-type domains.  A typical result, used by Bruna--Nagel--Wainger
\cite{BrunaNagelWainger}, is the following.  If the real polynomial
\[
P(t)=\sum_{j=2}^{m}a_jt^j,\qquad a_j\in\R,
\]
is convex on \([0,T]\), then there exists a positive constant \(c_m\), depending only on \(m\), such that
\[
P(t)\ge c_m\sum_{j=2}^{m}|a_j|t^j,
\qquad 0\le t\le T.
\]
Consequently, for \(A\ge1\) and \(0\le t\le T\), we obtain
\[
\begin{aligned}
	|P(At)|
	\le
	\sum_{j=2}^{m}|a_j|A^jt^j
	\le
	A^m\sum_{j=2}^{m}|a_j|t^j
	\le
	c_m^{-1}A^mP(t).
\end{aligned}
\]
Similar arguments appear in McNeal \cite{McNealBergman}.
Bruna--Charpentier--Dupain \cite[Lemma~2.1(b)]{BrunaCharpentierDupain} also obtained a corresponding estimate in the
complex-variable setting.

In our setting, we apply Lemma~\ref{lem:radial-chebyshev} directly
to real polynomials of total degree at most \(m\).
No convexity assumption is needed, and the factor \(T_m(A)\) is explicit.
This allows us to keep track of the dependence on the D'Angelo type
bound \(m\) in the later estimates.
\end{remark}

\bigskip

We now connect Lemma~\ref{lem:radial-chebyshev} to the directional scale
\(\tau\) by approximating the variation of \(\rho\) along each complex line
by its Taylor polynomial.  For \(q\in U_0\) and
\(v\in\C^n\setminus\{0\}\), define \(F_{q,v}:\C\to\R\) by
\(F_{q,v}(\zeta):=\rho(q+\zeta v)-\rho(q)\).
This function is smooth near the origin because \(\rho\) is smooth on
\(U_0\).  When \(q\in\Omega\cap U_0\), the definition of
\(\tau(q,v,\eps)\) says that it is the supremum of the radii \(r>0\)
for which \(F_{q,v}(\zeta)<\eps\) whenever \(|\zeta|<r\).
Thus estimates for \(F_{q,v}\) on disks translate directly into
estimates for \(\tau\).

Under the identification \(\C\simeq\R^2\), let \(P_{q,v}\) be the
Taylor polynomial of \(F_{q,v}\) at the origin through total degree
\(m\), and let \(R_{q,v}:=F_{q,v}-P_{q,v}\) be the Taylor remainder.
Then
\(P_{q,v}:\R^2\to\R\) is a polynomial of total degree at most \(m\), and it
has zero constant term because \(F_{q,v}(0)=0\).  Whenever a function \(G\)
is defined on the closed disk of radius \(r\), write
\(\|G\|_r:=\sup_{|\zeta|\le r}|G(\zeta)|\).

The finite-type assumption gives a uniform lower bound for the polynomial
part \(P_{q,v}\), while the smoothness of \(\rho\) gives a uniform estimate
for the remainder \(R_{q,v}\).  The following lemma makes these estimates
uniform in the center and the unit direction.
\begin{lemma}\label{lem:uniform-line-jets}
	There exist a neighborhood \(U\Subset U_0\) of \(p\), constants
	\(a,b>0\), and a radius \(r_0\in(0,1]\) such that the estimates
	\begin{equation*}
		\|P_{q,v}\|_r\ge ar^m,
		\qquad
		\|R_{q,v}\|_r\le br^{m+1}
	\end{equation*}
	hold for every \(q\in\Omega\cap U\), every unit vector \(v\in\C^n\),
	and every \(r\in(0,r_0]\).
\end{lemma}

\begin{proof}
	Choose a compact subset \(K\subset\partial\Omega\cap U_0\), containing \(p\) in its relative interior, such that the D'Angelo type of \(\partial\Omega\)
	is at most \(m\) on \(K\).  For \(\xi\in K\) and \(|v|=1\), write
	\[
	P_{\xi,v}(\zeta)
	=
	\sum_{\substack{j,k\ge0\\1\le j+k\le m}}
	a_{jk}(\xi,v)
	(\operatorname{Re}\zeta)^j
	(\operatorname{Im}\zeta)^k.
	\]
	The finite line-type bound, as explained in
	Subsection~\ref{subsec:finite-type}, gives
	\begin{equation*}
		\max_{1\le j+k\le m}|a_{jk}(\xi,v)|>0,
		\qquad
		\xi\in K,\quad |v|=1.
	\end{equation*}
	The coefficients \(a_{jk}(\xi,v)\) depend continuously on
	\((\xi,v)\).  Let \(S^{2n-1}=\{v\in\C^n:|v|=1\}\) be the unit
	sphere in \(\C^n\).  Since
	\(K\times S^{2n-1}\) is compact, there exists \(\eta>0\) such that
	\[
	\max_{1\le j+k\le m}|a_{jk}(\xi,v)|
	\ge \eta,
	\qquad
	\xi\in K,\quad v\in S^{2n-1}.
	\]
	
	Choose \(U\Subset U_0\) so that
	Lemmas~\ref{lem:signed-distance-tubular}
	and~\ref{lem:signed-distance-smooth} apply and
	\(\pi(\Omega\cap U)\subset K\).
	By continuity of the derivatives of \(\rho\), after shrinking \(U\)
	once more we obtain
	\[
		\max_{1\le j+k\le m}|a_{jk}(q,v)|
		\ge \frac{\eta}{2},
		\qquad
		q\in\Omega\cap U,\quad |v|=1.
	\]
	Consider the finite-dimensional vector space of real polynomials on
	\(\R^2\simeq\C\) of total degree at most \(m\) and with zero constant
	term.  Equivalence of norms on this
	space gives a constant \(c_m>0\), depending only on \(m\), such that
	\[
		\sup_{|\zeta|\le1}|P(\zeta)|
		\ge
		c_m\max_{j,k}|c_{jk}|,
	\]
	for every
	\[
	P(\zeta)
	=
	\sum_{1\le j+k\le m}
	c_{jk}
	(\operatorname{Re}\zeta)^j
	(\operatorname{Im}\zeta)^k.
	\]
	Apply this norm-equivalence estimate to
	\(P_{q,v}(r\zeta)\).  If \(0<r\le1\), then the coefficient of total
	degree \(j+k\) is multiplied by \(r^{j+k}\), and
	\(
	r^{j+k}\ge r^m.
\)
	Together with the uniform coefficient lower bound, this gives
	\[
	\|P_{q,v}\|_r
	\ge
	c_m\frac{\eta}{2}r^m.
	\]
	Thus the first estimate in Lemma~\ref{lem:uniform-line-jets} holds with
	\(
	a=c_m\frac{\eta}{2}.
	\)
	
	Finally, choose \(r_0\in(0,1]\) and shrink \(U\), if necessary, so
	that all points \(q+\zeta v\), with \(q\in\Omega\cap U\),
	\(|v|=1\), and \(|\zeta|\le r_0\), lie in a fixed compact subset of
	\(U_0\).  The derivatives of \(\rho\) of order \(m+1\) are uniformly
	bounded there.  Taylor's formula therefore gives
	\[
	|R_{q,v}(\zeta)|
	\le
	b|\zeta|^{m+1},
	\qquad |\zeta|\le r_0,
	\]
	for a uniform constant \(b>0\).  Hence
	\[
	\|R_{q,v}\|_r\le br^{m+1},
	\qquad 0<r\le r_0,
	\]
	which proves the second estimate.
\end{proof}

Combining Lemma~\ref{lem:uniform-line-jets} with
Lemma~\ref{lem:radial-chebyshev}, we obtain the following
property, which is the key quantitative input for the non-isotropic geometry
constructed below.

\begin{lemma}\label{lem:finite-type-dilation}
	After shrinking \(U\) if necessary, the directional radii satisfy
	\begin{equation*}
		\lim_{\eps\to0^+}
		\sup_{\substack{q\in\Omega\cap U\\ |v|=1}}
		\tau(q,v,\eps)
		=0.
	\end{equation*}
	Moreover, there exists \(\eps_*>0\) such that the estimate
	\begin{equation*}
		\tau\bigl(q,v,2T_m(5)\eps\bigr)
		\ge
		5\,\tau(q,v,\eps)
	\end{equation*}
	holds for every \(q\in\Omega\cap U\), \(v\ne0\), and
	\(0<\eps<\eps_*\).
	Equivalently, Lemma~\ref{lem:directional-norm} gives
	\[
	p_{q,2T_m(5)\eps}(h)
	\le
	\frac15 p_{q,\eps}(h)
	\]
	for every \(h\in\C^n\).
\end{lemma}
\begin{proof}
	Let \(a,b,r_0>0\) be the constants in
	Lemma~\ref{lem:uniform-line-jets}.  After decreasing \(r_0\), we may
	assume that
	\begin{equation}\label{eq:relative-remainder}
		\|R_{q,v}\|_r
		\le
		\frac14\|P_{q,v}\|_r,
		\qquad 0<r\le r_0,
	\end{equation}
	for every \(q\in\Omega\cap U\) and \(|v|=1\).  Indeed,
	Lemma~\ref{lem:uniform-line-jets} gives
	\[
	\|R_{q,v}\|_r\le br^{m+1},
	\qquad
	\|P_{q,v}\|_r\ge ar^m,
	\]
	so it is enough to take \(r_0\) sufficiently small.
	
	We first prove the uniform shrinking assertion.  Fix
	\(q\in\Omega\cap U\), \(|v|=1\), and let
	\(0<r<\min\{\tau(q,v,\eps),r_0\}\).
	By the definition of \(\tau(q,v,\varepsilon)\), we have
	\(F_{q,v}(\zeta)<\varepsilon\) for every \(|\zeta|\le r\).
	Since \(F_{q,v}\) is convex on the underlying real plane and
	\(F_{q,v}(0)=0\), it follows that
	\(F_{q,v}(\zeta)+F_{q,v}(-\zeta)\ge0\).
	Applying the preceding upper bound to both \(\zeta\) and \(-\zeta\)
	therefore gives \(\|F_{q,v}\|_r\le\eps\).
	Since
	\(
	F_{q,v}=P_{q,v}+R_{q,v}
	\),
	we obtain from \eqref{eq:relative-remainder} that
	\[
	\|P_{q,v}\|_r
	\le
	\|F_{q,v}\|_r+\|R_{q,v}\|_r
	\le
	\eps+\frac14\|P_{q,v}\|_r.
	\]
	Hence \(\|P_{q,v}\|_r\le\frac43\eps\).
	Combining this with Lemma~\ref{lem:uniform-line-jets} gives
	\(ar^m\le\frac43\eps\).
	
	Choose \(\eps_0>0\) so small that
	\((4\eps_0/(3a))^{1/m}<r_0/2\).
	We claim that \(\tau(q,v,\eps)<r_0/2\) whenever \(0<\eps<\eps_0\).
	Indeed, if \(\tau(q,v,\eps)\ge r_0/2\), then
	\(ar^m\le\frac43\eps\) holds for every \(0<r<r_0/2\).
	Letting \(r\to(r_0/2)^-\), we obtain
	\[
	a\left(\frac{r_0}{2}\right)^m
	\le
	\frac43\eps
	<
	\frac43\eps_0,
	\]
	contrary to the choice of \(\eps_0\).  Thus
	\(\tau(q,v,\eps)<r_0/2\).  We may now let
	\(r\to\tau(q,v,\eps)^-\) in the bound \(ar^m\le\frac43\eps\), which gives
	\[
	\tau(q,v,\eps)
	\le
	\left(\frac{4\eps}{3a}\right)^{1/m}.
	\]
	The estimate is uniform in \(q\) and \(v\), and the uniform shrinking
	assertion follows.
	
	We next prove the directional dilation estimate.  By the uniform shrinking
	assertion just proved, after choosing
	\(0<\eps_*\le\eps_0\) sufficiently small, we have
	\(5\,\tau(q,v,\eps)<r_0\) for every \(q\in\Omega\cap U\),
	every unit vector \(v\), and every \(0<\eps<\eps_*\).
	
	Fix such \(q,v,\eps\), and take
\(
	0<r<\tau(q,v,\eps).
\)
	The polynomial bound obtained above gives
	\(\|P_{q,v}\|_r\le\frac43\eps\).
	Lemma~\ref{lem:radial-chebyshev} gives
	\[
	\|P_{q,v}\|_{5r}
	\le
	T_m(5)\|P_{q,v}\|_r.
	\]
	Since \(5r<r_0\), we may also apply
	\eqref{eq:relative-remainder} at radius \(5r\).  Therefore
	\[
	\begin{aligned}
		\|F_{q,v}\|_{5r}
		&\le
		\|P_{q,v}\|_{5r}+\|R_{q,v}\|_{5r}
		\le
		\frac54\|P_{q,v}\|_{5r}
		\le
		\frac54 T_m(5)\|P_{q,v}\|_r\\
		&\le
		\frac53 T_m(5)\eps
		<
		2T_m(5)\eps.
	\end{aligned}
	\]
	It follows that
	\[
	q+\zeta v\in\Omega_{q,2T_m(5)\eps},
	\qquad |\zeta|\le5r.
	\]
	Hence
	\(
	\tau\bigl(q,v,2T_m(5)\eps\bigr)\ge5r.
\)
	Letting \(r\to\tau(q,v,\eps)^-\), we obtain
	\[
	\tau\bigl(q,v,2T_m(5)\eps\bigr)
	\ge
	5\,\tau(q,v,\eps),
	\]
	when \(|v|=1\).  The result for arbitrary \(v\ne0\) follows from
	\eqref{eq:tau-homogeneity}.
\end{proof}
\begin{remark}
For later use, we fix the factor \(5\).  The same argument works for
any fixed \(A>1\), with \(2T_m(A)\) in place of \(2T_m(5)\).  The particular
choice of \(5\) is only for convenience.
\end{remark}
For later use, set
\begin{equation*}
	M(\eps)
	=
	\sup_{\substack{q\in\Omega\cap U\\ |v|=1}}
	\tau(q,v,\eps).
\end{equation*}
By Lemma~\ref{lem:finite-type-dilation}, we have \(M(\eps)\to0\)
as \(\eps\to0^+\). Moreover, \(M\) is nondecreasing.
\begin{corollary}\label{cor:local-directional-norm}
	After decreasing \(\eps_*\) if necessary, the inclusion
	\(\mathbb B_\eps(q)\subset B(0,M(\eps))\) holds for every
	\(q\in\Omega\cap U\) and \(0<\eps<\eps_*\).
	In particular, \(\mathbb B_\eps(q)\) is bounded and
	\(p_{q,\eps}\) is a norm on \(\C^n\).
\end{corollary}

\begin{proof}
	Since \(M(\eps)\to0\) as \(\eps\to0^+\), we may decrease \(\eps_*\) so that
	\(M(\eps)<\infty\) whenever \(0<\eps<\eps_*\).
	Fix \(q\in\Omega\cap U\), \(0<\eps<\eps_*\), and
	\(h\in\mathbb B_\eps(q)\), \(h\ne0\).  Since
	\[
	\{q+\lambda h:|\lambda|\le1\}
	\Subset
	\Omega_{q,\eps},
	\]
	we have \(\tau(q,h,\eps)>1\).
	Writing \(h=|h|v\), where \(|v|=1\), and using
	\eqref{eq:tau-homogeneity}, we obtain
	\(|h| < \tau(q,v,\eps) \le M(\eps)\).
	This proves the asserted inclusion.  Hence
	\(\mathbb B_\eps(q)\) is bounded, and
	Lemma~\ref{lem:directional-norm} implies that \(p_{q,\eps}\) is a norm.
\end{proof}

Shrinking \(U\) once more, we may also assume that
\(\delta(z)<\eps_*\) for every \(z\in\Omega\cap U\).
Thus Corollary~\ref{cor:local-directional-norm} applies with
\(\eps=\delta(z)\).

\subsection{Engulfing and local quasi-distance}
\label{subsec:center-engulfing-quasidistance}

The non-isotropic norm constructed above gives rise to a family of
non-isotropic neighborhoods.  In this section, we study the stability of these
neighborhoods under changes of the center and establish an engulfing
property.  These properties allow us to define a local quasi-distance
adapted to the finite-type geometry.

The following lemma shows that a small displacement measured by
\(p_{q,\eps}\) changes the defining function only by a comparable amount.

\begin{lemma}\label{lem:level-displacement}
Let \(q\in\Omega\cap U\) and \(0<\eps<\eps_*\).  If
\(p_{q,\eps}(h)=r<1\), then
\[
 |\rho(q+h)-\rho(q)|\le r\eps.
\]
\end{lemma}

\begin{proof}
The case \(h=0\) is immediate.  If \(h\ne0\), then
Corollary~\ref{cor:local-directional-norm} implies that \(r>0\).  Fix \(a\)
with \(r<a<1\).
Lemma~\ref{lem:directional-norm} gives
\[
\tau(q,h,\eps)=\frac1r>\frac1a,
\]
and hence
\[
 \rho(q+h/a)<\rho(q)+\eps,
 \qquad
 \rho(q-h/a)<\rho(q)+\eps.
\]

Let
\(
f(t)=\rho(q+th)-\rho(q).
\)
The function \(f\) is convex and \(f(0)=0\).  Since
\(1=a(1/a)+(1-a)0\),
convexity gives
\(f(1)\le af(1/a)<a\eps\).

On the other hand,
\[
0=\frac{a}{1+a}\left(-\frac1a\right)
+\frac1{1+a}\cdot1.
\]
Therefore, we have
\[
0=f(0)
\le
\frac{a}{1+a}f(-1/a)+\frac1{1+a}f(1)
<
\frac{a}{1+a}\eps+\frac1{1+a}f(1),
\]
which implies
\(
f(1)>-a\eps.
\)
Letting \(a\to r^+\) proves the claim.
\end{proof}

 \begin{definition}\label{def:non-isotropic-neighborhood}
 	For \(q\in\Omega\) and \(\eps>0\), define
 	\[
 		B_\eps(q)
 		=
 		\Omega\cap
 		\left\{
 		q+h:p_{q,\eps}(h)<\frac18
 		\right\}.
 	\]
 	For \(A>0\), we write
 \[
 		A B_\eps(q)
 	=
 	\Omega\cap
 	\left\{
 	q+h:p_{q,\eps}(h)<\frac{A}{8}
 	\right\}.
 \]
 
 \end{definition}
Thus \(AB_\eps(q)=\Omega\cap\bigl(q+(A/8)\mathbb B_\eps(q)\bigr)\):
the set \(\mathbb B_\eps(q)\) is dilated and translated before
being intersected with \(\Omega\).
In general, when \(0<A<1\), this need not agree with the dilation of
the already truncated set \(B_\eps(q)\) about \(q\).
The numerical factor \(1/8\) is not intrinsic.  It is chosen only for
convenience: other choices would also give an engulfing property with
different constants.  We do not optimize these constants here, but choose
\(1/8\) so that the subsequent estimates remain simple and transparent.

For an open convex set \(K\subset\C^n\), \(x\in K\), and \(v\ne0\), define
\[
\tau_K(x,v)
=
\sup\left\{
r>0:
x+\lambda v\in K
\text{ for every }|\lambda|<r
\right\}.
\]
With this notation, it follows that
\(\tau(q,v,\eps)
=
\tau_{\Omega_{q,\eps}}(q,v)\)
.

\begin{lemma}\label{lem:convex-center-shift}
Let \(K\subset\C^n\) be open and convex.  Let \(q,h\in\C^n\), set
\(w=q+h\), and suppose that there exists \(b\in(0,1)\) such that
\[
q+\frac{h}{b}\in K
\qquad\text{and}\qquad
q-\frac{h}{b}\in K.
\]
Then \(q,w\in K\), and, for every nonzero vector \(v\in\C^n\), we have
\[
(1-b)\tau_K(q,v)
\le
\tau_K(w,v)
\le
(1+b)\tau_K(q,v).
\]
\end{lemma}

\begin{proof}
The two points in the hypothesis satisfy
\[
q
=
\frac12\left(q+\frac{h}{b}\right)
+
\frac12\left(q-\frac{h}{b}\right)
\]
and
\[
w
=
\frac{1+b}{2}\left(q+\frac{h}{b}\right)
+
\frac{1-b}{2}\left(q-\frac{h}{b}\right).
\]
Since \(0<b<1\), the convexity of \(K\) implies that \(q,w\in K\).

Fix a nonzero vector \(v\in\C^n\) and a number
\(0<r<\tau_K(q,v)\).  For every \(\lambda\in\C\) with
\(|\lambda|<r\), we have
\[
w+(1-b)\lambda v
=
b\left(q+\frac{h}{b}\right)
+(1-b)(q+\lambda v)
\in K.
\]
It follows that
\(\tau_K(w,v)\ge(1-b)r\).
Since \(r<\tau_K(q,v)\) was arbitrary, we obtain
\(\tau_K(w,v)\ge(1-b)\tau_K(q,v)\).

Conversely, fix \(0<r<\tau_K(w,v)\).  For every
\(\lambda\in\C\) with \(|\lambda|<r\), we have
\[
q+\frac{\lambda}{1+b}v
=
\frac1{1+b}(w+\lambda v)
+
\frac{b}{1+b}\left(q-\frac{h}{b}\right)
\in K.
\]
Therefore, we have
\[
\tau_K(q,v)\ge\frac{r}{1+b}.
\]
Since \(r<\tau_K(w,v)\) was arbitrary, this gives
\(\tau_K(w,v)\le(1+b)\tau_K(q,v)\),
which proves the upper bound.
\end{proof}

\begin{lemma}\label{lem:center-stability}
    Let \(q\in\Omega\cap U\), \(0<\eps<\eps_*\), and
    \(w\in B_\eps(q)\).  The estimate
    \begin{equation*}
        \frac34p_{q,\eps}(v)
        \le
        p_{w,\eps}(v)
        \le
        \frac43p_{q,\eps}(v)
    \end{equation*}
    holds for every \(v\in\C^n\).
\end{lemma}
\begin{proof}
Fix \(v\in\C^n\).  The estimate is immediate if \(v=0\) or \(w=q\),
so assume that \(v\ne0\) and \(w\ne q\).  We first compare the
directional radii. Write \(w=q+h\) and set
\(r=p_{q,\eps}(h)<\frac18\).
Since \(p_{q,\eps}\) is a norm, we have \(r>0\).  By
Lemma~\ref{lem:level-displacement},
\(|\rho(w)-\rho(q)|\le r\eps\).
Consequently, we have the inclusions
\begin{equation}\label{eq:three-level-inclusions}
\Omega_{q,(1-r)\eps}
\subset
\Omega_{w,\eps}
\subset
\Omega_{q,(1+r)\eps}.
\end{equation}

For the lower bound, Lemma~\ref{lem:p-scale}, applied with
\(\lambda=(1-r)^{-1}\), gives
\[
\tau(q,h,(1-r)\eps)
\ge
(1-r)\tau(q,h,\eps)
=
\frac{1-r}{r}.
\]
Let \(b\in(r/(1-r),1)\).  Then
\[
\frac1b
<
\frac{1-r}{r}
\le
\tau(q,h,(1-r)\eps),
\]
and hence
\[
q+\frac{h}{b},\ q-\frac{h}{b}
\in
\Omega_{q,(1-r)\eps}.
\]
Applying Lemma~\ref{lem:convex-center-shift} to this level set and using
the first inclusion in \eqref{eq:three-level-inclusions}, we obtain
\[
\begin{aligned}
\tau(w,v,\eps)
&\ge
(1-b)\tau(q,v,(1-r)\eps)\\
&\ge
(1-b)(1-r)\tau(q,v,\eps).
\end{aligned}
\]
Letting \(b\to\bigl(r/(1-r)\bigr)^+\), we conclude that
\[
\tau(w,v,\eps)
\ge
(1-2r)\tau(q,v,\eps)
>
\frac34\tau(q,v,\eps).
\]

For the upper bound, the second inclusion in
\eqref{eq:three-level-inclusions} gives
\[
\tau(w,v,\eps)
\le
\tau_{\Omega_{q,(1+r)\eps}}(w,v).
\]
Moreover, we have
\[
\tau(q,h,(1+r)\eps)
\ge
\tau(q,h,\eps)
=
\frac1r.
\]
Let \(b\in(r,1)\).  Then
\[
\frac1b
<
\frac1r
\le
\tau(q,h,(1+r)\eps),
\]
so
\[
q+\frac{h}{b},\ q-\frac{h}{b}
\in
\Omega_{q,(1+r)\eps}.
\]
Lemma~\ref{lem:convex-center-shift}, applied to the convex set
\(\Omega_{q,(1+r)\eps}\), gives
\[
\tau_{\Omega_{q,(1+r)\eps}}(w,v)
\le
(1+b)\tau(q,v,(1+r)\eps).
\]
By Lemma~\ref{lem:p-scale}, we also have
\[
\tau(q,v,(1+r)\eps)
\le
(1+r)\tau(q,v,\eps).
\]
Letting \(b\to r^+\), we obtain
\[
\tau(w,v,\eps)
\le
(1+r)^2\tau(q,v,\eps)
<
\left(\frac98\right)^2\tau(q,v,\eps)
<
\frac43\tau(q,v,\eps).
\]
Taking reciprocals and applying Lemma~\ref{lem:directional-norm}
yields the asserted estimate.
\end{proof}

We now fix the constant \(Q_m=2T_m(5)\).
By \eqref{eq:chebyshev-formula}, we have
\[
Q_m
=
(5+\sqrt{24})^m+(5-\sqrt{24})^m
<
10^m,
\qquad m\ge2.
\]

\begin{proposition}\label{prop:scale-inclusion}
	If \(q\in\Omega\cap U\) and \(0<\eps<\eps_*\), then
	\(5B_\eps(q)\subset B_{Q_m\eps}(q)\).
\end{proposition}

\begin{proof}
	Let \(z=q+h\in5B_\eps(q)\).  Then
	\(p_{q,\eps}(h)<5/8\).
	The directional dilation estimate in
	Lemma~\ref{lem:finite-type-dilation} gives
	\[
	p_{q,Q_m\eps}(h)
	\le
	\frac15 p_{q,\eps}(h)
	<
	\frac18.
	\]
	Since \(z\in\Omega\), it follows that
	\(z\in B_{Q_m\eps}(q)\).
\end{proof}

We next prove the engulfing property.

\begin{proposition}\label{prop:engulfing}
	Suppose that \(q,w\in\Omega\cap U\) and \(0<\eps<\eps_*\).
	If
	\(B_\eps(q)\cap B_\eps(w)\ne\varnothing\),
	then
	\[
	B_\eps(w)\subset5B_\eps(q)\ \text{and}\
	B_\eps(q)\subset5B_\eps(w).
	\]
\end{proposition}

\begin{proof}
	Choose
	\(\zeta\in B_\eps(q)\cap B_\eps(w)\).
	Applying Lemma~\ref{lem:center-stability} to the pairs
	\((q,\zeta)\) and \((w,\zeta)\), we obtain
	\[
		\left(\frac34\right)^2p_{q,\eps}(v)
		\le
		p_{w,\eps}(v)
		\le
		\left(\frac43\right)^2p_{q,\eps}(v),
		\qquad v\in\C^n.
	\]
	Let \(u\in B_\eps(w)\).  The triangle inequality for
	\(p_{q,\eps}\) gives
	\[
	\begin{aligned}
		p_{q,\eps}(u-q)
		&\le
		p_{q,\eps}(u-w)
		+p_{q,\eps}(w-\zeta)
		+p_{q,\eps}(\zeta-q).
	\end{aligned}
	\]
	Since \(u,w\in B_\eps(w)\) and
	\(\zeta\in B_\eps(q)\cap B_\eps(w)\), the preceding two-center
	comparison gives
	\[
	\begin{aligned}
		p_{q,\eps}(u-q)
		<
		\left(\frac43\right)^2\frac18
		+
		\left(\frac43\right)^2\frac18
		+
		\frac18
		=
		\frac{41}{72}
		<
		\frac58.
	\end{aligned}
	\]
	Thus \(u\in 5B_\eps(q)\).  Since \(u\in B_\eps(w)\) was arbitrary,
	we obtain
	\[
	B_\eps(w)\subset 5B_\eps(q).
	\]
	Interchanging \(q\) and \(w\) gives the reverse inclusion.
\end{proof}
The preceding estimates allow us to use the sets \(B_\eps(q)\) to define a
local quasi-distance.

\begin{definition}\label{def:dB}
	Let \(q,w\in\Omega\cap U\).  If \(w\in B_\eps(q)\) for some
	\(0<\eps<\eps_*\), define
	\begin{equation*}
		d_B(q,w)
		=
		\inf\left\{
		0<\eps<\eps_*:
		w\in B_\eps(q)
		\right\}.
	\end{equation*}
	Otherwise, define
	\(d_B(q,w)=\eps_*\).
\end{definition}

Since the sets \(B_\eps(q)\) increase with \(\eps\), the definition immediately
implies that
\begin{equation*}
 d_B(q,w)<S<\eps_*
 \quad\Longrightarrow\quad
 w\in B_S(q).
\end{equation*}

The next lemma shows that \(d_B\) is compatible
with the Euclidean topology.

\begin{lemma}\label{lem:dB-separation}
    For \(x,y\in\Omega\cap U\), the equality \(d_B(x,y)=0\) holds
    if and only if \(x=y\).
    Moreover, for every fixed \(x\in\Omega\cap U\) and every sequence
    \((y_\nu)\) in \(\Omega\cap U\), the following equivalence holds:
    \[
    y_\nu\longrightarrow x
    \quad\Longleftrightarrow\quad
    d_B(x,y_\nu)\longrightarrow0.
    \]
\end{lemma}

\begin{proof}
    Since \(x\in B_\eps(x)\) for every \(0<\eps<\eps_*\), we have
    \(d_B(x,x)=0\).  Corollary~\ref{cor:local-directional-norm} and
    Definition~\ref{def:non-isotropic-neighborhood} give
    \begin{equation}\label{eq:B-euclidean-size}
        B_\eps(q)
        \subset
        B_{\mathrm{Euc}}
        \left(q,\frac18M(\eps)\right).
    \end{equation}

    Now let \(x\ne y\). Since \(M(\eps)\to0\) as \(\eps\to0^+\), there exists
    \(0<\eps_{x,y}<\eps_*\) such that
    \(M(\eps)/8<|x-y|\) whenever \(0<\eps<\eps_{x,y}\).  It follows from
    \eqref{eq:B-euclidean-size} that \(y\notin B_\eps(x)\) for such
    \(\eps\), and therefore \(d_B(x,y)\ge\eps_{x,y}>0\).

    It remains to prove the equivalence of the two convergences.  Let
    \((y_\nu)\) be a sequence in \(\Omega\cap U\).  Suppose first that
    \(y_\nu\to x\), and fix \(0<\eta<\eps_*\).  Since \(\rho\) is
    \(1\)-Lipschitz, we have
    \[
    \rho(x+\zeta v)-\rho(x)\le|\zeta|,
    \]
    for every unit vector \(v\in\C^n\).  Hence
    \(\tau(x,v,\eta)\ge\eta\), and therefore
    \(p_{x,\eta}(h)\le |h|/\eta\) for every \(h\in\C^n\).  For all
    sufficiently large \(\nu\), the following estimate holds:
    \[
    p_{x,\eta}(y_\nu-x)<\frac18,
    \]
    so \(y_\nu\in B_\eta(x)\) and
    \(d_B(x,y_\nu)\le\eta\).  Since \(\eta>0\) is arbitrary, we conclude
    that \(d_B(x,y_\nu)\to0\).

    Conversely, suppose that \(d_B(x,y_\nu)\to0\).  There is nothing to
    prove for indices for which \(y_\nu=x\).  For every sufficiently large
    \(\nu\) with \(y_\nu\ne x\), the first part of the lemma gives
    \[
    0<d_B(x,y_\nu)<\frac{\eps_*}{2}.
    \]
    By the definition of the infimum, we may choose \(\eps_\nu\) such that
    \[
    d_B(x,y_\nu)<\eps_\nu<2d_B(x,y_\nu)
    \qquad\text{and}\qquad
    y_\nu\in B_{\eps_\nu}(x).
    \]
    It follows from \eqref{eq:B-euclidean-size} that
    \(|x-y_\nu|<\frac18M(\eps_\nu)\).
    As \(\eps_\nu\to0\), we have \(M(\eps_\nu)\to0\), and hence
    \(|x-y_\nu|\to0\). Thus \(y_\nu\to x\).
\end{proof}
The following theorem shows that \(d_B\) is indeed a quasi-distance on
\(\Omega\cap U\).
\begin{theorem}\label{thm:dB-quasi}
For \(q,w\in\Omega\cap U\), the following quasi-symmetry
estimate holds:
\begin{equation*}
d_B(w,q)
\le
Q_m d_B(q,w).
\end{equation*}
Moreover, for \(x,y,z\in\Omega\cap U\), the following
quasi-triangle inequality holds:
\begin{equation*}
d_B(x,z)
\le
Q_m\bigl(d_B(x,y)+d_B(y,z)\bigr).
\end{equation*}
\end{theorem}

\begin{proof}
	We first prove the quasi-symmetry estimate.  If
	\(Q_m d_B(q,w)\ge\eps_*\), then
	\[d_B(w,q)\le\eps_*\le Q_m d_B(q,w).\]
	Suppose that \(Q_m d_B(q,w)<\eps_*\).  By the definition of \(d_B\),
	we may choose \(\eps>d_B(q,w)\), arbitrarily close to
	\(d_B(q,w)\), such that \(w\in B_\eps(q)\) and
	\(Q_m\eps<\eps_*\).  Since \(w\in B_\eps(w)\), the two sets
	\(B_\eps(q)\) and \(B_\eps(w)\) intersect.  Proposition~\ref{prop:engulfing}
	therefore gives
	\(q\in 5B_\eps(w)\).
	By Proposition~\ref{prop:scale-inclusion},
	\(q\in B_{Q_m\eps}(w)\), and hence
	\(d_B(w,q)\le Q_m\eps\).
	Letting \(\eps\to d_B(q,w)^+\) proves the quasi-symmetry estimate.
	
	We now prove the quasi-triangle estimate.  Set
	\(D=d_B(x,y)+d_B(y,z)\).
	If \(Q_mD\ge\eps_*\), then
	\[d_B(x,z)\le\eps_*\le Q_mD.\]
	
	Suppose that \(Q_mD<\eps_*\).  Choose
	\(a>d_B(x,y)\) and \(b>d_B(y,z)\), arbitrarily close to these
	quantities, such that
	\[
	y\in B_a(x),\qquad z\in B_b(y),
	\qquad
	Q_m(a+b)<\eps_*.
	\]
	Set \(E=a+b\).  Since the sets \(B_\eps(q)\) increase with \(\eps\),
	we have \(y\in B_E(x)\) and \(z\in B_E(y)\).  Thus
	\(B_E(x)\cap B_E(y)\ne\varnothing\), and
	Proposition~\ref{prop:engulfing} gives
	\(B_E(y)\subset5B_E(x)\).
	In particular, \(z\in5B_E(x)\).  Proposition~\ref{prop:scale-inclusion}
	then gives \(z\in B_{Q_mE}(x)\), so
	\(d_B(x,z)\le Q_mE=Q_m(a+b)\).
	Letting \(a\to d_B(x,y)^+\) and
	\(b\to d_B(y,z)^+\) proves the quasi-triangle estimate.
\end{proof}

\section{Kobayashi distance estimates and local Gromov hyperbolicity}
\label{sec:kobayashi-comparison}

We first establish the distance comparison in
Theorem~\ref{thm:two-sided-kobayashi} and then deduce the local Gromov
hyperbolicity estimate in Theorem~\ref{thm:main}.

The constructions above are local near the boundary point \(p\), while the
Kobayashi distance on \(\Omega\) is a global object.  To compare the local
geometry with \(K_\Omega\), we use the following quantitative localization
result of Nikolov--Thomas.

Recall that a domain \(D\subset\C^n\) is called \emph{\(\C\)-convex} if
its intersection with every complex affine line is either empty or
connected and simply connected.  We say that \(D\) is
\emph{\(\C\)-convexifiable near} \(p\in\partial D\) if there exist a
neighborhood \(U\) of \(p\) and a biholomorphic map \(\Phi\), defined
on a neighborhood of \(\overline{D\cap U}\), such that
\(\Phi(D\cap U)\) is \(\C\)-convex.
\begin{theorem}[{\cite[Theorem~1.1]{NikolovThomas}}]
	\label{thm:nikolov-thomas-localization}
	Let \(\Omega\subset\C^n\) be a domain and let \(p\in\partial \Omega\).
	Assume that \(\Omega\) is \(\C\)-convexifiable near \(p\) and that \(p\) is of
	type at most \(s\ge2\).  Then there exists a neighborhood \(U_0\) of \(p\)
	such that, whenever
	\(
	V_0\Subset U\subset U_0
\)
	are neighborhoods of \(p\) and \(\Omega\cap U\) is connected, there exists
	\(C>0\) such that
	\begin{equation*}
		K_\Omega(z,w)
		\le
		K_{\Omega\cap U}(z,w)
		\le
		K_\Omega(z,w)+C|z-w|^{1/s},
	\end{equation*}
	for all \(z,w\in \Omega\cap V_0\).
\end{theorem}

We now apply the preceding theorem in our setting.  Since \(\Omega\) is
convex, it is \(\C\)-convexifiable near \(p\).
Moreover, by
assumption, the D'Angelo type of \(\partial\Omega\) is at most \(m\) near
\(p\).  Hence Theorem~\ref{thm:nikolov-thomas-localization} applies with
\(s=m\), and provides a neighborhood \(U_0\) of \(p\).
Since all the preceding constructions are local near \(p\), after shrinking
\(U\) if necessary, we may assume that \(U\) is a bounded convex
neighborhood of \(p\), that
\(
\overline{U}\subset U_0,
\)
and that all the estimates obtained above remain valid on \(\Omega\cap U\).
In particular, \(\Omega\cap U\) is convex and hence connected.

Fix a neighborhood \(V_0\Subset U\) of \(p\).  Applying
Theorem~\ref{thm:nikolov-thomas-localization}, we obtain a constant \(C>0\)
such that
\[
0
\le
K_{\Omega\cap U}(z,w)-K_\Omega(z,w)
\le
C|z-w|^{1/m},
\qquad
z,w\in\Omega\cap V_0.
\]
Choose a smaller neighborhood \(V\Subset V_0\) of \(p\) such that
\(C\operatorname{diam}(V)^{1/m}\le1\).
We will shrink \(V\) further when necessary, without changing the notation.
Then
\begin{equation}\label{eq:localization}
	K_\Omega(z,w)
	\le
	K_{\Omega\cap U}(z,w)
	\le
	K_\Omega(z,w)+1,
	\qquad
	z,w\in\Omega\cap V.
\end{equation}
Thus it is enough to establish the local comparison in \(\Omega\cap U\);
the corresponding estimate for \(K_\Omega\) then follows with an additive
error of at most \(1\).

For the remainder of the section, put \(X=\Omega\cap U\).  Choose
\(S_0>0\) so small that \(2S_0<\eps_*\).
In \eqref{eq:S0-upper-conditions} below, \(S_0\) will be decreased further,
if necessary, without changing the notation.
\begin{lemma}\label{lem:depth-stability}
Let \(z\in X\), \(0<\eps<\eps_*\), and
\(w\in B_\eps(z)\).  Then
\begin{equation*}
 |\delta(w)-\delta(z)|<\frac{\eps}{8}.
\end{equation*}
In particular, if \(\delta(z)<S_0\) and
\(\eps=\delta(z)/4\), then
\[
 \frac{31}{32}\delta(z)
 <
 \delta(w)
 <
 \frac{33}{32}\delta(z).
\]
\end{lemma}

\begin{proof}
Write \(w=z+h\).  By Definition~\ref{def:non-isotropic-neighborhood},
\(
 p_{z,\eps}(h)<\frac18.
\)
Lemma~\ref{lem:level-displacement} therefore gives
\[
 |\rho(w)-\rho(z)|
 \le p_{z,\eps}(h)\eps
 <\frac{\eps}{8}.
\]
Since \(z,w\in\Omega\), we have \(\delta=-\rho\) at both points, and
hence
\[
 |\delta(w)-\delta(z)|
 =
 |\rho(w)-\rho(z)|
 <\frac{\eps}{8}.
\]
Taking \(\eps=\delta(z)/4\) gives the last two inequalities.
\end{proof}

\begin{lemma}\label{lem:crossing-cost}
	Let \(z\in X\) satisfy \(\delta(z)<S_0\), and let
	\(\gamma:[0,1]\to\Omega\cap U\) be absolutely continuous with
	\(\gamma(0)=z\).  Suppose that \(t_1\in(0,1]\) is the first exit time of \(\gamma\) from
	\[
	\frac12 B_{\delta(z)/4}(z);
	\]
	that is, \(\gamma(t)\) belongs to this set for \(0\le t<t_1\), while
	\(\gamma(t_1)\) lies on its relative boundary in \(\Omega\).
	Then
	\[
	L_k\bigl(\gamma|_{[0,t_1]}\bigr)\ge\frac{1}{176}.
	\]
\end{lemma}

\begin{proof}
	For \(0\le t<t_1\), the curve satisfies
	\[
	\gamma(t)
	\in
	\frac12 B_{\delta(z)/4}(z)
	\subset
	B_{\delta(z)}(z).
	\]
	Hence Lemma~\ref{lem:center-stability}, applied with
	\(\eps=\delta(z)\), gives
	\[
		p_{\gamma(t),\delta(z)}(v)
		\ge
		\frac34 p_{z,\delta(z)}(v),
		\qquad v\in\C^n.
	\]
	By Lemma~\ref{lem:depth-stability}, we have
	\[
	\frac{31}{32}\delta(z)
	<
	\delta(\gamma(t))
	<
	\frac{33}{32}\delta(z).
	\]
Since \(\delta(z)<S_0\) and \(2S_0<\eps_*\), we have
\[
\delta(\gamma(t))
<
\frac{33}{32}\delta(z)
<
\eps_*,
\]
so Lemma~\ref{lem:p-scale} applies.  Using the monotonicity of \(p_{q,\eps}\) in \(\eps\) and
	Lemma~\ref{lem:p-scale}, we obtain
	\[
	\begin{aligned}
		p_{\gamma(t),\delta(\gamma(t))}(v)
		\ge
		p_{\gamma(t),(33/32)\delta(z)}(v)
		\ge
		\frac{32}{33}\,
		p_{\gamma(t),\delta(z)}(v).
	\end{aligned}
	\]
	Combining these estimates gives
	\[
		p_{\gamma(t),\delta(\gamma(t))}(v)
		\ge
		\frac8{11}p_{z,\delta(z)}(v).
	\]
For every \(z\in\Omega\), we have
\[
\Omega_{z,\delta(z)}
=
\{w\in\C^n:\rho(w)<0\}
=
\Omega.
\]
Therefore, for every nonzero vector \(v\in\C^n\), we have
\[
p_{z,\delta(z)}(v)
=
\frac{|v|}{\delta_\Omega(z;v)}.
\]
Lemma~\ref{lem:graham} now gives
\begin{equation*}
    \frac12 p_{z,\delta(z)}(v)
    \le
    k_\Omega(z;v)
    \le
    p_{z,\delta(z)}(v),
    \qquad z\in\Omega,\quad v\in\C^n\setminus\{0\}.
\end{equation*}
Both sides of the preceding inequality vanish when \(v=0\).
Thus the same inequalities hold for every \(v\in\C^n\).
Using the infinitesimal comparison and the preceding moving-norm bound,
we obtain
\[
\begin{aligned}
	L_k\bigl(\gamma|_{[0,t_1]}\bigr)
	&=
	\int_0^{t_1}
	k_\Omega(\gamma(t);\gamma'(t))\,\dd t
	\ge
	\frac12
	\int_0^{t_1}
	p_{\gamma(t),\delta(\gamma(t))}
	(\gamma'(t))\,\dd t\\
	&\ge
	\frac4{11}
	\int_0^{t_1}
	p_{z,\delta(z)}(\gamma'(t))\,\dd t.
\end{aligned}
\]
Since \(\gamma\) is absolutely continuous, we have
\[
\gamma(t_1)-z
=
\int_0^{t_1}\gamma'(t)\,\dd t.
\]
Since
\(
\delta(z)<S_0<\eps_*,
\)
Corollary~\ref{cor:local-directional-norm} shows that
\(p_{z,\delta(z)}\) is a norm.  Hence, using the triangle inequality,
\[
p_{z,\delta(z)}\bigl(\gamma(t_1)-z\bigr)
\le
\int_0^{t_1}
p_{z,\delta(z)}(\gamma'(t))\,\dd t.
\]
Therefore, we obtain
\[
L_k\bigl(\gamma|_{[0,t_1]}\bigr)
\ge
\frac4{11}
p_{z,\delta(z)}\bigl(\gamma(t_1)-z\bigr).
\]

	Since \(\gamma(t_1)\) lies on the boundary, relative to \(\Omega\), of
	\(\frac12 B_{\delta(z)/4}(z)\),
	Definition~\ref{def:non-isotropic-neighborhood} gives
	\[
	p_{z,\delta(z)/4}\bigl(\gamma(t_1)-z\bigr)
	=
	\frac1{16}.
	\]
	Lemma~\ref{lem:p-scale} with \(\lambda=4\) then gives
	\[
	p_{z,\delta(z)}\bigl(\gamma(t_1)-z\bigr)
	\ge
	\frac14
	p_{z,\delta(z)/4}\bigl(\gamma(t_1)-z\bigr)
	=
	\frac1{64}.
	\]
	Consequently, we have
	\[
	L_k\bigl(\gamma|_{[0,t_1]}\bigr)
	\ge
	\frac4{11}\cdot\frac1{64}
	=
	\frac1{176}.
	\]
\end{proof}

Set $\kappa_0=1/176$.

\subsection{A finite-chain quasi-distance}
We first symmetrize the local quasi-distance.  For
\(x,y\in X\), set
\[
 \widehat d_B(x,y)
 =
 \max\{d_B(x,y),d_B(y,x)\}.
\]
By definition, \(\widehat d_B\) is symmetric, and
Lemma~\ref{lem:dB-separation} shows that it separates points. We shall
use \(\widehat d_B\) below as a symmetric quasi-distance. Although
Lemma~\ref{lem:dB-separation} establishes the qualitative compatibility
of \(d_B\) with the Euclidean topology, the construction below requires
uniform quantitative control. We therefore begin with the following
technical lemma.
\begin{lemma}\label{lem:dB-euclidean-control}
	For \(q,w\in\Omega\cap U\), the following estimates hold.
	
	If
	\(8|q-w|<\eps_*\),
	then
	\begin{equation*}
		\widehat d_B(q,w)
		\le
		8|q-w|.
	\end{equation*}
	Conversely, if \(0<\eps<\eps_*\) and
	\(d_B(q,w)<\eps\),
	then
	\begin{equation*}
		|q-w|
		<
		\frac18 M(\eps).
	\end{equation*}
\end{lemma}

\begin{proof}The second estimate follows immediately from \eqref{eq:B-euclidean-size}, which was already
used in the proof of Lemma \ref{lem:dB-separation}. It remains to prove the first estimate.

	Suppose first that \(8|q-w|<\eps_*\).  If \(q=w\), then
	the first estimate is immediate.  Assume that \(q\ne w\),
	and fix
	\(8|q-w|<\eps<\eps_*\).
	Set \(v=(w-q)/|w-q|\in\C^n\).
	Since \(\rho\) is \(1\)-Lipschitz and \(|v|=1\), we have
	\[
	\rho(q+\zeta v)-\rho(q)
	\le
	|\zeta|
	<
	\eps
	\qquad\text{whenever }|\zeta|<\eps.
	\]
	It follows that
	\(\tau(q,v,\eps)\ge\eps\).
	Since \(w-q=|w-q|v\), Lemma~\ref{lem:directional-norm} gives
	\[
	p_{q,\eps}(w-q)
	\le
	\frac{|w-q|}{\eps}
	<
	\frac18.
	\]
	Thus \(w\in B_\eps(q)\), and therefore
	\(d_B(q,w)\le\eps\).
	Letting \(\eps\to(8|q-w|)^+\) gives
	\(d_B(q,w)\le8|q-w|\).
	Interchanging \(q\) and \(w\) gives the same estimate for
	\(d_B(w,q)\), and hence the first estimate follows.
	
	Now suppose that \(d_B(q,w)<\eps<\eps_*\).  Choose
	\(d_B(q,w)<\eps'<\eps\)
	such that \(w\in B_{\eps'}(q)\).  By
	\eqref{eq:B-euclidean-size},
	\[
	|q-w|
	<
	\frac18M(\eps').
	\]
	Since \(M\) is nondecreasing,
	\(M(\eps')\le M(\eps)\),
	which proves the second estimate.
\end{proof}

The definition of \(\widehat d_B\) and
Theorem~\ref{thm:dB-quasi} give
\begin{equation}\label{eq:dB-hat-comparison}
 d_B(x,y)
 \le
 \widehat d_B(x,y)
 \le
 Q_m d_B(x,y),
\end{equation}
where \(Q_m=2T_m(5)<10^m\).  The same theorem also gives
\begin{equation}\label{eq:dB-hat-max-triangle}
 \widehat d_B(x,z)
 \le
 Q_m\bigl(
   \widehat d_B(x,y)+\widehat d_B(y,z)
 \bigr)
 \le
 2Q_m
 \max\{\widehat d_B(x,y),\widehat d_B(y,z)\}.
\end{equation}
Since \(Q_m<10^m\) and \(m\ge2\), we have
\((2Q_m)^{1/(4m)}<2\).

Put \(\mathfrak q(x,y)=\widehat d_B(x,y)^{1/(4m)}\).
It follows from \eqref{eq:dB-hat-max-triangle}  that
\[
	\mathfrak q(x,z)\le2\max\{\mathfrak q(x,y),\mathfrak q(y,z)\}.
\]
Also, Lemma~\ref{lem:dB-separation} shows that \(\mathfrak q(x,y)=0\) only when
\(x=y\).
We now obtain a metric from this quasi-metric.  By a result of Frink
\cite{Frink1937} (or see \cite[Theorem~1.2]{Schroeder}),
the preceding quasi-triangle inequality gives a metric
\(d_{\mathfrak q}\) such that
\[
\frac14\mathfrak q(x,y)\le d_{\mathfrak q}(x,y)\le \mathfrak q(x,y).
\]
This metric yields the following chain estimate.

\begin{lemma}\label{lem:power-chain}
Put \(\alpha=1/(4m)\). Every finite sequence \(x_0,\ldots,x_N\) in
\(X\) satisfies
\begin{equation}\label{eq:symmetric-power-chain}
\widehat d_B(x_0,x_N)^\alpha
\le4\sum_{j=1}^N\widehat d_B(x_{j-1},x_j)^\alpha.
\end{equation}
In particular,
\[
d_B(x_0,x_N)^\alpha
\le\widehat d_B(x_0,x_N)^\alpha
\le8\sum_{j=1}^N d_B(x_{j-1},x_j)^\alpha.
\]
\end{lemma}
\begin{proof}
The metric comparison above and the triangle inequality give
\[
\begin{aligned}
\widehat d_B(x_0,x_N)^\alpha
&\le4d_{\mathfrak q}(x_0,x_N)\\
&\le4\sum_{j=1}^N d_{\mathfrak q}(x_{j-1},x_j)
\le4\sum_{j=1}^N\widehat d_B(x_{j-1},x_j)^\alpha.
\end{aligned}
\]
The second estimate follows from
\(\widehat d_B\le Q_m d_B\) and \(Q_m^\alpha<2\).
\end{proof}
We now use \(\widehat d_B\) to define a finite-chain quasi-distance.
Rather than measuring \(x\) and \(y\) only by the single value
\(\widehat d_B(x,y)\), we allow them to be joined by a chain with at most
\(4m\) links and measure the chain by the size of its largest link.
We then minimize this quantity over all such chains.

\begin{definition}\label{def:finite-chain-quasidistance}
	For \(x,y\in X\), define the following infimum over all chains
	\(x=x_0,\ldots,x_N=y\) in \(X\) with \(1\le N\le4m\):
	\begin{equation*}
		r_m(x,y)=\inf\left\{\max_{1\le j\le N}\widehat d_B(x_{j-1},x_j)\right\}.
	\end{equation*}
\end{definition}
Thus \(r_m(x,y)\) is the smallest possible value of the largest link
among all chains from \(x\) to \(y\) with at most \(4m\) links.
The restriction on the number of links is essential.  Since
\(\widehat d_B\) is locally compatible with the Euclidean topology,
allowing arbitrarily many links would permit a continuous path to be
subdivided into pieces of arbitrarily small \(\widehat d_B\)-size, so
the resulting infimum could vanish. 
The choice \(4m\) is matched to the exponent \(1/(4m)\) in
Lemma~\ref{lem:power-chain}.
The value \(4m\) is merely a convenient choice; it is neither
canonical nor claimed to be optimal.
On the other hand, the one-link chain gives
\(r_m(x,y)\le \widehat d_B(x,y)\).

Finally, define
\[
R_m(x,y)
=
r_m(x,y)+\max\{\delta(x),\delta(y)\},
\qquad
g_m(x,y)
=
\log\frac{R_m(x,y)}
{\sqrt{\delta(x)\delta(y)}}.
\]

The next lemma shows that \(r_m\)  is a quasi-distance on \(X\).

\begin{lemma}\label{lem:chain-estimates}
For \(x,y,z\in X\), one has
\[
\widehat d_B(x,y)^{1/(4m)}
\le16m\,r_m(x,y)^{1/(4m)}.
\]
In particular, \(r_m\) is symmetric and \(r_m(x,y)=0\) only when \(x=y\).
Moreover,
\[
r_m(x,z)\le2Q_m\max\{r_m(x,y),r_m(y,z)\},
\]
and
\[
R_m(x,z)\le3Q_m\max\{R_m(x,y),R_m(y,z)\}.
\]
\end{lemma}
\begin{proof}
If \(x=x_0,\ldots,x_N=y\), \(N\le4m\), is a chain whose
\(\widehat d_B\)-links are at most \(r\), then
\eqref{eq:symmetric-power-chain} gives
\[
\widehat d_B(x,y)^{1/(4m)}
\le4N r^{1/(4m)}\le16m r^{1/(4m)}.
\]
Taking the infimum proves the comparison and separation; reversing
chains proves symmetry.

For the quasi-triangle estimate, concatenate chains from \(x\) to \(y\)
and from \(y\) to \(z\), each with at most \(4m\) links, whose largest
links are less than
\[
r=\max\{r_m(x,y),r_m(y,z)\}+\eta,\qquad \eta>0.
\]
The concatenation has at most \(8m\) links. Replace consecutive pairs
of links by single links, leaving the last link unchanged if necessary.
By \eqref{eq:dB-hat-max-triangle}, each new link is at most \(2Q_mr\),
and the resulting chain has at most \(4m\) links. Letting \(\eta\to0^+\)
gives the asserted inequality for \(r_m\).
Finally,
\[
\begin{aligned}
R_m(x,z)
&\le2Q_m\max\{r_m(x,y),r_m(y,z)\}
+\max\{\delta(x),\delta(z)\}\\
&\le(2Q_m+1)\max\{R_m(x,y),R_m(y,z)\}\\
&\le3Q_m\max\{R_m(x,y),R_m(y,z)\}.
\end{aligned}
\]
\end{proof}

\begin{remark}
\label{rem:mcneal-comparison}
Let \(d_{M}\) denote McNeal's local quasi-distance defined
using the adapted polydiscs \(P(x,\eps)\); see
\cite[Proposition~5.1]{McNealBergman}.  For each fixed smoothly bounded
convex finite-type domain and a sufficiently small boundary neighborhood,
there exists \(A\ge1\) such that
\[
A^{-1}d_{M}(x,y)
\le d_B(x,y)
\le A d_{M}(x,y),
\qquad x,y\in\Omega\cap U.
\]

Indeed, let \(\mu_{x,\eps}\) be the Minkowski functional of the
centered polydisc \(P(x,\eps)-x\).  McNeal's extremal-basis estimate
\cite[Proposition~2.2]{McNealBergman}, together with the equivalence of
the \(\ell^1\)- and \(\ell^\infty\)-norms in the fixed dimension, gives
constants \(C_0\ge1\) and \(\eps_0>0\) such that
\[
C_0^{-1}\mu_{x,\eps}(h)
\le p_{x,\eps}(h)
\le C_0\mu_{x,\eps}(h),
\]
whenever \(x\in\Omega\cap U\), \(0<\eps<\eps_0\), and \(h\in\C^n\).
Choose an integer \(k\ge1\) with \(5^k>8C_0\), and set
\(A=Q_m^k\).  After decreasing \(\eps_0\) so that
\(A\eps_0<\eps_*\), iterating the directional dilation estimate in
Lemma~\ref{lem:finite-type-dilation} gives
\[
p_{x,A\eps}(h)\le5^{-k}p_{x,\eps}(h),
\qquad
p_{x,\eps}(h)\le5^{-k}p_{x,A^{-1}\eps}(h).
\]
Together with the preceding gauge comparison, these inequalities imply
\[
\Omega\cap P(x,A^{-1}\eps)
\subset
B_\eps(x)
\subset
\Omega\cap P(x,A\eps),
\qquad 0<\eps<\eps_0.
\]
The definitions of \(d_{M}\) and \(d_B\) now give the stated
comparison.

It follows from \eqref{eq:dB-hat-comparison}, the one-link bound from
Definition~\ref{def:finite-chain-quasidistance}, the first estimate in
Lemma~\ref{lem:chain-estimates}, and the preceding comparison that,
for each fixed domain and boundary neighborhood, the four local
quasi-distances
\[
d_{M},\qquad d_B,\qquad \widehat d_B,\qquad r_m
\]
are mutually Lipschitz equivalent.  The corresponding comparison
constants may depend on the domain, the chosen neighborhood, and the
dimension.
\end{remark}

\subsection{Lower bound for the Kobayashi distance}

We next derive a lower bound for the length of a curve in terms of
\(r_m(x,y)\) and the boundary distances of its endpoints.  Let
\(\gamma:[0,1]\to X\) be absolutely continuous, with
\(\gamma(0)=x\) and \(\gamma(1)=y\), and set
\(H=\max_{0\le t\le1}\delta(\gamma(t))\).
Let \(t_x\) and \(t_y\) be, respectively, the first and the last times
at which
\(\delta(\gamma(t))=H\).
Applying Lemma~\ref{lem:normal-lower} to
\(\gamma|_{[0,t_x]}\) and
\(\gamma|_{[t_y,1]}\), we obtain
\[
	L_k(\gamma|_{[0,t_x]})
	+
	L_k(\gamma|_{[t_y,1]})
	\ge
	\frac12\log\frac{H}{\delta(x)}
	+
	\frac12\log\frac{H}{\delta(y)}=
	\log\frac{H}{\sqrt{\delta(x)\delta(y)}}.
\]
This gives the contribution coming from the change in the distance to
the boundary from \(\delta(x)\) and \(\delta(y)\) up to \(H\).

We now extract additional information from the tangential motion of the
curve.  Divide the initial and final portions into pieces on which the
boundary distance changes by a factor of \(2\).  On each such piece, we
follow the curve and record successive first exits from sets of the form
\[
\frac12 B_{\delta(q)/4}(q).
\]
By Lemma~\ref{lem:crossing-cost}, every completed first exit contributes
at least \(\kappa_0\) to the Kobayashi length.

The two estimates come from the same portions of the curve and therefore
cannot be added directly.  To combine them, we assign to each recorded
link a weight determined by the boundary distance at which the link is
created.  The following elementary lemma partitions the links, in their
original order, into at most \(4m\) consecutive groups, with controlled
total weight in each group.

\begin{lemma}\label{lem:weight-grouping}
	Let \(\omega_1,\ldots,\omega_L\in(0,1]\), and set
	\(\mathsf W=\sum_{\ell=1}^L\omega_\ell\).
	For every integer \(J\ge1\), the sequence
	\(\omega_1,\ldots,\omega_L\), in its given order, can be divided into
	at most \(J\) consecutive nonempty groups, each having total weight at
	most \(\mathsf W/J+1\).
\end{lemma}

\begin{proof}
	Set \(\theta=\mathsf W/J\).
	Starting from \(\omega_1\), place consecutive weights into the current
	group until its total weight first reaches or exceeds \(\theta\).  Close
	the group at that point and begin a new one with the next weight.  If
	the total weight of the remaining terms is less than \(\theta\), place
	all of them into one final group.
	Whenever a group is closed by reaching \(\theta\), its total weight
	before the last term is added is less than \(\theta\).  Since every
	weight is at most \(1\), the total weight of such a group is therefore
	less than \(\theta+1=\mathsf W/J+1\) and is at least \(\theta\).
	There can be at most \(J\) groups closed in this way.  Indeed, \(J\)
	such groups already have total weight at least
	\(J\theta=\mathsf W\),
	which is the total weight of the whole sequence.  If fewer than \(J\)
	groups are closed, the remaining terms form at most one final group,
	whose total weight is less than \(\theta\).  Hence the sequence is
	divided into at most \(J\) consecutive nonempty groups, and every group
	has total weight at most \(\mathsf W/J+1\).
\end{proof}
\begin{lemma}\label{lem:weighted-exits}
	Let \(\gamma:[0,1]\to X\) be absolutely continuous, with
	\(\gamma(0)=x\) and \(\gamma(1)=y\).  Put
	\(H=\max_{0\le t\le1}\delta(\gamma(t))\).
	If \(H<S_0\), then
	\begin{equation*}
		L_k(\gamma)
		\ge
		\log\frac{H}{\sqrt{\delta(x)\delta(y)}}
		+
		\frac{\kappa_0m}{2}
		\left(\frac{r_m(x,y)}{H}\right)^{1/(4m)}
		-5m.
	\end{equation*}
\end{lemma}

\begin{proof}
	Let
	\[
	t_x=\min\{t\in[0,1]:\delta(\gamma(t))=H\},
	\qquad
	t_y=\max\{t\in[0,1]:\delta(\gamma(t))=H\}.
	\]
	Then \(t_x\le t_y\).
	We first decompose \(\gamma|_{[0,t_x]}\) according to dyadic levels of
	the boundary distance.  Choose \(k_x\ge0\) such that
	\(H2^{-k_x}\le\delta(x)<H2^{-k_x+1}\),
	and set
	\[
	h_0^x=\delta(x),
	\qquad
	h_j^x=H2^{-k_x+j},
	\quad 1\le j\le k_x.
	\]
	Then \(h_{k_x}^x=H\) and
	\[
	h_j^x\le2h_{j-1}^x,
	\qquad 1\le j\le k_x.
	\]
	If \(k_x=0\), then \(\delta(x)=H\) and \(t_x=0\), so there is nothing
	to decompose.  Suppose \(k_x\ge1\).  Set \(\theta_0=0\), and for
	\(1\le j\le k_x\), let \(\theta_j\) be the first time after
	\(\theta_{j-1}\) such that
	\(\delta(\gamma(\theta_j))=h_j^x\).
	By continuity of \(\delta\circ\gamma\) and the definition of \(t_x\),
	these times are well defined and satisfy
	\[
	0=\theta_0<\theta_1<\cdots<\theta_{k_x}=t_x.
	\]
	For
	\[
	\gamma_j^x=\gamma|_{[\theta_{j-1},\theta_j]},
	\]
	the boundary distances of the endpoints are \(h_{j-1}^x\) and
	\(h_j^x\), and
	\[
	\delta(\gamma(t))\le h_j^x,
	\qquad
	\theta_{j-1}\le t\le\theta_j.
	\]
	We treat \(\gamma|_{[t_y,1]}\) in the same way, after reversing its
	orientation.  Define
	\[
	\gamma^y(\theta)=\gamma(1-\theta),
	\qquad
	0\le\theta\le1-t_y.
	\]
	Applying the same construction to \(\gamma^y\), we obtain subcurves
	\[
	\gamma_j^y,
	\qquad
	1\le j\le k_y,
	\]
	whose endpoint boundary distances are \(h_{j-1}^y\) and \(h_j^y\), and
	along which the boundary distance is at most \(h_j^y\).
	
	Thus, for \(\ast\in\{x,y\}\), each subcurve
	\(\gamma_j^\ast\) is oriented from boundary distance
	\(h_{j-1}^\ast\) to \(h_j^\ast\), and
	\(\delta\le h_j^\ast\)
	along the whole subcurve.  Hence Lemma~\ref{lem:normal-lower} gives
	\begin{equation}\label{eq:normal-cost-one-level}
		L_k(\gamma_j^\ast)
		\ge
		v_j^\ast
		:=
		\frac12\log\frac{h_j^\ast}{h_{j-1}^\ast},
		\qquad
		0\le v_j^\ast\le\frac12\log2.
	\end{equation}
	We next record first exits along each subcurve.  Fix
	\(\gamma_j^\ast\), with \(\ast\in\{x,y\}\), and write it as
	\(
	\sigma:[a,b]\to X.
\)
	Set \(s_0=a\).  Suppose \(s_\ell<b\), and put
	\(
	u_\ell=\sigma(s_\ell).
	\)
	If
	\[
	\sigma(t)\in
	\frac12 B_{\delta(u_\ell)/4}(u_\ell)
	\qquad
	\text{for every }s_\ell\le t<b,
	\]
	record
\(
	\bigl(u_\ell,\sigma(b)\bigr)
\)
	as the final link and stop.  Otherwise, let \(s_{\ell+1}<b\) be the
	first time at which \(\sigma\) leaves
	\[
	\frac12 B_{\delta(u_\ell)/4}(u_\ell).
	\]
	Record the link
	\(
	\bigl(u_\ell,\sigma(s_{\ell+1})\bigr)
	\)
	and continue from \(s_{\ell+1}\).
Suppose first that the construction produces infinitely many first exits,
with corresponding times
\(a=s_0<s_1<s_2<\cdots<b\).
For \(\ell\ge1\), set
\[
\sigma_\ell
=
\sigma|_{[s_{\ell-1},s_\ell]}.
\]
  Since \(H<S_0\), Lemma~\ref{lem:crossing-cost} gives
\[
L_k(\sigma_\ell)\ge\kappa_0,
\qquad \ell\ge1.
\]
Hence, for every \(N\ge1\), we have
\[
L_k(\sigma)
\ge
\sum_{\ell=1}^N L_k(\sigma_\ell)
\ge
N\kappa_0.
\]
Therefore \(L_k(\sigma)=+\infty\), and the asserted estimate follows
immediately.
We may thus assume that only finitely many first exits occur.  Let
\(n_j^\ast\) denote their number.  Then
\[
a=s_0<s_1<\cdots<s_{n_j^\ast}<s_{n_j^\ast+1}=b,
\]
where, for \(1\le\ell\le n_j^\ast\),
\(
\sigma_\ell
=
\sigma|_{[s_{\ell-1},s_\ell]}
\)
ends at the first exit from
\[
\frac12
B_{\delta(\sigma(s_{\ell-1}))/4}
\bigl(\sigma(s_{\ell-1})\bigr),
\]
while
\(
\sigma_{n_j^\ast+1}
=
\sigma|_{[s_{n_j^\ast},b]}
\)
is the final part of the subcurve.  Thus the construction produces
\(n_j^\ast+1\) links.

For \(0\le\ell\le n_j^\ast\), set
\[
u_\ell=\sigma(s_\ell),
\qquad
v_\ell=\sigma(s_{\ell+1}).
\]
By construction, we have
\(v_\ell\in B_{\delta(u_\ell)/4}(u_\ell)\).
Since \(u_\ell\) lies on \(\gamma_j^\ast\), along which the boundary
distance is at most \(h_j^\ast\), we have
\[
d_B(u_\ell,v_\ell)
\le
\frac{\delta(u_\ell)}4
\le
\frac{h_j^\ast}{4}.
\]
Using the quasi-symmetry of \(d_B\), we obtain
\[
d_B(v_\ell,u_\ell)
\le
Q_m d_B(u_\ell,v_\ell)
\le
\frac{Q_mh_j^\ast}{4}.
\]
Therefore
\begin{equation}\label{eq:link-at-level}
	\widehat d_B(u_\ell,v_\ell)
	\le
	\frac{Q_mh_j^\ast}{4}.
\end{equation}
For \(1\le\ell\le n_j^\ast\),
Lemma~\ref{lem:crossing-cost} also gives
\(
L_k(\sigma_\ell)\ge\kappa_0.
\)
Hence, we have
\[
L_k(\gamma_j^\ast)
\ge
\sum_{\ell=1}^{n_j^\ast}L_k(\sigma_\ell)
\ge
\kappa_0 n_j^\ast.
\]
Together with \eqref{eq:normal-cost-one-level}, this yields
\begin{equation}\label{eq:two-costs-one-level}
	L_k(\gamma_j^\ast)
	\ge
	\max\{v_j^\ast,\kappa_0 n_j^\ast\}.
\end{equation}

Put \(\alpha=1/(4m)\) and \(\beta=\frac12\log2\).
Assign to each link recorded on \(\gamma_j^\ast\) the weight
\[
\omega_j^\ast=\left(\frac{h_j^\ast}{H}\right)^\alpha\in(0,1].
\]
By \eqref{eq:link-at-level} and \(Q_m^\alpha<2\), a link of
weight \(\omega_j^\ast\) satisfies
\[
\widehat d_B(u,v)^\alpha\le2H^\alpha\omega_j^\ast.
\]
These weights are adapted to the power-chain estimate. Their sum over
the levels of either end portion satisfies
\[
\sum_{j=1}^{k_\ast}\omega_j^\ast
\le\sum_{\ell=0}^{\infty}2^{-\ell\alpha}
=\frac1{1-2^{-\alpha}}.
\]

To combine the two costs in \eqref{eq:two-costs-one-level}, use the
following inequality, valid for \(0\le\omega\le1\):
\begin{equation}\label{eq:layer-budget}
\begin{aligned}
\max\{v,\kappa_0n\}
&\ge(1-\omega)v+\omega\kappa_0n\\
&=v+\kappa_0(n+1)\omega-(v+\kappa_0)\omega.
\end{aligned}
\end{equation}
Applying this with \(v=v_j^\ast\), \(n=n_j^\ast\), and
\(\omega=\omega_j^\ast\), and using \(v_j^\ast\le\beta\), gives
\begin{equation}\label{eq:weighted-layer-cost}
L_k(\gamma_j^\ast)
\ge v_j^\ast+\kappa_0(n_j^\ast+1)\omega_j^\ast
-(\beta+\kappa_0)\omega_j^\ast.
\end{equation}
Thus the full normal term is retained, with a summable loss.

Set
\[
\mathsf W_{\mathrm{end}}
=\sum_{\ast\in\{x,y\}}\sum_{j=1}^{k_\ast}
(n_j^\ast+1)\omega_j^\ast.
\]
The normal terms telescope to
\(\log(H/\sqrt{\delta(x)\delta(y)})\). Summing
\eqref{eq:weighted-layer-cost} over the two end portions yields
\begin{equation}\label{eq:arm-weight}
\begin{aligned}
&L_k(\gamma|_{[0,t_x]})+L_k(\gamma|_{[t_y,1]})\\
&\qquad\ge
\log\frac{H}{\sqrt{\delta(x)\delta(y)}}
+\kappa_0\mathsf W_{\mathrm{end}}
-\frac{2(\beta+\kappa_0)}{1-2^{-\alpha}}.
\end{aligned}
\end{equation}

On the middle interval \([t_x,t_y]\), use the same first-exit
construction and assign weight \(1\) to every link. If \(t_x=t_y\),
put \(\mathsf W_{\mathrm{mid}}=0\). Otherwise, infinitely many
completed exits give infinite length, so we may assume there are
\(n_0\) completed exits and one final link, and put
\(\mathsf W_{\mathrm{mid}}=n_0+1\). In both cases,
\begin{equation}\label{eq:middle-weight}
L_k(\gamma|_{[t_x,t_y]})
\ge\kappa_0\mathsf W_{\mathrm{mid}}-\kappa_0.
\end{equation}
Every recorded middle link also satisfies
\begin{equation}\label{eq:mid}
d_B(u,v)\le H/4,\qquad
\widehat d_B(u,v)\le Q_mH/4.
\end{equation}
Set \(\mathsf W=\mathsf W_{\mathrm{end}}+\mathsf W_{\mathrm{mid}}\).
Combining \eqref{eq:arm-weight} and \eqref{eq:middle-weight}, we obtain
\begin{equation}\label{eq:length-versus-weight}
L_k(\gamma)\ge
\log\frac{H}{\sqrt{\delta(x)\delta(y)}}
+\kappa_0\mathsf W
-\frac{2(\beta+\kappa_0)}{1-2^{-\alpha}}-\kappa_0.
\end{equation}

List the recorded links in their order from \(x\) to \(y\), reversing
the order and orientation of those constructed on the final end
portion. Symmetry of \(\widehat d_B\) preserves their size estimates.
Every recorded link, including the middle links, has a weight
\(\omega\in(0,1]\) satisfying
\begin{equation}\label{eq:weighted-link-size}
\widehat d_B(u,v)^\alpha\le2H^\alpha\omega.
\end{equation}
At least one link is recorded: either the middle interval is
nontrivial, or at least one end portion is nontrivial.
Lemma~\ref{lem:weight-grouping}, with \(J=4m\), partitions the links
into at most \(4m\) consecutive nonempty groups, each with total
weight at most \(\mathsf W/(4m)+1\).

For a group with successive vertices \(a=z_0,\ldots,z_N=b\), the
symmetric estimate \eqref{eq:symmetric-power-chain} gives
\[
\begin{aligned}
\widehat d_B(a,b)^\alpha
&\le4\sum_{j=1}^N\widehat d_B(z_{j-1},z_j)^\alpha\\
&\le8H^\alpha\sum_{j=1}^N\omega_j
\le8H^\alpha\left(\frac{\mathsf W}{4m}+1\right).
\end{aligned}
\]
The group endpoints form an admissible chain for \(r_m\), hence
\[
r_m(x,y)^\alpha\le
8H^\alpha\left(\frac{\mathsf W}{4m}+1\right),
\qquad
\mathsf W\ge\frac m2\left(\frac{r_m(x,y)}H\right)^\alpha-4m.
\]
Substitution into \eqref{eq:length-versus-weight} yields
\[
L_k(\gamma)\ge
\log\frac{H}{\sqrt{\delta(x)\delta(y)}}
+\frac{\kappa_0m}{2}\left(\frac{r_m(x,y)}H\right)^\alpha
-\mathcal B_m,
\]
where
\[
\mathcal B_m
=\frac{2(\beta+\kappa_0)}{1-2^{-\alpha}}+\kappa_0+4\kappa_0m.
\]
The elementary inequality
\((1-e^{-t})^{-1}\le1+t^{-1}\), \(t>0\), gives
\[
\begin{aligned}
\mathcal B_m
&\le\left(4+\frac{8\kappa_0}{\log2}+4\kappa_0\right)m
+\log2+3\kappa_0\\
&<\left(\frac92+\frac{43}{2}\kappa_0\right)m
<5m.
\end{aligned}
\]
Here we used \(m\ge2\), \(1/2<\log2<1\), and \(\kappa_0=1/176\).
This proves the lemma.
\end{proof}

We now choose the endpoint neighborhood \(V\) for the lower estimate.
By Lemma~\ref{lem:dB-euclidean-control}, after shrinking
\(V\Subset U\) so that \(8\operatorname{diam}(V)<\eps_*\), we have
\[
\widehat d_B(z,w)\le8|z-w|\le8\operatorname{diam}(V),
\qquad z,w\in\Omega\cap V.
\]
Since \(\sup_{z\in\Omega\cap V}\delta(z)\to0\) as \(V\) shrinks to \(p\),
we may shrink \(V\) once more so that
\[
8\operatorname{diam}(V)+\sup_{z\in\Omega\cap V}\delta(z)<S_0.
\]
Consequently,
\begin{equation}\label{eq:local-smallness-final}
\widehat d_B(z,w)+\max\{\delta(z),\delta(w)\}<S_0,
\qquad z,w\in\Omega\cap V.
\end{equation}
Since \(r_m(z,w)\le\widehat d_B(z,w)\), it follows that
\begin{equation}\label{eq:R-local-smallness}
R_m(z,w)=r_m(z,w)+\max\{\delta(z),\delta(w)\}<S_0,
\qquad z,w\in\Omega\cap V.
\end{equation}

\begin{remark}\label{rem:weighted-exit-idea}
The choice \(J=4m=\alpha^{-1}\), where \(\alpha=1/(4m)\), is the
source of the linear error. If the grouping step uses any
\(1\le J\le4m\), the same argument gives
\[
r_m(x,y)^\alpha
\le8H^\alpha\left(\frac{\mathsf W}{J}+1\right),
\qquad
\mathsf W\ge\frac J8\left(\frac{r_m(x,y)}H\right)^\alpha-J.
\]
Taking \(J=4m\) produces the factor \(m\) in the barrier term.
Indeed, with \(s=H/r_m(x,y)\) and \(u=s^{-1/(4m)}\), the terms
depending on \(s\) become
\[
\log s+\frac{\kappa_0m}{2}s^{-1/(4m)}
=m\left(-4\log u+\frac{\kappa_0}{2}u\right).
\]
Their infimum is \(m\) times a constant independent of \(m\).
If \(J\) were bounded independently of \(m\), the same mechanism
would instead minimize \(-4m\log u+Cu\) with \(C\) independent of
\(m\), giving an \(m\log m\) loss. Thus the reciprocal scaling
\(J\asymp\alpha^{-1}\asymp m\) is essential to this argument.
\end{remark}

We now use the preceding  estimate to obtain the lower
bound for the Kobayashi distance.

\begin{theorem}\label{thm:lower-bound}
For all \(x,y\in\Omega\cap V\),
\[
K_\Omega(x,y)\ge g_m(x,y)-31m.
\]
\end{theorem}
\begin{proof}
Put \(X=\Omega\cap U\), \(r=r_m(x,y)\), and
\(D=\max\{\delta(x),\delta(y)\}\). We first prove
\begin{equation}\label{eq:intrinsic-lower-bound}
K_X(x,y)\ge
g_m(x,y)-(1+4\log1408)m-\log2.
\end{equation}
Lengths below are measured using \(k_\Omega\). Since \(k_X\ge k_\Omega\),
a lower bound for every such length in \(X\) also bounds \(K_X\).

If \(r\le D\), Lemma~\ref{lem:normal-lower} gives
\[
K_X(x,y)\ge K_\Omega(x,y)
\ge\frac12\left|\log\frac{\delta(x)}{\delta(y)}\right|
\ge g_m(x,y)-\log2.
\]
Assume \(r>D\), and let \(\gamma\subset X\) be any absolutely
continuous curve from \(x\) to \(y\). Write
\(H=\max_t\delta(\gamma(t))\).
If \(H\ge r\), applying Lemma~\ref{lem:normal-lower} to the portions
before the first and after the last attainment of \(H\) gives
\[
L_k(\gamma)\ge
\log\frac{H}{\sqrt{\delta(x)\delta(y)}}
\ge\log\frac r{\sqrt{\delta(x)\delta(y)}}
>g_m(x,y)-\log2.
\]
Here we used \(r+D<2r\).

It remains to consider \(H<r\). By \eqref{eq:R-local-smallness},
\(H<r<S_0\), so Lemma~\ref{lem:weighted-exits} applies. Set
\(u=(r/H)^{1/(4m)}\ge1\). Since \(\kappa_0=1/176\),
\[
L_k(\gamma)\ge
\log\frac r{\sqrt{\delta(x)\delta(y)}}
+m\left(-4\log u+\frac u{352}-5\right).
\]
The function \(-4\log u+u/352\), \(u\ge1\), attains its minimum at
\(u=1408\), where its value is \(4-4\log1408\). Hence
\[
\begin{aligned}
L_k(\gamma)
&\ge\log\frac r{\sqrt{\delta(x)\delta(y)}}
-(1+4\log1408)m\\
&>g_m(x,y)-(1+4\log1408)m-\log2.
\end{aligned}
\]
Taking the infimum over curves proves \eqref{eq:intrinsic-lower-bound}.
Finally, \eqref{eq:localization} gives
\[
K_\Omega(x,y)\ge
g_m(x,y)-(1+4\log1408)m-\log2-1.
\]
Since
\[
1+4\log1408+\frac{\log2+1}{2}<31
\]
and \(m\ge2\), the desired estimate follows.
\end{proof}

\subsection{Upper bound for the Kobayashi distance}
We now prove an upper bound for the Kobayashi distance by constructing
an explicit path between two nearby points.  The idea is to move the
two endpoints slightly into the interior along a fixed direction and
then join the lifted points by a bounded number of segments.

After an affine unitary change of coordinates, we may assume that the
inward unit normal at \(p\) is the positive \(u\)-direction.  Denote
this direction by \(\nu_0\).  We write
\[
z=(\zeta,u),
\qquad
\zeta\in\R^{2n-1},
\quad
u\in\R.
\]
The set \(X=\Omega\cap U\), and hence the functions
\(d_B\), \(\widehat d_B\), \(r_m\), \(R_m\), and \(g_m\), retain their
definitions based on the neighborhood \(U\) fixed above; neither \(U\)
nor \(X\) will be altered below.  Choose a
bounded Euclidean neighborhood \(U_1\) of \(p\), with
\(\overline{U_1}\subset U\), so that
\[
\Omega\cap U_1
=
\{(\zeta,u)\in U_1:u>\varphi(\zeta)\}.
\]
Set \(h(\zeta,u)=u-\varphi(\zeta)\).
Thus \(h(z)\) is the vertical height of \(z\) above the boundary graph.

Let \(\nu_{\mathrm{in}}(q)\) denote the inward unit normal to
\(\partial\Omega\) at \(q\).  Take the preceding \(U_1\) initially
sufficiently small that
\begin{equation}\label{eq:graph-comparison}
\|D\varphi\|_\infty\le\frac1{10},
\qquad
|\nu_{\mathrm{in}}(q)-\nu_0|\le\frac1{10},
\qquad
\frac12 |h(z)|
\le
|\rho(z)|
\le
2|h(z)|,
\end{equation}
for \(z\in U_1\) and \(q\in\partial\Omega\cap U_1\).
The first two estimates show that the fixed direction \(\nu_0\) is
uniformly close to the inward normal direction, while the last one
compares the vertical height with the signed distance to the boundary.

We require in addition that
\(\overline{\partial\Omega\cap U_1}\) is contained in a relatively
compact smooth boundary patch on which the nearest-point projection is
defined.
Compactness of this patch and its \(C^2\)-regularity yield a number \(R>0\)
such that, for every \(q\in\partial\Omega\cap U_1\), the following
inclusion holds:
\begin{equation*}
	B(q+R\nu_{\mathrm{in}}(q),R)
	\subset\Omega.
\end{equation*}
Choose Euclidean neighborhoods \(V_0\) and \(W\) of \(p\) such that
\[
p\in V_0\Subset W\Subset U_1
\qquad\text{and}\qquad
\pi(\overline W)\subset\partial\Omega\cap U_1,
\]
and set
\[
d_0
=
\dist(\overline{V_0},\C^n\setminus W),
\qquad
d_1
=
\dist(\overline W,\C^n\setminus U_1).
\]
Then \(d_0,d_1>0\).

Recall that \(2S_0<\eps_*\).  Since
\(M(t)\to0\) as \(t\to0^+\), after decreasing \(S_0\) we may
assume that
\begin{equation}\label{eq:S0-upper-conditions}
	3S_0\le R,
	\qquad
	M(S_0)+2S_0<d_1,
	\qquad
	\frac m2M(S_0)<d_0.
\end{equation}
Finally, shrink the endpoint neighborhood \(V\), without changing
the notation, so that \(V\Subset V_0\) and
\[
8\operatorname{diam}(V)
+\sup_{z\in\Omega\cap V}\delta(z)<S_0,
\]
where \(S_0\) is the final value satisfying
\eqref{eq:S0-upper-conditions}.
Lemma~\ref{lem:dB-euclidean-control} and the one-link bound for \(r_m\)
then give \eqref{eq:local-smallness-final} and
\eqref{eq:R-local-smallness} with this final \(S_0\).
Only \(V\) and \(S_0\) have been reduced; the fixed set \(X\) and the
functions \(r_m\) and \(g_m\) are unchanged.

We begin with several elementary geometric estimates.

\begin{lemma}\label{lem:translated-disc}
	Let \(a\in\Omega\cap W\) and \(0<t<S_0\), and set
	\(A=a+2t\nu_0\).
	For every nonzero vector \(\xi\in\C^n\), the directional distances
	satisfy
	\begin{equation*}
		\delta_\Omega(A;\xi)
		\ge
		\tau(a,\xi,t)|\xi|
		=
		\delta_{\Omega_{a,t}}(a;\xi).
	\end{equation*}
\end{lemma}

\begin{proof}
	Fix \(\lambda\in\C\) satisfying
	\(|\lambda|<\tau(a,\xi,t)\),
	and set \(z=a+\lambda\xi\).
	By \eqref{eq:tau-homogeneity} and the definition of \(M\), we have
	\[
	|\lambda\xi|
	<
	\tau(a,\xi,t)|\xi|
	=
	\tau\left(a,\frac{\xi}{|\xi|},t\right)
	\le
	M(t)
	\le
	M(S_0).
	\]
	Since \(a\in W\), \eqref{eq:S0-upper-conditions} gives
	\(
	|z-a|<d_1
	\)
	and
	\[
	|z+2t\nu_0-a|
	\le
	|z-a|+2t
	<
	M(S_0)+2S_0
	<
	d_1.
	\]
	Hence, both \(z\) and \(z+2t\nu_0\) belong to \(U_1\).
	
	By the definition of \(\tau(a,\xi,t)\), we have
	\(\rho(z)<\rho(a)+t<t\),
	since \(a\in\Omega\).  If \(z\in\Omega\), then \(h(z)>0\).
	Otherwise,
	\(
	0\le\rho(z)<t,
	\)
	and \eqref{eq:graph-comparison} gives
	\(h(z)\ge-|h(z)|>-2t\).
	Thus in either case
	\(
	h(z)>-2t.
	\)
	Since \(\nu_0\) is the positive \(u\)-direction, we have
	\(h(z+2t\nu_0) = h(z)+2t > 0\).
	As \(z+2t\nu_0\in U_1\), the graph representation of
	\(\Omega\cap U_1\) implies that
	\(
	z+2t\nu_0\in\Omega.
	\)
	Equivalently, we have
	\[
	A+\lambda\xi\in\Omega
	\qquad
	\text{whenever }
	|\lambda|<\tau(a,\xi,t).
	\]
	Therefore
	\[
	\delta_\Omega(A;\xi)
	\ge
	\tau(a,\xi,t)|\xi|.
	\]
	The equality
	\[
	\tau(a,\xi,t)|\xi|
	=
	\delta_{\Omega_{a,t}}(a;\xi)
	\]
	follows from the definition of \(\tau\).
\end{proof}

\begin{lemma}\label{lem:horizontal-segment}
Let \(X,Y\in\Omega\) be distinct, and let \(c>0\). If
\[
\min\{\delta_\Omega(X;Y-X),\delta_\Omega(Y;Y-X)\}
\ge c|X-Y|,
\]
then \(K_\Omega(X,Y)\le c^{-1}\).
\end{lemma}
\begin{proof}
Set \(\xi=Y-X\) and \(L=X+\C\xi\). The planar domain
\(\Omega\cap L\) is convex, so its boundary distance is concave.
Consequently, along \(\gamma(t)=X+t\xi\), \(0\le t\le1\),
\[
\delta_\Omega(\gamma(t);\xi)
=\delta_{\Omega\cap L}(\gamma(t))
\ge\min\{\delta_\Omega(X;\xi),\delta_\Omega(Y;\xi)\}
\ge c|\xi|.
\]
Lemma~\ref{lem:graham} therefore gives
\[
K_\Omega(X,Y)\le
\int_0^1 k_\Omega(\gamma(t);\xi)\,\dd t
\le\int_0^1\frac{|\xi|}{\delta_\Omega(\gamma(t);\xi)}\,\dd t
\le c^{-1}.
\]
\end{proof}

\begin{lemma}\label{lem:chain-lifting}
	Let \(x,y\in\Omega\cap V\) and set
	\(D=\max\{\delta(x),\delta(y)\}\).
	For every \(\eta\) satisfying \(r_m(x,y)<\eta<S_0-D\), there exists
	a chain \(x=x_0,\ldots,x_N=y\) in \(\Omega\cap W\), with
	\(1\le N\le4m\), whose lifted points
	\(A_j=x_j+2S\nu_0\), where \(S=\eta+D\), satisfy
	\[
	K_\Omega(A_{j-1},A_j)\le\frac18,
	\qquad 1\le j\le N.
	\]
\end{lemma}

\begin{proof}
	The interval \((r_m(x,y),S_0-D)\) is nonempty by
	\eqref{eq:R-local-smallness}.  Fix \(\eta\) in this interval and
	put \(S=\eta+D<S_0\).
	By the definition of \(r_m\), there exists a chain
	\[
	x=x_0,\ldots,x_N=y,
	\qquad
	N\le4m,
	\]
	with \(x_j\in X\) and
	\[
	\widehat d_B(x_{j-1},x_j)<\eta,
	\qquad
	1\le j\le N.
	\]
	
	We first show that every vertex of the chain belongs to \(W\).
	Since
	\(d_B(x_{j-1},x_j) < \eta < S_0\),
	Lemma~\ref{lem:dB-euclidean-control} gives
	\[
	|x_j-x_{j-1}|
	<
	\frac18M(\eta)
	\le
	\frac18M(S_0).
	\]
	For \(0\le j\le N\), the triangle inequality gives
	\[
	|x_j-x|
	\le
	\sum_{k=1}^j|x_k-x_{k-1}|
	\le
	\frac{j}{8}M(S_0)
	\le
	\frac m2M(S_0)
	<
	d_0.
	\]
	Since \(x\in V\) and \(V\Subset V_0\), the definition of \(d_0\)
	implies that
	\[
	x_j\in W,
	\qquad
	0\le j\le N.
	\]
	
	Fix \(1\le j\le N\) and set
	\(
	\xi_j=x_j-x_{j-1}.
	\)
	If \(\xi_j=0\), then \(A_{j-1}=A_j\), and there is nothing to prove.
	Assume \(\xi_j\ne0\).
	Since
	\[
	d_B(x_{j-1},x_j)
	\le
	\widehat d_B(x_{j-1},x_j)
	<
	\eta
	<
	S,
	\]
	we have
	\(
	x_j\in B_S(x_{j-1}),
	\)
	and hence
	\[
	p_{x_{j-1},S}(\xi_j)<\frac18.
	\]
	By Lemma~\ref{lem:directional-norm}, we have
	\(\tau(x_{j-1},\xi_j,S)>8\).
	Similarly, the inequality
	\(d_B(x_j,x_{j-1})<S\)
	gives
	\[
	\tau(x_j,\xi_j,S)
	=
	\tau(x_j,-\xi_j,S)
	>
	8.
	\]
	
	Since
	\[
	x_{j-1},x_j\in\Omega\cap W
	\qquad\text{and}\qquad
	S<S_0,
	\]
	Lemma~\ref{lem:translated-disc} gives
	\[
	\delta_\Omega(A_{j-1};\xi_j)
	>
	8|\xi_j|,
	\qquad
	\delta_\Omega(A_j;\xi_j)
	>
	8|\xi_j|.
	\]
	Since
	\(A_j-A_{j-1}=\xi_j\),
	Lemma~\ref{lem:horizontal-segment}, with \(c=8\), yields
	\[
	K_\Omega(A_{j-1},A_j)\le\frac18.
	\]
	This proves the lemma.
\end{proof}

\begin{lemma}\label{lem:normal-leg}
Let \(0<T<S_0\), and let \(a\in\Omega\cap W\) satisfy
\(\delta(a)\le T\).  If \(A=a+2T\nu_0\), then
\begin{equation*}
K_\Omega(a,A)
\le
\frac12\log\left(1+\frac{2T}{\delta(a)}\right)+1.
\end{equation*}
\end{lemma}

\begin{proof}
Let \(q=\pi(a)\) and set
\(A_\perp=a+2T\nu_{\mathrm{in}}(q)\).
Since \(\delta(a)+2T\le3T<3S_0\le R\), the points \(a\) and
\(A_\perp\) lie on the same radius of the interior tangent ball
\[
B=B(q+R\nu_{\mathrm{in}}(q),R)\subset\Omega.
\]
By Lemma~\ref{lem:ball-radial-distance}, we obtain
\[
\begin{aligned}
K_\Omega(a,A_\perp)
&\le K_B(a,A_\perp)\\
&=
\frac12\log\left(1+\frac{2T}{\delta(a)}\right)
+
\frac12\log
\frac{2R-\delta(a)}
     {2R-\delta(a)-2T}.
\end{aligned}
\]
Since \(\delta(a)\le T\) and \(3T<R\), the last quotient is less than
\(2\).  Hence
\[
K_\Omega(a,A_\perp)
\le
\frac12\log\left(1+\frac{2T}{\delta(a)}\right)
+\frac12\log2.
\]

By \eqref{eq:graph-comparison}, we have
\[
|A-A_\perp|
=
2T|\nu_0-\nu_{\mathrm{in}}(q)|
\le \frac{T}{5}.
\]
Moreover, \(B\subset\Omega\) gives
\(\delta(A_\perp)\ge\delta(a)+2T\).
Since
\[
|A-A_\perp|
\le\frac T5
<\delta(A_\perp),
\]
the segment \([A_\perp,A]\) is contained in \(\Omega\); in particular,
\(A\in\Omega\).  Since the boundary distance is \(1\)-Lipschitz, we
also have
\[
\delta(A)\ge\delta(A_\perp)-|A-A_\perp|
\ge\delta(a)+\frac95T.
\]
By concavity of the boundary distance on \(\Omega\), every point of the
segment \([A_\perp,A]\) has boundary distance at least
\(\delta(a)+\frac95T\).  Thus Lemma~\ref{lem:graham} gives
\[
K_\Omega(A_\perp,A)
\le
\frac{|A-A_\perp|}
     {\delta(a)+\frac95T}
\le\frac19.
\]
The triangle inequality now gives
\[
K_\Omega(a,A)
<
\frac12\log\left(1+\frac{2T}{\delta(a)}\right)+1.
\]
\end{proof}

\begin{theorem}\label{thm:upper-bound}
After shrinking \(V\) if necessary, for all \(x,y\in\Omega\cap V\),
\[
K_\Omega(x,y)\le g_m(x,y)+\frac m2+2+\log3.
\]
\end{theorem}
\begin{proof}
The assertion is immediate for \(x=y\). Otherwise, put
\(D=\max\{\delta(x),\delta(y)\}\), choose
\(r_m(x,y)<\eta<S_0-D\), and set \(S=\eta+D\).
Lemma~\ref{lem:chain-lifting} gives a chain with at most \(4m\) links
whose lifted points \(A_j=x_j+2S\nu_0\) satisfy
\[
\sum_{j=1}^N K_\Omega(A_{j-1},A_j)\le\frac{N}{8}\le\frac m2.
\]
Applying Lemma~\ref{lem:normal-leg} to the two endpoints gives
\[
\begin{aligned}
K_\Omega(x,y)
&\le\frac12\log\left(1+\frac{2S}{\delta(x)}\right)
+\frac12\log\left(1+\frac{2S}{\delta(y)}\right)
+\frac m2+2\\
&\le\log\frac S{\sqrt{\delta(x)\delta(y)}}+\frac m2+2+\log3.
\end{aligned}
\]
The last inequality uses \(S\ge\max\{\delta(x),\delta(y)\}\).
Letting \(\eta\to r_m(x,y)^+\), so that \(S\to R_m(x,y)^+\),
proves the result.
\end{proof}
\begin{proof}[Proof of Theorem~\ref{thm:two-sided-kobayashi}]
After the final shrinkage of \(V\), Theorems~\ref{thm:lower-bound}
and~\ref{thm:upper-bound} hold simultaneously. Combining their
estimates proves the assertion.
\end{proof}

\subsection{Local Gromov hyperbolicity}
\label{subsec:gromov-hyperbolicity}

We now use the distance estimates to prove Theorem~\ref{thm:main}.
Recall that \(X=\Omega\cap U\) is the neighborhood on which
\(r_m\), \(R_m\), and \(g_m\) are defined.

\begin{lemma}\label{lem:multiplicative-four-point}
Let \(x_1,x_2,x_3,x_4\in X\), and write
\(R_{ij}=R_m(x_i,x_j)\).
Then
\begin{equation*}
R_{12}R_{34}
\le
9Q_m^2
\max\{R_{13}R_{24},R_{14}R_{23}\}.
\end{equation*}
\end{lemma}

\begin{proof}
Among
\[
R_{13},\quad R_{14},\quad R_{23},\quad R_{24},
\]
choose the smallest one.  By interchanging \(x_1\) with \(x_2\),
\(x_3\) with \(x_4\), or both, we may assume that it is \(R_{13}\).
Then the last estimate in Lemma~\ref{lem:chain-estimates} gives
\[
R_{12}
\le
3Q_m\max\{R_{13},R_{23}\}
=
3Q_mR_{23},
\]
and similarly,
\[
R_{34}
\le
3Q_m\max\{R_{13},R_{14}\}
=
3Q_mR_{14}.
\]
Multiplying the two inequalities gives
\[
R_{12}R_{34}
\le
9Q_m^2R_{14}R_{23},
\]
and hence the asserted inequality.
\end{proof}

Set
\[
\mathcal E_m=\frac{63}{2}m+2+\log(9Q_m).
\]
Since \(Q_m<10^m\) and \(m\ge2\),
\[
\mathcal E_m
<\left(\frac{63}{2}+\log10\right)m+2+\log9
<36m.
\]

\begin{theorem}\label{thm:four-point}
There exists a neighborhood \(V\) of \(p\) such that, for all
\(x_1,x_2,x_3,x_4\in\Omega\cap V\), the following four-point
inequality holds:
\[
\begin{aligned}
K_\Omega(x_1,x_2)+K_\Omega(x_3,x_4)
\le
\max\bigl\{&
K_\Omega(x_1,x_3)+K_\Omega(x_2,x_4),\\
&
K_\Omega(x_1,x_4)+K_\Omega(x_2,x_3)
\bigr\}
+2\mathcal E_m.
\end{aligned}
\]
\end{theorem}

\begin{proof}
By the definition of \(g_m\), we have
\[
g_m(x_i,x_j)
=
\log R_{ij}
-\frac12\log\delta(x_i)
-\frac12\log\delta(x_j).
\]
Thus each of the three pairings contains the same boundary-distance
term.  For example,
\[
g_m(x_1,x_2)+g_m(x_3,x_4)
=
\log(R_{12}R_{34})
-\frac12\log\prod_{i=1}^4\delta(x_i),
\]
and the analogous formulas hold for the other two pairings.
Taking logarithms in Lemma~\ref{lem:multiplicative-four-point} therefore
gives
\[
\begin{aligned}
g_m(x_1,x_2)+g_m(x_3,x_4)
\le
\max\bigl\{&
g_m(x_1,x_3)+g_m(x_2,x_4),\\
&
g_m(x_1,x_4)+g_m(x_2,x_3)
\bigr\}
+2\log(3Q_m).
\end{aligned}
\]

By Theorem~\ref{thm:upper-bound},
\[
K_\Omega(x_1,x_2)+K_\Omega(x_3,x_4)
\le g_m(x_1,x_2)+g_m(x_3,x_4)+m+4+2\log3.
\]
For either pairing on the right, Theorem~\ref{thm:lower-bound} gives
\[
g_m(x_i,x_j)+g_m(x_k,x_\ell)
\le K_\Omega(x_i,x_j)+K_\Omega(x_k,x_\ell)+62m.
\]
The total error is therefore
\[
m+4+2\log3+2\log(3Q_m)+62m
=2\left(\frac{63}{2}m+2+\log(9Q_m)\right)
=2\mathcal E_m.
\]
This proves the theorem.
\end{proof}

By Definition~\ref{def:local-hyperbolicity-constant},
Theorem~\ref{thm:four-point}, and the preceding estimate for
\(\mathcal E_m\), we obtain
\[
\delta_{\mathrm{loc}}(\Omega,p)
\le
\mathcal E_m
<
36m.
\]
This proves Theorem~\ref{thm:main}.

\begin{remark}
The constant \(\mathcal E_m\) depends only on \(m\).  The constants used
to choose the tubular neighborhood, \(S_0\), and the final neighborhood
\(V\) may depend on the domain, the boundary point, and the dimension.
They affect only the size of the neighborhood on which the estimates
hold and do not enter the bound for \(\mathcal E_m\).
\end{remark}

\begin{corollary}\label{cor:homogeneous-tangent-model}
Let \(\lambda_1,\ldots,\lambda_{n-1}>0\), and let \(P\) be a convex
real polynomial on \(\R^{2n-2}\simeq\C^{n-1}\) satisfying
\[
P(t^{\lambda_1}z_1,\ldots,t^{\lambda_{n-1}}z_{n-1})
=
tP(z_1,\ldots,z_{n-1}),
\qquad t>0.
\]
Consider the convex model domain
\[
\mathcal M_P
=
\bigl\{
(w,z)\in\C\times\C^{n-1}:
\Re w>P(z)
\bigr\}.
\]
If \(\partial\mathcal M_P\) has D'Angelo type at most \(m\ge2\)
near the origin, then
\[
\delta(\mathcal M_P,K_{\mathcal M_P})
=
\delta_{\mathrm{loc}}(\mathcal M_P,0)
\le
\mathcal E_m
<
36m.
\]
\end{corollary}

\begin{proof}
For \(t>0\), define
\[
A_t(w,z_1,\ldots,z_{n-1})
=
\bigl(
tw,t^{\lambda_1}z_1,\ldots,t^{\lambda_{n-1}}z_{n-1}
\bigr).
\]
Thus
\(A_t(\mathcal M_P)=\mathcal M_P\) and  \(A_t\) is an automorphism of
\(\mathcal M_P\).

Let \(V\) be any Euclidean neighborhood of the origin and let
\(x_1,\ldots,x_4\in\mathcal M_P\).  We have
\(A_t(x_j)\to0\) as \(t\to0^+\).  Hence, for sufficiently small
\(t\), the images satisfy
\[
A_t(x_j)\in\mathcal M_P\cap V,
\qquad 1\le j\le4.
\]
By biholomorphic invariance, we have
\[
K_{\mathcal M_P}(A_t(x_i),A_t(x_j))
=
K_{\mathcal M_P}(x_i,x_j).
\]
Thus any prescribed quadruple of points in \(\mathcal M_P\) can be
mapped into \(\mathcal M_P\cap V\) by an automorphism, with all pairwise
Kobayashi distances preserved.  It follows that
\[
\delta(\mathcal M_P,K_{\mathcal M_P})
\le
\delta\bigl(
\mathcal M_P\cap V,
\left.K_{\mathcal M_P}\right|_{\mathcal M_P\cap V}
\bigr).
\]
The reverse inequality follows immediately because
\(\mathcal M_P\cap V\subset\mathcal M_P\).  Hence
\[
\delta(\mathcal M_P,K_{\mathcal M_P})
=
\delta\bigl(
\mathcal M_P\cap V,
\left.K_{\mathcal M_P}\right|_{\mathcal M_P\cap V}
\bigr).
\]
Taking the infimum over \(V\) gives
\[
\delta(\mathcal M_P,K_{\mathcal M_P})
=
\delta_{\mathrm{loc}}(\mathcal M_P,0).
\]
The estimate
\[
\delta_{\mathrm{loc}}(\mathcal M_P,0)
\le
\mathcal E_m
<
36m
\]
then follows from Theorem~\ref{thm:four-point}.
\end{proof}

\section{Linear growth is optimal}\label{sec:lower-bound}

We now prove the lower bound in Theorem~\ref{thm:universal-growth}. Firstly we construct an example in complex dimension two. For
each $m\geq1$, we consider the following homogeneous convex model of type $2m$:
\begin{equation*}\label{eq:lb-final-model}
 \mathcal T_m
 =
 \left\{(Q,Z)\in\mathbb C^2:
 \operatorname{Re}Q>(\operatorname{Re}Z)^{2m}\right\}.
\end{equation*}

\begin{theorem}\label{thm:lb-final-model}
The domain $\mathcal T_m$ is a $\mathbb C$-proper convex domain. Its boundary
is smooth and real analytic, has D'Angelo type at most $2m$ at every point,
and has type exactly $2m$ at the origin. Moreover,
\begin{equation}\label{eq:lb-final-lower}
 m\log2\leq \delta_{\mathrm{loc}}(\mathcal T_m,0)
 =
 \delta(\mathcal T_m,K_{\mathcal T_m})
 \leq 72m .
\end{equation}
\end{theorem}

We first verify the geometric properties of the model. The function
\[
 \rho(Q,Z)=(\operatorname{Re}Z)^{2m}-\operatorname{Re}Q
\]
is convex and real analytic, and its derivative in the real $Q$-direction is
$-1$. Thus $\rho$ is a smooth defining function for $\mathcal T_m$. The
inequality
\[
 x^{2m}\geq2mx-(2m-1),
 \qquad x\in\mathbb R,
\]
shows that the injective affine map
\[
 (Q,Z)\longmapsto\bigl(Q,Q-2mZ+(2m-1)\bigr)
\]
maps $\mathcal T_m$ into $\mathbb H\times\mathbb H$, the Cartesian product
of two copies of the right half-plane
\[
 \mathbb H:=\{w\in\mathbb C:\operatorname{Re}w>0\}.
\]
Composing with the Cayley transform in each coordinate gives a holomorphic
embedding of $\mathcal T_m$ into the bidisc. In particular,
$\mathcal T_m$ is $\mathbb C$-proper and Kobayashi hyperbolic.

We next determine the boundary type. Let $(q,z)$ be a nonconstant
holomorphic curve germ at the origin, and set
\[
 \ell=\operatorname{ord}_0q,
 \qquad
 k=\operatorname{ord}_0z,
\]
with order $+\infty$ assigned to an identically zero component. If
$z\not\equiv0$, the first nonzero homogeneous term of
$(\operatorname{Re}z)^{2m}$ has positive circular mean and therefore cannot
cancel the harmonic leading term of $\operatorname{Re}q$. Consequently,
\[
 \operatorname{ord}_0\bigl(\rho\circ(q,z)\bigr)
 =\min\{\ell,2mk\}
 \leq2m\min\{\ell,k\}.
\]
If $z\equiv0$, the same estimate follows from the nonzero normal component.
Thus the type at the origin is at most $2m$, while the curve
$\zeta\mapsto(0,\zeta)$ shows that it is exactly $2m$.

Now fix an arbitrary boundary point $(Q^0,Z^0)$ and put
$a=\operatorname{Re}Z^0$. After translating this point to the origin, absorb
the affine tangential term into the normal coordinate by setting
\[
 \widetilde q=q-2ma^{2m-1}z.
\]
In the resulting coordinates, the defining function is
\[
 -\operatorname{Re}\widetilde q
 +(a+\operatorname{Re}z)^{2m}-a^{2m}
 -2ma^{2m-1}\operatorname{Re}z.
\]
If $a\neq0$, the tangential remainder has a positive quadratic leading term;
if $a=0$, it is exactly $(\operatorname{Re}z)^{2m}$. The same circular-mean
argument therefore shows that the boundary type is at most $2m$ at every
point.

For $s>0$, the homogeneous dilation
\begin{equation}\label{eq:lb-final-dilation}
 A_s(Q,Z)=(s^{2m}Q,sZ)
\end{equation}
is an automorphism of $\mathcal T_m$. Let $V$ be any Euclidean neighborhood
of the origin. Given a fixed quadruple in $\mathcal T_m$, applying $A_s$ for
$s>0$ sufficiently small moves all four points into $V$ without changing any
of their pairwise Kobayashi distances. It follows that
\[
 \delta(\mathcal T_m,K_{\mathcal T_m})
 \leq
 \delta\bigl(\mathcal T_m\cap V,
 K_{\mathcal T_m}|_{\mathcal T_m\cap V}\bigr).
\]
The reverse inequality follows from the inclusion
$\mathcal T_m\cap V\subset\mathcal T_m$. Taking the infimum over $V$ and Corollary \ref{cor:homogeneous-tangent-model} gives
\begin{equation}\label{123}
 \delta_{\mathrm{loc}}(\mathcal T_m,0)
 =\delta(\mathcal T_m,K_{\mathcal T_m})< 72m.
\end{equation}
It remains to establish the lower bound in
\eqref{eq:lb-final-lower}.

Consider the totally real slice
\[
 \Sigma_m=\{(h,iy):h>0,\ y\in\mathbb R\}.
\]
For each $y\in\mathbb R$, the inclusion
$Q\mapsto(Q,iy)$ from $\mathbb H$ into $\mathcal T_m$ has the holomorphic
left inverse $(Q,Z)\mapsto Q$ and is therefore an isometric embedding for the
Kobayashi distance. The maps
\[
 T_t(Q,Z)=(Q,Z+it),
 \qquad t\in\mathbb R,
\]
are also automorphisms of $\mathcal T_m$. Combining these observations with
the dilation \eqref{eq:lb-final-dilation}, we obtain, for $h>0$,
$\lambda\geq1$, and $y,v\in\mathbb R$,
\begin{equation*}
 \begin{aligned}
 K_{\mathcal T_m}\bigl((h,iy),(\lambda^{2m}h,iy)\bigr)
   &=m\log\lambda,\\
 k_{\mathcal T_m}\bigl((\lambda^{2m}h,iy);(0,iv)\bigr)
   &=\lambda^{-1}k_{\mathcal T_m}\bigl((h,iy);(0,iv)\bigr).
 \end{aligned}
\end{equation*}
The first identity follows from the holomorphic retract onto $\mathbb H$. For
the second, the automorphism
\[
 T_y\circ A_{\lambda^{-1}}\circ T_{-y}
\]
maps $(\lambda^{2m}h,iy)$ to $(h,iy)$ and sends $(0,iv)$ to
$(0,i\lambda^{-1}v)$. The identity follows from the biholomorphic invariance
and homogeneity of the infinitesimal Kobayashi metric.

The following lemma uses the reflection symmetry and convexity of the model
to show that distances with endpoints in $\Sigma_m$ can be computed without
leaving the slice.

\begin{lemma}\label{lem:lb-final-slice}
For $x,y\in\Sigma_m$, the distance $K_{\mathcal T_m}(x,y)$ equals the
intrinsic path distance obtained by integrating $k_{\mathcal T_m}$ along
$\Sigma_m$.
\end{lemma}

\begin{proof}
The domain $\mathcal T_m$ is invariant under the antiholomorphic involution
\[
 R(Q,Z)=(\overline Q,-\overline Z),
\]
whose fixed-point set is precisely $\Sigma_m$. Since every path in
$\Sigma_m$ is also a path in $\mathcal T_m$, the ambient distance is bounded
above by the intrinsic path distance on the slice.

For the reverse inequality, fix distinct $x,y\in\Sigma_m$. Since
$\mathcal T_m$ is a $\mathbb C$-proper convex domain, it is taut by
\cite[Theorem~1.1]{BracciSaracco}. Hence
\cite[Lemma~3.3]{BracciSaracco} provides a holomorphic extremal disc
$f:\mathbb D\to\mathcal T_m$ and $r\in(0,1)$ such that
\[
 f(0)=x,
 \qquad
 f(r)=y,
 \qquad
 K_{\mathbb D}(0,r)=K_{\mathcal T_m}(x,y).
\]
The reflected map
\(
 f^R(\zeta)=R\bigl(f(\overline\zeta)\bigr)
\)
is holomorphic. By convexity,
\(
 F=\tfrac12(f+f^R)
\)
also maps $\mathbb D$ into $\mathcal T_m$. Since $x$ and $y$ are fixed by
$R$, the map $F$ has the same endpoints as $f$. Moreover,
$F([0,r])\subset\Sigma_m$. Infinitesimal contraction gives
\[
 \operatorname{length}_{k_{\mathcal T_m}}\bigl(F|_{[0,r]}\bigr)
 \leq K_{\mathbb D}(0,r)
 =K_{\mathcal T_m}(x,y).
\]
Thus the intrinsic path distance on $\Sigma_m$ is bounded above by
$K_{\mathcal T_m}(x,y)$. Together with the opposite inequality, this proves
the claim.
\end{proof}

We now introduce coordinates adapted to the dilation and establish elementary
estimates for the induced infinitesimal metric. For $a,b\in\mathbb R$, set
\[
 N_m(a,b)=k_{\mathcal T_m}\bigl((1,0);(a,ib)\bigr).
\]
Let
\[
 \pi:\mathcal T_m\to\mathbb H,
 \qquad \pi(Q,Z)=Q,
\]
and
\[
 \iota:\mathbb H\to\mathcal T_m,
 \qquad \iota(Q)=(Q,0).
\]
Since $\pi\circ\iota$ is the identity, holomorphic contraction gives
\[
 N_m(a,b)\geq k_{\mathbb H}(1;a)=\frac{|a|}{2},
 \qquad
 N_m(a,0)=\frac{|a|}{2}.
\]
To obtain a lower bound in terms of \(|b|\), consider the complex affine map
\[
 L(Q,Z)=Q-2mZ+(2m-1).
\]
As observed above, $L(\mathcal T_m)\subset\mathbb H$. Since
\[
 L(1,0)=2m,
 \qquad
 dL_{(1,0)}(a,ib)=a-2mib,
\]
holomorphic contraction yields
\[
 N_m(a,b)
 \geq k_{\mathbb H}(2m;a-2mib)
 =\frac{|a-2mib|}{4m}
 \geq\frac{|b|}{2}.
\]

For an upper bound, let $S=|a|+|b|>0$ and
define
\[
 \phi(\zeta)
 =
 \left(1+\frac{a\zeta}{2S},\frac{ib\zeta}{2S}\right).
\]
For every $\zeta\in\mathbb D$, one has
\[
 \operatorname{Re}\phi_1(\zeta)>\frac12,
 \qquad
 \bigl|\operatorname{Re}\phi_2(\zeta)\bigr|<\frac12,
\]
and hence $\phi(\mathbb D)\subset\mathcal T_m$. Since
\[
 \phi(0)=(1,0),
 \qquad
 \phi'(0)=\left(\frac{a}{2S},\frac{ib}{2S}\right),
\]
infinitesimal contraction gives
\[
 k_{\mathcal T_m}\left((1,0);
 \left(\frac{a}{2S},\frac{ib}{2S}\right)\right)
 \leq k_{\mathbb D}(0;1)=1.
\]
By homogeneity of the infinitesimal Kobayashi metric,
\[
 N_m(a,b)\leq2S=2(|a|+|b|).
\]
The same inequality is immediate when $a=b=0$.

On \(\Sigma_m\), one has \(Q>0\) and \(Z\in i\mathbb R\). Since
\(A_s\) multiplies \(Q\) by \(s^{2m}\) and \(Z\) by \(s\), we write
\[
 Q=e^{2m\sigma},
 \qquad
 Z=imY,
\]
and define
\[
 \Psi_m(\sigma,Y)=(e^{2m\sigma},imY),
 \qquad (\sigma,Y)\in\mathbb R^2.
\]
Its differential is
\[
 d\Psi_m|_{(\sigma,Y)}(a,b)
 =(2me^{2m\sigma}a,imb).
\]
For \(s>0\), we have
\[
\begin{aligned}
 A_s\bigl(\Psi_m(\sigma,Y)\bigr)
 =
 \bigl(s^{2m}e^{2m\sigma},imsY\bigr)
 =
 \Psi_m(\sigma+\log s,sY).
\end{aligned}
\]
Thus
\[
 \Psi_m^{-1}\circ A_s\circ\Psi_m(\sigma,Y)
 =(\sigma+\log s,sY).
\]

To express the metric at \(\Psi_m(\sigma,Y)\) in terms of \(N_m\), we use
an automorphism that maps \(\Psi_m(\sigma,Y)\) to \((1,0)\). Recall that \(T_t(Q,Z)=(Q,Z+it)\) and set
\[
 B_{\sigma,Y}
 =
 T_{-m e^{-\sigma}Y}\circ A_{e^{-\sigma}}.
\]
Then,
\[
 B_{\sigma,Y}\bigl(\Psi_m(\sigma,Y)\bigr)
 =T_{-m e^{-\sigma}Y}  \bigl(1,im e^{-\sigma}Y\bigr) =(1,0).
\]
The differential of \(A_{e^{-\sigma}}\) is
\[
 dA_{e^{-\sigma}}(u,v)
 =
 (e^{-2m\sigma}u,e^{-\sigma}v),
\]
and the differential of \(T_{-m e^{-\sigma}Y}\) is the identity. Therefore,
\[
 dB_{\sigma,Y}|_{\Psi_m(\sigma,Y)}
 \bigl(d\Psi_m|_{(\sigma,Y)}(a,b)\bigr)
 =
 dA_{e^{-\sigma}}
 \bigl(2me^{2m\sigma}a,imb\bigr)
 =(2ma,im e^{-\sigma}b).
\]
By invariance and homogeneity of the infinitesimal Kobayashi metric,
\[
 k_{\mathcal T_m}\bigl(
 \Psi_m(\sigma,Y);
 d\Psi_m|_{(\sigma,Y)}(a,b)\bigr)
 =
 k_{\mathcal T_m}\bigl(
 (1,0);(2ma,im e^{-\sigma}b)\bigr)
=
 mN_m(2a,e^{-\sigma}b).
\]

Define
\[
 F_m(\sigma,Y;a,b)
 =
 N_m(2a,e^{-\sigma}b).
\]
For a piecewise smooth curve
\(\gamma:[0,1]\to\mathbb R^2\), define
\(
 L_{F_m}(\gamma)
 =
 \int_0^1
 F_m\bigl(\gamma(t);\gamma'(t)\bigr)\,dt,
\)
and set
\[
 \bar d_m(\xi,\eta)
 =
 \inf_\gamma L_{F_m}(\gamma),
\]
where the infimum is taken over all such curves joining \(\xi\) to
\(\eta\).

The preceding computation shows that
\(
 L_{k_{\mathcal T_m}}
 \bigl(\Psi_m\circ\gamma\bigr)
 =
 mL_{F_m}(\gamma).
\)
Since \(\Psi_m\) parametrizes \(\Sigma_m\), Lemma~\ref{lem:lb-final-slice}
therefore gives
\begin{equation}\label{eq:lb-final-factorization}
 K_{\mathcal T_m}\bigl(\Psi_m(\xi),\Psi_m(\eta)\bigr)
 =
 m\bar d_m(\xi,\eta),
 \qquad
 \xi,\eta\in\mathbb R^2.
\end{equation}
In this formula, the factor \(m\) is displayed explicitly, although
\(\bar d_m\) still depends on \(m\) through \(N_m\).
The estimates for \(N_m\) give
\begin{equation}\label{eq:lb-final-F-bounds}
 F_m(\sigma,Y;a,b)
 \geq
 \max\left\{|a|,\frac12e^{-\sigma}|b|\right\},
 \qquad
 F_m(\sigma,Y;a,0)=|a|.
\end{equation}

For \(A\geq0\), define
\[
 f_m(A)=\bar d_m\bigl((0,0),(0,A)\bigr).
\]
By \eqref{eq:lb-final-factorization},
\[
 K_{\mathcal T_m}\bigl((1,0),(1,imA)\bigr)
 =
 mf_m(A).
\]
We now determine the behavior of \(f_m(A)\) as \(A\to\infty\).

\begin{lemma}\label{lem:lb-final-log}
For each fixed integer \(m\geq1\), there exists \(b_m\in\mathbb R\) such
that
\begin{equation*}\label{eq:lb-final-log}
 f_m(A)=2\log A+b_m+o(1),
 \qquad A\to\infty.
\end{equation*}
Consequently, for every fixed \(a,b>0\),
\begin{equation*}\label{eq:lb-final-difference}
 f_m(aT)-f_m(bT)\longrightarrow2\log(a/b),
 \qquad T\to\infty.
\end{equation*}
\end{lemma}

\begin{proof}
For \(A>0\), set
\[
 h_m(A)=f_m(A)-2\log A.
\]
We first show that \(h_m\) is nonincreasing. For \(s\in\mathbb R\), define
\[
 \Phi_s(\sigma,Y)=(\sigma+s,e^sY).
\]
Then
\(
 d\Phi_s|_{(\sigma,Y)}(u,v)=(u,e^sv),
\)
and therefore
\[
 F_m\bigl(
 \Phi_s(\sigma,Y);
 d\Phi_s|_{(\sigma,Y)}(u,v)\bigr)
 =
 N_m\bigl(2u,e^{-(\sigma+s)}e^sv\bigr)
 =
 F_m(\sigma,Y;u,v).
\]
Thus \(\Phi_s\) preserves the \(F_m\)-length of every piecewise smooth
curve.

Let \(B\geq A>0\) and set
\(
 s=\log(B/A).
\)
If \(\gamma\) joins \((0,0)\) to \((0,A)\), then
\(\Phi_s\circ\gamma\) joins \((s,0)\) to \((s,B)\) and has the same
length as \(\gamma\). Consider also the curves
\[
 \alpha_0(r)=(r,0),
 \qquad
 \alpha_1(r)=(s-r,B),
 \qquad 0\leq r\leq s.
\]
By \eqref{eq:lb-final-F-bounds},
\[
 L_{F_m}(\alpha_0)=L_{F_m}(\alpha_1)=s.
\]
Concatenating \(\alpha_0\), \(\Phi_s\circ\gamma\), and \(\alpha_1\)
gives a curve from \((0,0)\) to \((0,B)\). Taking the infimum over
\(\gamma\) gives
\[
 f_m(B)
 \leq
 f_m(A)+2s
 =
 f_m(A)+2\log(B/A).
\]
Hence
\(
 h_m(B)\leq h_m(A),
\)
so \(h_m\) is nonincreasing.

We next prove that \(h_m\) is bounded below. Let
\[
 \gamma(t)=(\sigma(t),Y(t)),
 \qquad 0\leq t\leq1,
\]
be a piecewise smooth curve joining \((0,0)\) to \((0,A)\), and write
\(
 L=L_{F_m}(\gamma).
\)
For \(0\leq t\leq1\), set
\[
 \ell(t)
 =
 \int_0^t
 F_m\bigl(\gamma(r);\gamma'(r)\bigr)\,dr.
\]
Thus \(\ell(t)\) is the length of \(\gamma|_{[0,t]}\), while
\(L-\ell(t)\) is the length of \(\gamma|_{[t,1]}\).
Since \(F_m(\sigma,Y;a,b)\geq|a|\), we have
\[
 |\sigma(t)|
 \leq
 \int_0^t|\sigma'(r)|\,dr
 \leq\ell(t)\quad \text{and} \quad
 |\sigma(t)|
 \leq
 \int_t^1|\sigma'(r)|\,dr
 \leq L-\ell(t).
\]
Therefore
\(
 |\sigma(t)|
 \leq
 \min\{\ell(t),L-\ell(t)\}.
\)
Moreover, for almost every \(t\),
\[
 \ell'(t)
 =
 F_m\bigl(\gamma(t);\gamma'(t)\bigr)
 \geq
 \frac12e^{-\sigma(t)}|Y'(t)|,
\]
and hence
\[
 |Y'(t)|
 \leq
 2e^{\sigma(t)}\ell'(t)
 \leq
 2e^{\min\{\ell(t),L-\ell(t)\}}\ell'(t).
\]
Since \(Y(0)=0\) and \(Y(1)=A\), it follows that
\[
 A
 \leq
 \int_0^1|Y'(t)|\,dt
 \leq
 2\int_0^1
 e^{\min\{\ell(t),L-\ell(t)\}}\ell'(t)\,dt
 =
 2\int_0^L e^{\min\{u,L-u\}}\,du
 =
 4(e^{L/2}-1).
\]
Thus
\(
 L\geq2\log(1+A/4).
\)
Taking the infimum over all curves joining \((0,0)\) to \((0,A)\) gives
\[
 f_m(A)\geq2\log(1+A/4).
\]
Consequently, for \(A\geq1\),
\[
 h_m(A)
 \geq
 2\log(1+A/4)-2\log A=2\log\left(\frac1A+\frac14\right)
 \geq-2\log4.
\]
Hence \(h_m\) is nonincreasing and bounded below, so it converges to
some \(b_m\in\mathbb R\). 
Finally,
\[
 f_m(aT)-f_m(bT)
 =
 2\log(aT)-2\log(bT)+o(1)
 =
 2\log(a/b)+o(1),
\]
which completes the proof.
\end{proof}

We now prove the lower bound in \eqref{eq:lb-final-lower}. 
By \eqref{eq:global-fourpoint-constant}, we shall construct sequences
\(x_n^1,x_n^2,x_n^3,x_n^4\in\mathcal T_m\) for which the difference
between the largest and the second largest pairing sums tends to
\(2m\log2\).
Choose
\[
 u_1=0,\qquad u_2=1,\qquad u_3=\sqrt2,\qquad u_4=1+\sqrt2,
\]
and define
\[
 x_n^j=(1,i m n u_j),\qquad n\geq1,\quad 1\leq j\leq4.
\]
Using imaginary translations, the automorphism \(Z\mapsto-Z\), and the
definition of \(f_m\), we obtain
\[
 K_{\mathcal T_m}(x_n^i,x_n^j)
 =mf_m\bigl(n|u_i-u_j|\bigr).
\]
For \(i\neq j\), Lemma~\ref{lem:lb-final-log} gives
\[
 K_{\mathcal T_m}(x_n^i,x_n^j)
 =2m\log n+2m\log|u_i-u_j|+mb_m+o(1)
\]
as \(n\to\infty\).
For a pairing \((ij)(k\ell)\), set
\[
 S_{(ij)(k\ell)}(n)
 =K_{\mathcal T_m}(x_n^i,x_n^j)
 +K_{\mathcal T_m}(x_n^k,x_n^\ell),
\]
and let \(C_n=4m\log n+2mb_m\). Then
\[
 S_{(ij)(k\ell)}(n)
 =C_n+2m\log\bigl(|u_i-u_j|\,|u_k-u_\ell|\bigr)+o(1).
\]
The products \(|u_i-u_j|\,|u_k-u_\ell|\) corresponding to the pairings
\((12)(34)\), \((13)(24)\), and \((14)(23)\) are \(1\), \(2\), and \(1\),
respectively. Hence
\[
 \begin{aligned}
 S_{(12)(34)}(n)&=C_n+o(1),\\
 S_{(13)(24)}(n)&=C_n+2m\log2+o(1),\\
 S_{(14)(23)}(n)&=C_n+o(1).
 \end{aligned}
\]
Let \(S_{\max}(n)\) and \(S_{\mathrm{mid}}(n)\) denote the largest and the
second largest of these three sums. Then
\[
 S_{\max}(n)=C_n+2m\log2+o(1),\qquad
 S_{\mathrm{mid}}(n)=C_n+o(1).
\]
Therefore
\begin{equation*}
 \lim_{n\to\infty}\frac12
 \bigl(S_{\max}(n)-S_{\mathrm{mid}}(n)\bigr)
 =m\log2.
\end{equation*}
It follows from \eqref{eq:global-fourpoint-constant} and \eqref{123} that
\[
 \delta_{\mathrm{loc}}(\mathcal T_m,0)
 =\delta(\mathcal T_m,K_{\mathcal T_m})
 \geq m\log2.
\]
This proves \eqref{eq:lb-final-lower} and completes the proof of
Theorem~\ref{thm:lb-final-model}.

\begin{corollary}\label{cor:lb-final-higher-dimensional}
For integers $m,d\geq1$, let
\[
 \mathcal T_{m,d}
 =
 \left\{(Q,Z)\in\mathbb C\times\mathbb C^d:
 \operatorname{Re}Q>
 \sum_{j=1}^d(\operatorname{Re}Z_j)^{2m}\right\},
 \qquad \mathcal T_{m,1}=\mathcal T_m.
\]
Then $\mathcal T_{m,d}$ is a $\mathbb C$-proper convex domain with smooth
real-analytic boundary. Its boundary has D'Angelo type at most $2m$ at every
point and type exactly $2m$ at the origin. Moreover,
\[
 m\log2
 \leq\delta_{\mathrm{loc}}(\mathcal T_{m,d},0)
 =\delta(\mathcal T_{m,d},K_{\mathcal T_{m,d}})
 <72m.
\]
\end{corollary}
\begin{proof}The arguments used for \(\mathcal T_m\) extend directly to
\(\mathcal T_{m,d}\), with \((\Re Z)^{2m}\) replaced by
\(\sum_{j=1}^d(\Re Z_j)^{2m}\), and give all the stated geometric and
boundary-type properties.
Consider the holomorphic maps
\[
I:\mathcal T_m\longrightarrow\mathcal T_{m,d},\quad
 I(Q,W)=(Q,W,0,\ldots,0),\]
\[P:\mathcal T_{m,d}\longrightarrow\mathcal T_m,\quad
 P(Q,Z_1,\ldots,Z_d)=(Q,Z_1).
\]
Since $P\circ I=\operatorname{id}_{\mathcal T_m}$, holomorphic contraction
shows that $I$ is an isometric embedding for the Kobayashi distance.
Consequently,
\[
\delta(\mathcal T_{m,d},K_{\mathcal T_{m,d}})
\geq
\delta(\mathcal T_m,K_{\mathcal T_m})
\geq m\log 2.
\]
The maps $A_s(Q,Z)=(s^{2m}Q,sZ)$ are automorphisms of
$\mathcal T_{m,d}$ and move every fixed quadruple arbitrarily close to the
origin as $s\downarrow0$. Hence
\[
\delta_{\mathrm{loc}}(\mathcal T_{m,d},0)
=
\delta(\mathcal T_{m,d},K_{\mathcal T_{m,d}}).
\]
The upper bound follows from Theorem~\ref{thm:main}, applied with the type
bound $2m$. 
\end{proof}

\begin{proof}[Proof of Theorem~\ref{thm:universal-growth}]
We first show that the local hyperbolicity constant is unchanged after
intersecting the domain with a sufficiently small convex neighborhood of
the distinguished boundary point. Choose a neighborhood \(U\) of \(p\)
such that
\(
 D=\Omega\cap U
\)
is convex. By Theorem~\ref{thm:nikolov-thomas-localization}, there exist
a neighborhood \(V_*\Subset U\) of \(p\), a constant \(C>0\), and an
exponent \(s>0\) such that
\[
 0\leq K_D(z,w)-K_\Omega(z,w)\leq C|z-w|^s,
 \qquad z,w\in\Omega\cap V_*.
\]

Let \(V\Subset V_*\) be a neighborhood of \(p\), and set
\(
 \varepsilon_V=C\operatorname{diam}(V)^s.
\)
For points in \(\Omega\cap V\), every distance computed with \(K_D\)
differs from the corresponding distance computed with \(K_\Omega\) by at
most \(\varepsilon_V\). Hence each of the three pairing sums differs by at
most \(2\varepsilon_V\), and the corresponding values of
\[
 \frac12(S_{\max}-S_{\mathrm{mid}})
\]
differ by at most \(2\varepsilon_V\). Since
\(\varepsilon_V\to0\) as \(V\) shrinks to \(p\), it follows that
\[
 \delta_{\mathrm{loc}}(\Omega,p)
 =\delta_{\mathrm{loc}}(\Omega\cap U,p).
\]
We may therefore assume that \(\Omega\) is convex.

Let \(M\geq2\) be even. Theorem~\ref{thm:four-point} gives a constant
\(\mathcal E_M<36M\), independent of the domain and its dimension, such
that
\[
 \delta_{\mathrm{loc}}(\Omega,p)\leq\mathcal E_M
\]
for every domain occurring in the definition of \(\mathfrak H(M)\).
Taking the supremum over these domains gives
\[
 \mathfrak H(M)\leq\mathcal E_M<36M.
\]

For the lower bound, set \(m=M/2\). The domain \(\mathcal T_m\) has
boundary type \(M\) at the origin, and
Theorem~\ref{thm:lb-final-model} gives
\[
 \mathfrak H(M)
 \geq\delta_{\mathrm{loc}}(\mathcal T_m,0)
 \geq\frac{M}{2}\log2.
\]
For every fixed complex dimension \(n\geq2\), the same lower bound follows
from Corollary~\ref{cor:lb-final-higher-dimensional} by taking \(d=n-1\).
Combining the upper and lower bounds completes the proof.
\end{proof}

\section{Applications}
\label{sec:applications}

This section discusses two consequences:
conditions under which local Gromov hyperbolicity yields global
Gromov hyperbolicity, and a local bound on asymptotic upper curvature.

\subsection{From local to global Gromov hyperbolicity}
Theorem~\ref{thm:main} gives a quantitative estimate near each
finite-type boundary point.  We next discuss how the pointwise local Gromov hyperbolic
constant relates to global Gromov hyperbolicity.
The local-to-global problem for the Kobayashi distance was studied
systematically by Bracci, Gaussier, Nikolov, and Thomas
\cite{BracciGaussierNikolovThomas}.  Their notion of local Gromov
hyperbolicity is intrinsic: near a boundary point \(p\), one considers a
truncated domain \(\Omega\cap U\) equipped with its own Kobayashi distance
\(K_{\Omega\cap U}\).  They also introduce a local visibility condition and
prove that a bounded domain is both visible and Gromov hyperbolic
if and only if it is both locally visible and locally Gromov hyperbolic
\cite[Theorem~1.2]{BracciGaussierNikolovThomas}.

The local Gromov hyperbolic constant considered here is different.
Definition~\ref{def:local-hyperbolicity-constant} restricts only the
points to a small neighborhood of \(p\), while all distances are
measured by \(K_\Omega\).  Thus the definition of the local Gromov hyperbolic constant does not use
the intrinsic Kobayashi distance of a truncated domain.
The result of Bracci--Gaussier--Nikolov--Thomas applies in greater
generality and also addresses the geometry near the artificial
boundary \(\partial U\cap\Omega\).

The remaining issue is to relate different boundary regions.  By the
Hopf--Rinow theorem, a bounded complete Kobayashi hyperbolic domain is
geodesic for its Kobayashi distance.  Such a domain \(\Omega\) is called
\emph{visible} if, for every pair of distinct points
\(\xi,\eta\in\partial\Omega\), there exist disjoint Euclidean neighborhoods
\(U_\xi,U_\eta\) and a compact set \(L\Subset\Omega\) such that every
\(K_\Omega\)-geodesic segment joining a point of
\(\Omega\cap U_\xi\) to a point of \(\Omega\cap U_\eta\) meets \(L\).
Thus visibility provides the global information needed to connect the
separate local descriptions near different boundary points
\cite{BracciGaussierNikolovThomas}.

The following proposition may be viewed as an ambient-metric analogue
of the local-to-global implication in
\cite[Theorem~1.2]{BracciGaussierNikolovThomas}.
\begin{proposition}\label{prop:visible-local-to-global}
Let \(\Omega\Subset\C^n\) be complete Kobayashi hyperbolic and visible.
If \(\delta_{\mathrm{loc}}(\Omega,p)<\infty\) for every
\(p\in\partial\Omega\), then \(\delta(\Omega,K_\Omega)<\infty\).
\end{proposition}

\begin{proof}
Suppose, to the contrary, that \((\Omega,K_\Omega)\) is not Gromov
hyperbolic, and fix \(o\in\Omega\).
Gromov hyperbolicity can equivalently be tested at one fixed basepoint,
up to a universal change of the constant; see
\cite[Chapter~III.H]{BridsonHaefliger}.
Consequently, there exist \(x_j,y_j,z_j\in\Omega\) such that
\begin{equation}\label{eq:unbounded-product-defect}
\min\{(x_j\mid y_j)_o,(y_j\mid z_j)_o\}
-
(x_j\mid z_j)_o
\longrightarrow\infty.
\end{equation}
Since Gromov products are nonnegative, both
\((x_j\mid y_j)_o\) and \((y_j\mid z_j)_o\) tend to infinity.  Moreover, we have
\[
(u\mid v)_o
\le
\min\{K_\Omega(o,u),K_\Omega(o,v)\},
\]
so each of the sequences \(x_j,y_j,z_j\) leaves every compact subset of
\(\Omega\).  Since \(\overline\Omega\) is compact, after passing to a
subsequence we may assume that
\[
x_j\to\xi,\qquad
y_j\to\eta,\qquad
z_j\to\zeta
\qquad
\text{in }\partial\Omega.
\]

We claim that \(\xi=\eta=\zeta\).
Suppose, for example, that \(\xi\ne\eta\).  By visibility, there is a
compact set \(L\Subset\Omega\) such that every geodesic segment joining
\(x_j\) to \(y_j\) meets \(L\) for all sufficiently large \(j\).
Choose \(a_j\in[x_j,y_j]_\Omega\cap L\).
Since \(a_j\) lies on a geodesic from \(x_j\) to \(y_j\), we have
\[
(x_j\mid y_j)_o
\le
K_\Omega(o,a_j)
\le
\max_{a\in L}K_\Omega(o,a),
\]
contradicting \((x_j\mid y_j)_o\to\infty\).  Hence
\(\xi=\eta\).  Applying the same argument to \(y_j,z_j\) gives
\(\eta=\zeta\).

Write \(p=\xi=\eta=\zeta\).
Since \(\delta_{\mathrm{loc}}(\Omega,p)<\infty\), there is a
neighborhood \(V\) of \(p\) and some \(\delta_p<\infty\) such that
\[
\bigl(
\Omega\cap V,
\left.K_\Omega\right|_{\Omega\cap V}
\bigr)
\]
is \(\delta_p\)-hyperbolic.  Fix \(a\in\Omega\cap V\).  For all
sufficiently large \(j\), the four points
\(x_j,y_j,z_j,a\) lie in \(\Omega\cap V\), and hence
\[
(x_j\mid z_j)_a
\ge
\min\{(x_j\mid y_j)_a,(y_j\mid z_j)_a\}
-\delta_p.
\]
For any \(u,v\in\Omega\), the triangle inequality gives
\[
\bigl|(u\mid v)_o-(u\mid v)_a\bigr|
\le
K_\Omega(o,a).
\]
Therefore, we have
\[
\min\{(x_j\mid y_j)_o,(y_j\mid z_j)_o\}
-
(x_j\mid z_j)_o
\le
\delta_p+2K_\Omega(o,a),
\]
contradicting \eqref{eq:unbounded-product-defect}.
\end{proof}

\begin{remark}
Proposition~\ref{prop:visible-local-to-global} is qualitative.  Even if
the local constants are uniformly bounded over
\(\partial\Omega\), the proposition does not give a bound for
\(\delta(\Omega,K_\Omega)\) in terms of this common local bound alone.
Visibility ensures that geodesics joining different boundary regions
pass through a compact subset of \(\Omega\), but the Kobayashi size of
such a compact set may depend on the global geometry of the domain.
\end{remark}
In the bounded \(C^\infty\)-smooth convex setting, the visibility
argument can be replaced by affine scaling.  Building on Zimmer's work
\cite{ZimmerFiniteType}, one obtains, at a suitable infinite-type
boundary point, an affine scaling limit whose boundary contains a
nontrivial complex affine disk.  Such a limit is not Gromov hyperbolic
for the Kobayashi distance.

Recall that a convex domain is called \(\C\)-proper if it contains no complex
affine line.  We use the local Hausdorff topology on the space
of \(\C\)-proper convex domains.

\begin{proposition}
	\label{prop:local-constant-scaling-limit}
	Let \(\Omega\subset\C^n\) be a \(\C\)-proper convex domain and let
	\(p\in\partial\Omega\).  Suppose that
	\(A_j\in\operatorname{Aff}(\C^n)\) are complex affine automorphisms such
	that
	\[
	\Omega_j:=A_j(\Omega)\longrightarrow\Omega_\infty
	\]
	in the local Hausdorff topology, where \(\Omega_\infty\) is a
	\(\C\)-proper convex domain.  Assume moreover that the scaling is
	centered at \(p\), in the sense that for every compact set
	\(E\Subset\Omega_\infty\), the following convergence holds:
	\begin{equation*}
		\sup_{x\in E}|A_j^{-1}(x)-p|
		\longrightarrow0.
	\end{equation*}
	Then
	\[
	\delta(\Omega_\infty,K_{\Omega_\infty})
	\le
	\delta_{\mathrm{loc}}(\Omega,p).
	\]
\end{proposition}

\begin{proof}
	The conclusion is immediate if
	\(\delta_{\mathrm{loc}}(\Omega,p)=+\infty\).
	Assume that this constant is finite and fix
	\(\Delta>\delta_{\mathrm{loc}}(\Omega,p)\).
	By the definition of the local hyperbolicity constant, there exists a
	neighborhood \(V\) of \(p\) such that
	\[
	\bigl(
	\Omega\cap V,
	\left.K_\Omega\right|_{\Omega\cap V}
	\bigr)
	\]
	is \(\Delta\)-hyperbolic.
	
	Let \(x_1,x_2,x_3,x_4\in\Omega_\infty\).  By local Hausdorff
	convergence, these points belong to \(\Omega_j\) for all sufficiently
	large \(j\).  Since the scaling is centered at \(p\), their inverse
	images under \(A_j\) lie in \(\Omega\cap V\) for all sufficiently large
	\(j\).  By biholomorphic invariance and the four-point inequality, we have
	\[
	\begin{aligned}
		K_{\Omega_j}(x_1,x_3)+K_{\Omega_j}(x_2,x_4)
		\le
		\max\bigl\{&
		K_{\Omega_j}(x_1,x_2)+K_{\Omega_j}(x_3,x_4),\\
		&
		K_{\Omega_j}(x_1,x_4)+K_{\Omega_j}(x_2,x_3)
		\bigr\}
		+2\Delta.
	\end{aligned}
	\]
	
	By \cite[Theorem~4.1]{ZimmerFiniteType},
	\(K_{\Omega_j}\) converges locally uniformly to
	\(K_{\Omega_\infty}\).  Passing to the limit gives the same
	inequality in \(\Omega_\infty\).  Since the quadruple was arbitrary,
	\[
	\delta(\Omega_\infty,K_{\Omega_\infty})
	\le
	\Delta.
	\]
	Letting
	\(\Delta\to\delta_{\mathrm{loc}}(\Omega,p)^+\)
	proves the result.
\end{proof}

We recall two notions of extremality for a convex set
\(C\subset\C^n\).

\begin{definition}
	A point \(q\in C\) is a \emph{real extreme point} of \(C\) if every
	representation \(q=tx+(1-t)y\), with \(x,y\in C\) and \(0<t<1\),
	satisfies \(x=y=q\).
	
	A point \(q\in C\) is a \emph{complex extreme point} of \(C\) if the
	condition \(q+\zeta v\in C\) for every sufficiently small
	\(\zeta\in\C\) forces \(v=0\).
	Equivalently, \(q\) is not the center of any nontrivial complex affine
	disk contained in \(C\).
\end{definition}

Every real extreme point is complex extreme.  Indeed, if
\(v\in\C^n\setminus\{0\}\) and
\(q+\zeta v\in C\) for all sufficiently small \(\zeta\), then, for some
sufficiently small \(r>0\), one can write
\[
q=\frac12(q+rv)+\frac12(q-rv),
\]
contradicting the real extremality of \(q\).
\begin{lemma}
	\label{lem:extreme-infinite-type-point}
	Let \(\Omega\Subset\C^n\) be a convex domain with \(C^\infty\)-smooth
	boundary.  If \(\partial\Omega\) is not of finite line type, then there
	exists \(q\in\partial\Omega\) of infinite line type which is a complex
	extreme point of \(\overline\Omega\).
\end{lemma}

\begin{proof}
Let \(\rho\) be the signed distance function and consider the unit complex tangent bundle
\[
\mathcal S
=
\{(x,v):x\in\partial\Omega,\ v\in T_x^{\C}\partial\Omega,\ |v|=1\}.
\]
This bundle is compact.  For \(N\ge1\), let
\(\mathcal E_N\subset\mathcal S\) consist of the pairs \((x,v)\) satisfying
\[
\frac{\partial^{a+b}}
{\partial\zeta^a\partial\bar\zeta^b}
\rho(x+\zeta v)\bigg|_{\zeta=0}=0
\qquad
\text{whenever }1\le a+b\le N.
\]
Each \(\mathcal E_N\) is closed, and
\(\mathcal E_{N+1}\subset\mathcal E_N\).  Since the boundary has no
uniform finite line-type bound, every \(\mathcal E_N\) is nonempty. The Cantor intersection theorem therefore gives
\(\bigcap_{N\ge1}\mathcal E_N\ne\varnothing\).  Choose
\((p,v_0)\) in this intersection.  Then \(p\) has infinite line type.

If \(p\) is a complex extreme point of \(\overline\Omega\), there is
nothing to prove.  Otherwise, there exist
\(v\in\C^n\setminus\{0\}\) and \(\eps>0\)
	such that
	\[
	D:=\{p+\zeta v:|\zeta|<\eps\}\subset\overline\Omega.
	\]
	Let \(\ell\) be a real affine functional supporting
	\(\overline\Omega\) at \(p\), normalized so that
	\[
	\ell\le0\quad\text{on }\overline\Omega,
	\qquad
	\ell(p)=0.
	\]
	The map
	\(\zeta\longmapsto \ell(p+\zeta v)\)
	is real affine, nonpositive near \(0\), and vanishes at \(0\).  Hence it
	vanishes identically, and therefore
	\[
	D\subset
	F:=\overline\Omega\cap\{\ell=0\}
	\subset\partial\Omega.
	\]
	
	The set \(F\) is nonempty, compact, and convex.  Fix \(a\in\C^n\), and
	choose \(q\in F\) at which the function \(x\mapsto|x-a|^2\) attains its
	maximum on \(F\).  We claim that \(q\) is a real extreme point of \(F\).
	Indeed, if
	\[
	q=tx+(1-t)y,
	\qquad
	x,y\in F,\quad 0<t<1,
	\]
	with \(x\ne y\), then the strict convexity of \(|\,\cdot-a|^2\) gives
	\[
	|q-a|^2
	<
	t|x-a|^2+(1-t)|y-a|^2
	\le
	|q-a|^2,
	\]
	a contradiction.
	
	We next show that \(q\) is a real extreme point of
	\(\overline\Omega\).  Suppose that
	\[
	q=tx+(1-t)y,
	\qquad
	x,y\in\overline\Omega,\quad 0<t<1.
	\]
	Since
	\[
	0=\ell(q)=t\ell(x)+(1-t)\ell(y)
	\]
	and \(\ell(x),\ell(y)\le0\), we have
	\(\ell(x)=\ell(y)=0\).  Hence \(x,y\in F\), and the real extremality of
	\(q\) in \(F\) gives \(x=y=q\).  Thus \(q\) is a real extreme point of
	\(\overline\Omega\), and therefore also a complex extreme point of
	\(\overline\Omega\).
	
	It remains to show that \(q\) has infinite line type.  For \(0<t<1\),
	set
	\(p_t=(1-t)q+tp\).
	Since \(F\) is convex and contains both \(q\) and \(D\), it follows that
	\[
	\{p_t+\zeta v:|\zeta|<t\eps\}
	=
	(1-t)q+tD
	\subset F
	\subset\partial\Omega.
	\]
	
	Let \(\rho\) be a smooth defining function for \(\Omega\) in a
	neighborhood of \(q\).  Since \(p_t\to q\) as \(t\to0^+\), for all
	sufficiently small \(t>0\), the following identity holds:
	\[
	\rho(p_t+\zeta v)=0
	\qquad
	\text{for }|\zeta|<t\eps.
	\]
	Consequently, we have
	\[
	\frac{\partial^{a+b}}
	{\partial\zeta^a\partial\bar\zeta^b}
	\rho(p_t+\zeta v)\bigg|_{\zeta=0}
	=0
	\qquad
	\text{for all }a,b\ge0.
	\]
	Letting \(t\to0^+\) and using the smoothness of \(\rho\), we obtain
	\[
	\frac{\partial^{a+b}}
	{\partial\zeta^a\partial\bar\zeta^b}
	\rho(q+\zeta v)\bigg|_{\zeta=0}
	=0
	\qquad
	\text{for all }a,b\ge0.
	\]
	Thus \(\rho(q+\zeta v)\) vanishes to infinite order at \(0\), so \(q\)
	has infinite line type.
\end{proof}

The following lemma is essentially due to Zimmer \cite{ZimmerFiniteType,ZimmerLimitSet}.
\begin{lemma}\label{lem:pointed-scaling}
Let $\Omega\Subset\C^n$ be a convex domain with $C^\infty$-smooth
boundary, and let $q\in\partial\Omega$ have infinite line type. Then
there exist points $p_j\in\Omega$, complex affine automorphisms
$A_j\in\Aff(\C^n)$, a $\C$-proper convex domain $\Omega_\infty$, and a
point $u\in\Omega_\infty$ such that
\[
  p_j\longrightarrow q,\qquad A_j(p_j)=u,
  \qquad A_j(\Omega)\longrightarrow\Omega_\infty
\]
in the local Hausdorff topology. Moreover,
$\partial\Omega_\infty$ contains a nontrivial complex affine disk.
\end{lemma}

\begin{proof}
The pointed scaling construction in the proof of
\cite[Proposition~9.3, especially Lemmas~9.4 and 9.5]{ZimmerLimitSet}, applied
at $q$, gives $p_j\to q$, affine automorphisms
$B_j\in\Aff(\C^n)$, a point $u\in\Omega_\infty$, and a $\C$-proper
convex domain $\Omega_\infty$ such that
\[
  B_j(\Omega,p_j)\longrightarrow(\Omega_\infty,u)
\]
in the pointed local Hausdorff topology and
$\partial\Omega_\infty$ contains a nontrivial complex affine disk.
Set
\[
  T_j(z)=z+u-B_j(p_j),\qquad A_j=T_j\circ B_j.
\]
Since $B_j(p_j)\to u$, we have $T_j\to\Id_{\C^n}$ locally uniformly.
Thus $A_j(\Omega)\to\Omega_\infty$ and $A_j(p_j)=u$.
\end{proof}

\begin{theorem}\label{thm:local-to-global}
Let $\Omega\Subset\C^n$ be a convex domain with $C^\infty$-smooth
boundary. If $\delta_{\mathrm{loc}}(\Omega,p)<\infty$ for every
$p\in\partial\Omega$, then $(\Omega,K_\Omega)$ is Gromov hyperbolic.
\end{theorem}

\begin{proof}
We first prove that $\partial\Omega$ is of finite type. Suppose, to
the contrary, that it is not. Lemma~\ref{lem:extreme-infinite-type-point} gives a point
$q\in\partial\Omega$ of infinite line type which is a complex extreme
point of $\overline\Omega$. Apply Lemma~\ref{lem:pointed-scaling}.
There exist $p_j\to q$, complex affine automorphisms $A_j$, a point
$u\in\Omega_\infty$, and a $\C$-proper convex domain $\Omega_\infty$
such that
\[
  A_j(p_j)=u,\qquad \Omega_j:=A_j(\Omega)\longrightarrow\Omega_\infty
\]
in the local Hausdorff topology, and $\partial\Omega_\infty$ contains
a nontrivial complex affine disk.

We claim that the scaling is centered at $q$. Write
\[
  A_j^{-1}(z)=p_j+L_j(z-u),
\]
where $L_j$ is complex linear. Choose $r>0$ with
$\overline{B(u,r)}\Subset\Omega_\infty$. By local Hausdorff
convergence, $B(u,r)\subset\Omega_j$ for all sufficiently large $j$.
Its inverse image lies in the bounded domain $\Omega$, so the
operators $L_j$ are uniformly bounded.

If a subsequence converged to a nonzero operator $L$, then
\[
  q+L(z-u)=\lim_j A_j^{-1}(z)\in\overline\Omega,
  \qquad z\in B(u,r).
\]
Choosing $v$ with $Lv\ne0$ would produce a nontrivial complex affine
disk in $\overline\Omega$ centered at $q$, contrary to the complex
extremality of $q$. Hence $L_j\to0$. Therefore, for every compact
$E\Subset\Omega_\infty$,
\[
  \sup_{z\in E}|A_j^{-1}(z)-q|
  \le |p_j-q|+\|L_j\|\sup_{z\in E}|z-u|
  \longrightarrow0.
\]
Thus the scaling is centered at $q$. Proposition~\ref{prop:local-constant-scaling-limit} gives
\[
  \delta(\Omega_\infty,K_{\Omega_\infty})
  \le \delta_{\mathrm{loc}}(\Omega,q)<\infty.
\]
Hence $(\Omega_\infty,K_{\Omega_\infty})$ is Gromov hyperbolic. This
contradicts \cite[Theorem~3.1]{ZimmerFiniteType}, since
$\partial\Omega_\infty$ contains a nontrivial complex affine disk.
It follows that $\partial\Omega$ is of finite type. Therefore
$(\Omega,K_\Omega)$ is Gromov hyperbolic by
\cite[Theorem~1.1]{ZimmerFiniteType}.
\end{proof}

\subsection{Asymptotic upper curvature}

We record a coarse-curvature consequence of
Theorem~\ref{thm:main}.  Following Bonk and Foertsch
\cite{BonkFoertsch}, let \(\kappa\in[-\infty,0)\).  A metric space
\((X,d)\) is called an \(\mathrm{AC}_u(\kappa)\)-space if there exist
a basepoint \(o\in X\) and a constant \(c\ge0\) such that, for every
finite chain \(x_0=x,x_1,\ldots,x_N=x'\) in \(X\) with \(N\ge1\), one has
\begin{equation}\label{eq:acu-chain-condition}
	(x\mid x')_o
	\ge
	\min_{1\le j\le N}(x_{j-1}\mid x_j)_o
	-
	\frac{1}{\sqrt{-\kappa}}\log N
	-c.
\end{equation}
Here and below, \(1/\sqrt{\infty}=0\).  The asymptotic upper curvature
of \((X,d)\) is defined by
\[
K_u(X,d)
:=
\inf\left\{
\kappa\in[-\infty,0):
(X,d)\text{ is an }\mathrm{AC}_u(\kappa)\text{-space}
\right\},
\]
with the convention that the infimum of the empty set is \(+\infty\).
The definition of an \(\mathrm{AC}_u(\kappa)\)-space is independent of
the choice of basepoint, up to a change in \(c\), and \(K_u\) is
invariant under rough isometries.

\begin{corollary}[Local asymptotic upper curvature]
	\label{cor:local-asymptotic-curvature}
	Suppose that the hypotheses of Theorem~\ref{thm:main} hold, and set
	\(\kappa_m:=-(\log 2/(36m))^2\).
	Then there exists a neighborhood \(V\) of \(p\) such that
	\(\Omega\cap V\), equipped with the restriction of \(K_\Omega\),
	is an \(\mathrm{AC}_u(\kappa_m)\)-space.  Consequently,
	\begin{equation*}
		K_u\left(
		\Omega\cap V,\,
		\left.K_\Omega\right|_{\Omega\cap V}
		\right)
		\le
		-\left(\frac{\log 2}{36m}\right)^2.
	\end{equation*}
\end{corollary}

\begin{proof}
	By Theorem~\ref{thm:main} and the definition of the local
	hyperbolicity constant, there exists a neighborhood \(V\) of \(p\)
	such that the space \(X:=\Omega\cap V\), equipped with
	\(d:=\left.K_\Omega\right|_{X\times X}\), satisfies
	\(\delta(X,d)<36m\).
	Choose a number \(\Delta\) such that
	\(\delta(X,d)<\Delta<36m\).
	Then \((X,d)\) is \(\Delta\)-hyperbolic in the sense of
	\eqref{eq:gromov-product-inequality}.
	
	Fix \(o\in X\).  We claim that every finite chain
	\(x_0,\ldots,x_N\) in \(X\) satisfies
	\begin{equation*}
		(x_0\mid x_N)_o
		\ge
		\min_{1\le j\le N}(x_{j-1}\mid x_j)_o
		-
		\Delta\left\lceil\log_2 N\right\rceil.
	\end{equation*}
	We prove this by induction on \(q:=\lceil\log_2 N\rceil\).
	The case \(q=0\), corresponding to \(N=1\), is immediate.  Suppose
	that \(q\ge1\).  Choose \(r\in\{1,\ldots,N-1\}\) so that both
	subchains \(x_0,\ldots,x_r\) and \(x_r,\ldots,x_N\) have length at
	most \(2^{q-1}\).  By the Gromov-product inequality, we have
	\[
	(x_0\mid x_N)_o
	\ge
	\min\bigl\{
	(x_0\mid x_r)_o,\,
	(x_r\mid x_N)_o
	\bigr\}
	-\Delta.
	\]
	Applying the induction hypothesis to the two subchains gives
	\[
	(x_0\mid x_N)_o
	\ge
	\min_{1\le j\le N}(x_{j-1}\mid x_j)_o
	-q\Delta,
	\]
	which proves the claimed chain estimate.
	
	Since
	\[
	\left\lceil\log_2 N\right\rceil
	\le
	\frac{\log N}{\log 2}+1,
	\]
	we obtain
	\[
	(x_0\mid x_N)_o
	\ge
	\min_{1\le j\le N}(x_{j-1}\mid x_j)_o
	-
	\frac{\Delta}{\log 2}\log N
	-\Delta.
	\]
	Because \(\Delta<36m\), it follows that
	\[
	(x_0\mid x_N)_o
	\ge
	\min_{1\le j\le N}(x_{j-1}\mid x_j)_o
	-
	\frac{36m}{\log 2}\log N
	-36m.
	\]
	Since \(1/\sqrt{-\kappa_m}=36m/\log2\), the preceding inequality
	is precisely
	\eqref{eq:acu-chain-condition}, with \(\kappa=\kappa_m\) and
	\(c=36m\).  Thus \((X,d)\) is an
	\(\mathrm{AC}_u(\kappa_m)\)-space, and
	the asserted inequality follows from the definition of
	\(K_u\).
\end{proof}

\begin{remark}
Corollary~\ref{cor:local-asymptotic-curvature} converts the universal
upper estimate \(\delta_{\mathrm{loc}}(\Omega,p)<36m\) into the
coarse-curvature bound \(K_u\le-((\log2)/(36m))^2\) on a sufficiently
small neighborhood, equipped with the restricted ambient distance.
The order \(m^{-2}\) is not claimed to be optimal.  This is a coarse metric
statement and should not be confused with a pointwise bound on the sectional,
holomorphic sectional, or bisectional curvature tensor of a smooth
Riemannian or K\"ahler metric.
\end{remark}

\end{document}